\documentclass[a4paper,10pt]{article}

\usepackage{geometry}
    \usepackage{amsmath}
    \usepackage{amsfonts,color,amssymb,mathrsfs,stmaryrd}
    \usepackage{amsthm}
    \usepackage{amssymb}
		\usepackage{url}

\usepackage{bera}

    \usepackage[normalem]{ulem}

\usepackage{graphicx}
\usepackage{wrapfig}
\usepackage{mathrsfs}

\newcommand{\PP}{\mathbb{P}}
\newcommand{\XX}{\mathbb{X}}

\newcommand{\wx}{\widehat{x}}

\numberwithin{equation}{section}

    \newtheorem{theo}{Theorem}\numberwithin{theo}{section}
    \newtheorem{coro}[theo]{Corollary}
    
    \newtheorem{prop}[theo]{Proposition}
    \newtheorem{lemm}[theo]{Lemma}

        \def\N{\mathbb{N}}
    \def\E{\mathbb{E}}
    \def\0{{\bf 0}}

    \renewcommand{\E}{\mathbb E \,}

    \newcommand{\MM}{\mathbb{M}}

    \newcommand{\Cov}{{\rm Cov}}
    \newcommand{\Var}{{\rm Var}}

    \def\bdm{\begin{displaymath}}
    \newcommand{\edm}{\end{displaymath}}
    \def\benu{\begin{enumerate}}
    \def\eenu{\end{enumerate}}
    \def\beqn{\begin{equation}}
    \def\eeqn{\end{equation}}
    \def\be{\begin{equation}}
    \def\ee{\end{equation}}
    \def\bea{\begin{eqnarray}}
    \def\eea{\end{eqnarray}}
    \newcommand{\bean}{\begin{eqnarray*}}
    \newcommand{\eean}{\end{eqnarray*}}
    \newcommand{\bear}{\begin{eqnarray}}
    \newcommand{\eear}{\end{eqnarray}}

    \def\R{\mathbb{R}}

		\def\dint{\textup{d}}

    \def\la{{\lambda}}

    \def\qed{\hfill\hbox{${\vcenter{\vbox{
        \hrule height 0.4pt\hbox{\vrule width 0.4pt height 6pt
        \kern5pt\vrule width 0.4pt}\hrule height 0.4pt}}}$}}

\usepackage{titlesec}

\titleformat*{\section}{\normalfont\large\bfseries}
\titleformat*{\subsection}{\normalfont\bfseries}

\date{\vspace{-0.95cm}}

\begin{document}

\title{Multivariate second order Poincar\'e inequalities for Poisson functionals}

\author{Matthias Schulte\footnotemark[1] \ \ and \ J. E. Yukich\footnotemark[2]}

\maketitle

\footnotetext[1]{University of Bern, Switzerland,
    matthias.schulte@stat.unibe.ch, supported by the SNF grant 186049}
\footnotetext[2]{Lehigh University,
    United States of America, jey0@lehigh.edu, supported in part by  NSF grant  DMS-1406410, a Simons Collaboration Grant, and SNF grant 186049}


\begin{abstract}

Given a vector $F =(F_1,\hdots,F_m)$ of Poisson functionals $F_1,\hdots,F_m$, we investigate the proximity between $F$ and
an $m$-dimensional centered Gaussian random vector $N_\Sigma$ with covariance matrix $\Sigma\in\R^{m\times m}$. Apart from finding proximity bounds for the $d_2$- and $d_3$-distances, based on classes of smooth test functions, we obtain proximity bounds for the $d_{convex}$-distance, based on the less tractable test functions comprised of indicators of convex sets. The bounds for all three distances are shown to be of the same order, which is presumably optimal. The bounds are multivariate counterparts of the univariate second order Poincar\'e inequalities and, as such, are expressed in terms of integrated moments of first and second order difference operators. The derived second order Poincar\'e inequalities for indicators of convex sets are made possible by a new bound on the second derivatives of the solution to the Stein equation for the multivariate normal distribution.
We present applications to the multivariate normal approximation of first order Poisson integrals and of statistics of Boolean models.



\end{abstract}

\section{Introduction and main results}

\subsection{Overview}\label{subsec:overview}

Roughly speaking, a first order Poincar\'e inequality for a random variable $F$ measures the closeness of $F$ to its mean. A second order Poincar\'e inequality \cite{Chatterjee} measures the closeness of $F$ to a Gaussian random variable, where distance is given by some specified metric on the space of distribution functions.  The paper \cite{LPS} establishes second order Poincar\'e inequalities for Poisson functionals $F$, with bounds given in terms of integrated moments of first and second order difference operators, which are an outcome of the research on the Malliavin-Stein method for Poisson functionals in the recent years; see, for example, \cite{EichelsbacherThaele2014,PSTU,Schulte2016} and the book \cite{PeccatiReitzner}. The bounds from \cite{LPS} can be usefully applied to yield rates of normal convergence for various functionals of Poisson processes, including those represented as a sum of stabilizing score functions \cite{LSY}.  The rates are presumably optimal as they coincide with rates of convergence in the classical central limit theorem.

The goal of this paper is to establish second order Poincar\'e inequalities for Poisson functionals
in the multivariate setting, providing multivariate counterparts to the univariate results of \cite{LPS}. The proofs combine Malliavin calculus on Poisson spaces with Stein's method of multivariate normal approximation. Optimal rates of normal convergence depend on good bounds on the terms occurring in a certain smoothing lemma. A main contribution of this paper is to provide such bounds via a new estimate on the second derivatives of the solution to the Stein equation for the multivariate normal distribution, which could be helpful for the multivariate normal approximation of other types of random vectors as well and, thus, might be of independent interest.  It is shown that this approach yields the same (presumably optimal) rates of multivariate normal convergence for the $d_{convex}$-distance based on non-smooth test functions as well as for the $d_2$- and $d_3$-distances based on smooth test functions (see Subsection \ref{sec:MainResults} for definitions of the distances).

We start by making our terms precise and recalling the univariate set-up. Let $\eta$ be a Poisson process over a measurable space $(\mathbb{X},  \mathcal{F})$ with a $\sigma$-finite intensity measure $\lambda$ (see e.g.\ \cite{LastPenroseBook} for more details on Poisson processes). One can think of $\eta$ as a random element in the space $\mathbf{N}$ of all $\sigma$-finite counting measures equipped with the $\sigma$-field generated by the mappings $\nu\mapsto \nu(A)$, $A\in \mathcal{F}$. We call a random variable $F$ a Poisson functional if there is a measurable map $f:\mathbf{N}\to\R$ such that $F=f(\eta)$ almost surely. For such a Poisson functional $F$ the difference operator is given by
\begin{equation}\label{eqn:DifferenceOperator}
D_xF :=
f(\eta+\delta_x)-f(\eta), \quad x\in\mathbb{X},
\end{equation}
where $\delta_x$ denotes the Dirac measure of $x$. We say that $F$ belongs to the domain of the difference operator, i.e., $F\in\operatorname{dom} D$, if $\E F^2<\infty$ and
\begin{equation}\label{eqn:DomD}
\int_{\mathbb{X}} \E(D_xF)^2 \, \lambda(\dint x)<\infty.
\end{equation}
Iterating the definition of the difference operator, one obtains
$$
D^2_{x_1,x_2} F:=D_{x_1}(D_{x_2}F) = f(\eta+\delta_{x_1}+\delta_{x_2})-f(\eta+\delta_{x_1})-f(\eta+\delta_{x_2})+f(\eta), \quad x_1,x_2\in\mathbb{X}.
$$

It is natural to investigate the proximity between
the distribution of $F$ and that of a standard Gaussian
random variable $N$. To compare two random variables $Y$ and $Z$ or, more precisely, their distributions, one can use the Kolmogorov distance
\begin{equation}\label{eqn:DefinitiondK}
d_K(Y,Z):=\sup_{u\in\R} |\mathbb{P}(Y\leq u) - \mathbb{P}(Z\leq u)|,
\end{equation}
which is the supremum norm of the difference of the distribution functions of $Y$ and $Z$, or the Wasserstein distance
$$
d_W(Y,Z):= \sup_{h\in\operatorname{Lip}(1)} |\E h(Y) - \E h(Z)|,
$$
where $\operatorname{Lip}(1)$ stands for the set of functions $h: \R\to\R$ with Lipschitz constant at most one. Note that the $d_K$-distance is always defined, while the $d_W$-distance requires finiteness of $\E |Y|$ and $\E|Z|$.

When $F\in\operatorname{dom} D$, $\E F=0$, and $\Var F=1$, the main results of \cite{LPS} establish the inequalities
\begin{equation}\label{eqn:PoincaredW}
d_W(F,N) \leq \tau_1 + \tau_2 + \tau_3
\end{equation}
and
\begin{equation}\label{eqn:PoincaredK}
d_K(F,N) \leq \tau_1 + \tau_2 + \tau_3 + \tau_4 + \tau_5 + \tau_6,
\end{equation}
where $\tau_1,\hdots,\tau_6$ are integrals over moments involving only $DF$ and $D^2F$ (see Subsection 1.2 in \cite{LPS} for exact formulas).
The proximity bounds  \eqref{eqn:PoincaredW} and \eqref{eqn:PoincaredK}, whose proofs rely on previous Malliavin-Stein bounds in \cite{PSTU} and \cite{EichelsbacherThaele2014,Schulte2016}, respectively, are second order Poincar\'e inequalities, as
 described in \cite{LPS}. The reason for this name is that the `first order' Poincar\'e inequality
$$
\Var F \leq \int_{\mathbb{X}} \E (D_xF)^2 \, \lambda(\dint x)
$$
for $F\in\operatorname{dom} D$ bounds the variance in terms of the first difference operator, whereas the first and the second difference operator control the closeness to Gaussianity in \eqref{eqn:PoincaredW} and \eqref{eqn:PoincaredK}. The term second order Poincar\'e inequality was coined in \cite{Chatterjee} in a similar Gaussian framework, where one has the first two derivatives instead of the first two difference operators.

For many Poisson functionals $F$ the second order Poincar\'e inequalities \eqref{eqn:PoincaredW} and \eqref{eqn:PoincaredK} may be evaluated since the first two difference operators have a clear interpretation via the operation of adding additional points. This is the advantage of these findings over  Malliavin-Stein bounds for normal approximation of Poisson functionals which either require the knowledge of the chaos expansion of $F$ (see, for example, \cite{EichelsbacherThaele2014,HLS2016,PSTU,Schulte2016}) or which involve bounds expressed in terms of gradient operators and conditional expectations as in \cite{PT}.

 Inequality \eqref{eqn:PoincaredK} yields rates of normal approximation for some classic problems in stochastic geometry and some non-linear functionals of Poisson-shot-noise processes \cite{LPS}, as well as for functionals of convex hulls of random samples in a smooth convex body, statistics of nearest neighbors graphs, the number of maximal points in a random sample, and estimators of surface area and volume arising in set approximation \cite{LSY}. The rates of convergence for these examples are presumably optimal.

Often one is not only interested in the behavior of a single Poisson functional but in that of a vector $F=(F_1,\hdots,F_m)$ of Poisson functionals $F_1,\hdots,F_m$ with $m\in\N$. In this situation, one can compare $F$ with an $m$-dimensional centered Gaussian random vector $N_\Sigma$ with covariance matrix $\Sigma\in\R^{m\times m}$. We are not only interested in the weak convergence of the vector $F$ of Poisson functionals to a limit random vector $N_{\Sigma}$, which can be deduced from the univariate case by the Cramer-Wold technique, but in quantitative bounds for the proximity between $F$ and $N_\Sigma$. In other words, we seek the multivariate counterparts of \eqref{eqn:PoincaredW} and \eqref{eqn:PoincaredK}.

In this paper $F$ and $N_\Sigma$ are compared with respect to distances based on smooth and non-smooth test functions. One of our main achievements is to show that for each distance, the bounds are of the same, presumably optimal, order. In general, it is more intricate to deal with non-smooth test functions when one uses Stein's method for multivariate normal approximation. For some bounds for smooth test functions having the same order as in the univariate case we refer to \cite[Chapter 12]{CGS} and the references therein. For non-smooth test functions, even obtaining the rate $n^{-1/2}$ in the classical central limit theorem for sums of $n$ i.i.d.\ random vectors via Stein's method is challenging  \cite{BhattacharyaHolmes2010, Goetze1991}. The abstract multivariate normal approximation results in terms of the dependence structure in \cite{RR} and \cite[Chapter 12]{CGS} and in terms of exchangeable pairs in \cite{ReinertRoellin2009} contain at least additional logarithmic factors compared to what one would expect from the case of smooth test functions or from the univariate case. Recently, these logarithms were removed in \cite{Fang} and \cite{FangRoellin} (see also \cite{FangPhD}), using the dependence structure and Stein couplings, respectively. However, it seems that none of these findings could be applied to systematically achieve the normal approximation bounds for Poisson functionals given by our main results.

\subsection{Statement of main results}\label{sec:MainResults}

Let us now give a precise formulation of our results. We start with distances defined in terms of smooth test functions, namely the
$d_2$- and the $d_3$-distances. Let $\mathcal{H}_m^{(2)}$ be the set of all $C^2$-functions $h: \R^m \to \R$ such that
$$
|h(x)-h(y)|\leq \|x-y\|, \quad x,y\in\R^m, \quad \text{and} \quad \sup_{x\in\R^m} \|\operatorname{Hess} h(x)\|_{op}\leq 1,
$$
where $\operatorname{Hess} h$ denotes the Hessian matrix of $h$ and $\|\cdot\|_{op}$ stands for the operator norm of a matrix. By $\mathcal{H}_m^{(3)}$ we denote the class of all $C^3$-functions $h: \R^m \to\R$ such that the absolute values of the second and third partial derivatives are bounded by one. Using this notation, we define, for $m$-dimensional random vectors $Y$ and $Z$,
$$
d_2(Y,Z):=\sup_{h\in\mathcal{H}^{(2)}_m} |\E h(Y) - \E h(Z)|
$$
if $\E\|Y\|,\E\|Z\|<\infty$ and
$$
d_3(Y,Z):=\sup_{h\in\mathcal{H}^{(3)}_m} |\E h(Y) - \E h(Z)|
$$
if $\E\|Y\|^2,\E\|Z\|^2<\infty$.

The paper \cite{PSTU} was the first to combine Stein's method and the Malliavin calculus to obtain normal approximation of Poisson functionals. In \cite{PeccatiZheng}, the univariate main result of \cite{PSTU} for the $d_W$-distance is extended to vectors of Poisson functionals and the $d_2$- and the $d_3$-distances are considered. Evaluating these multivariate Malliavin-Stein bounds in the same way one evaluates in \cite{LPS} the univariate bounds from \cite{PSTU} and \cite{EichelsbacherThaele2014,Schulte2016} to derive \eqref{eqn:PoincaredW} and \eqref{eqn:PoincaredK}, one obtains the following multivariate second order Poincar\'e inequalities.

\begin{theo}\label{thm:GeneralMultivariateBound}
Let $F=(F_1,\hdots,F_m)$, $m\in\N$, be a vector of Poisson functionals $F_1,\hdots,F_m\\ \in\operatorname{dom}D$ with $\E F_i=0$, $i\in\{1,\hdots,m\}$. Define
\begin{align*}
\gamma_1 & :=  \bigg( \sum_{i,j=1}^m \int_{\mathbb{X}^3} \big( \E (D^2_{x_1,x_3}F_i)^2 (D^2_{x_2,x_3}F_i)^2 \big)^{1/2}\\
& \hskip 2.5cm  \times \big( \E (D_{x_1}F_j)^2 (D_{x_2}F_j)^2 \big)^{1/2} \lambda^3(\dint(x_1,x_2,x_3)) \bigg)^{1/2}   \allowdisplaybreaks\\
\gamma_2 & :=  \bigg(  \sum_{i,j=1}^m \int_{\mathbb{X}^3} \big( \E (D^2_{x_1,x_3}F_i)^2 (D^2_{x_2,x_3}F_i)^2 \big)^{1/2}\\
& \hskip 2.5cm  \times \big( \E (D^2_{x_1,x_3}F_j)^2 (D^2_{x_2,x_3}F_j)^2 \big)^{1/2} \, \lambda^3(\dint(x_1,x_2,x_3))\bigg)^{1/2} \allowdisplaybreaks\\
\gamma_3 & := \sum_{i=1}^m \int_{\mathbb{X}} \E |D_xF_i|^3 \, \lambda(\dint x)
\end{align*}
and let $\Sigma=(\sigma_{ij})_{i,j\in\{1,\hdots,m\}}\in\R^{m\times m}$ be positive semi-definite. Then
\begin{equation} \label{GMB}
d_3(F,N_{\Sigma}) \leq \frac{m}{2} \sum_{i,j=1}^m |\sigma_{ij}-\Cov(F_i,F_j)|
 + m  \gamma_1 + \frac{m}{2} \gamma_2 + \frac{m^2}{4} \gamma_3.
\end{equation}
If, additionally, $\Sigma$ is positive definite, then
\begin{equation} \label{GMBd2}
\begin{split}
d_2(F,N_{\Sigma}) & \leq \|\Sigma^{-1}\|_{op} \|\Sigma\|_{op}^{1/2}  \sum_{i,j=1}^m |\sigma_{ij}-\Cov(F_i,F_j)|
 + 2\|\Sigma^{-1}\|_{op} \|\Sigma\|_{op}^{1/2}  \gamma_1\\
& \quad  + \|\Sigma^{-1}\|_{op} \|\Sigma\|_{op}^{1/2} \gamma_2 + \frac{\sqrt{2\pi} m^2 }{8} \|\Sigma^{-1}\|_{op}^{3/2} \|\Sigma\|_{op} \gamma_3.
\end{split}
\end{equation}
\end{theo}

Note that $\gamma_1$, $\gamma_2$, and $\gamma_3$ have a structure similar to that of $\tau_1$, $\tau_2$, and $\tau_3$ in \eqref{eqn:PoincaredW} and \eqref{eqn:PoincaredK} and coincide with them up to some constant factors for $m=1$.

Let us now compare Theorem \ref{thm:GeneralMultivariateBound} with related results in the literature. The bounds in \cite{PeccatiZheng} are formulated in terms of the difference operator $D$ and the inverse Ornstein-Uhlenbeck generator $L^{-1}$ and do not, in general, readily lend themselves to off-the-shelf use.  In contrast, bounds such as \eqref{GMB} and \eqref{GMBd2} involving only difference operators are often tractable, as seen in our applications section and also in the companion paper \cite{SchulteYukich2017}.
 Theorem 8.1 of \cite{HLS2016} provides a bound on $d_3(F,N_{\Sigma})$, which relies on the findings of \cite{PeccatiZheng}, though this bound requires knowledge of the entire Wiener-It\^{o} chaos expansion for each of the components of $F$ and consequently may also  be less useful than \eqref{GMB}.  When the components of $F$ belong to a special class of Poisson $U$-statistics, which admit a finite chaos expansion with explicitly known kernels, the paper \cite{LPST} uses the results of \cite{PeccatiZheng} to establish bounds for the $d_3$-distance between $F$ and a Gaussian random vector. In \cite{BourguinPeccati}, the findings from \cite{PeccatiZheng} are generalized by comparing a vector of Poisson functionals with a random vector composed of Gaussian and Poisson random variables.

The paper \cite{KrokowskiThaele2017} derives multivariate second order Poincar\'e inequalities for functionals of Rademacher sequences. The considered $d_4$-distance is based on test functions such that the $\sup$-norms of the first four partial derivatives are bounded by one.

To some extent \eqref{GMB} and \eqref{GMBd2} can be seen as multivariate counterparts of \eqref{eqn:PoincaredW}. Indeed, as is the case with $d_W$, the distances $d_2$ and $d_3$ are based on continuous test functions, although the exact definitions involving $C^2$- and $C^3$-functions are distinct from the multivariate Wasserstein distance obtained by using  test functions $h: \R^m\to\R$ having Lipschitz constants at most one.

The Kolmogorov distance \eqref{eqn:DefinitiondK} is arguably  more interesting than the Wasserstein distance (and the $d_2$- and the $d_3$-distances for $m=1$), as it has a clearer interpretation as the supremum norm of the difference of the distribution functions, though it is often harder to deal with because the underlying test functions are discontinuous.
The straightforward multivariate analog to the univariate Kolmogorov distance for two $m$-dimensional random vectors $Y=(Y_1,\hdots,Y_m)$ and $Z=(Z_1,\hdots,Z_m)$ would be
\begin{equation}\label{eqn:SupMulti}
d_K(Y,Z):=\sup_{u_1,\hdots,u_m\in\R}|\mathbb{P}(Y_1\leq u_1,\hdots,Y_m\leq u_m)-\mathbb{P}(Z_1\leq u_1,\hdots,Z_m\leq u_m)|,
\end{equation}
which is again the supremum norm of the difference of the distribution functions of $Y$ and $Z$. In \eqref{eqn:SupMulti} one only takes into account rectangular solids aligned with coordinate planes, so that for a rotation $A\in\R^{m\times m}$ the distance between $AY$ and $AZ$ could be different from the distance between $Y$ and $Z$. Although convergence in the distance given in \eqref{eqn:SupMulti} still implies weak convergence, one would like to have invariance under rotation. To resolve this issue, one considers the following standard multivariate counterpart to the Kolmogorov distance \eqref{eqn:DefinitiondK}, defined for $m$-dimensional random vectors $Y$ and $Z$ \text{by}
$$
d_{convex}(Y,Z):=\sup_{h\in\mathcal{I}_m} |\E h(Y) - \E h(Z)|,
$$
where $\mathcal{I}_m$ is the set of all indicator functions of measurable convex sets in $\R^m$.

For a vector $F=(F_1,\hdots,F_m)$, $m\in\N$, of Poisson functionals $F_1,\hdots,F_m\in\operatorname{dom} D$ with $\E F_i=0$, $i\in\{1,\hdots,m\}$, we use the abbreviations  $D_xF:=(D_xF_1,\hdots,D_xF_m)$ for $x\in\mathbb{X}$, $D^2_{x,y}F:=(D^2_{x,y}F_1,\hdots,D^2_{x,y}F_m)$ for $x,y\in\mathbb{X}$, and
\begin{align*}
\gamma_4 & :=\bigg(\sum_{i,j=1}^m  \int_{\mathbb{X}} \E (D_xF_i)^4 \, \lambda(\dint x) + 6 \int_{\mathbb{X}^2} \big(\E (D^2_{x,y}F_i)^4\big)^{1/2}  \big(\E (D_xF_j)^4\big)^{1/2}  \, \lambda^2(\dint (x,y)) \\
& \hskip 1.75cm + 3 \int_{\mathbb{X}^2}  \big(\E (D^2_{x,y}F_i)^4\big)^{1/2} \big(\E (D^2_{x,y}F_j)^4 \big)^{1/2}  \, \lambda^2(\dint (x,y)) \bigg)^{1/2} \allowdisplaybreaks\\
\gamma_5 & := \bigg(3 \sum_{i,j=1}^m \int_{\mathbb{X}^3} \big( \E \mathbf{1}\{D^2_{x_1,y}F\neq \0, D^2_{x_2,y}F\neq \0\} \big(\|D_{x_1}F\|+\|D^2_{x_1,y}F\|\big)^{3/4}\\
& \hskip 3cm \times \big(\|D_{x_2}F\|+\|D^2_{x_2,y}F\|\big)^{3/4}  |D_{x_1}F_i|^{3/2} \, |D_{x_2}F_i|^{3/2}\big)^{2/3}\\
& \hskip 2.5cm \times \big( \E |D_{x_1}F_j|^3 |D_{x_2}F_j|^3 \big)^{1/3}  \, \lambda^3(\dint(x_1,x_2,y)) \\
& \quad \quad + \sum_{i,j=1}^m  \int_{\mathbb{X}^3} \big( \E \big( \|D_{x_1}F\|+\|D^2_{x_1,y}F\|\big)^{3/2} \big( \|D_{x_2}F\|+\|D^2_{x_2,y}F\|\big)^{3/2} \big)^{1/3}  \\
& \hskip 3cm \times \bigg( \frac{45}{2} \big(\E |D^2_{x_1,y}F_i|^3 |D^2_{x_2,y}F_i|^3 \big)^{1/3} \big(\E |D_{x_1}F_j|^3 |D_{x_2}F_j)^3 \big)^{1/3} \\
& \hskip 3.75cm  +  \frac{9}{2} \big(\E |D^2_{x_1,y}F_i|^3 \, |D^2_{x_2,y}F_i|^3 \big)^{1/3} \big(\E |D^2_{x_1,y}F_j|^3 |D^2_{x_2,y}F_j|^3 \big)^{1/3} \bigg) \\
& \hskip 4cm \lambda^3(\dint(x_1,x_2,y)) \bigg)^{1/3} \allowdisplaybreaks \\
\gamma_6 & := \bigg(3 \sum_{i,j=1}^m \int_{\mathbb{X}^3} \big( \E \mathbf{1}\{D^2_{x_1,y}F\neq \0, D^2_{x_2,y}F\neq \0\} \big(\|D_{x_1}F\|^2+\|D^2_{x_1,y}F\|^2\big)^{3/4}\\
& \hskip 3cm \times \big(\|D_{x_2}F\|^2+\|D^2_{x_2,y}F\|^2\big)^{3/4}  |D_{x_1}F_i|^{3/2} \, |D_{x_2}F_i|^{3/2}\big)^{2/3}\\
& \hskip 2.5cm \times \big( \E |D_{x_1}F_j|^3 |D_{x_2}F_j|^3 \big)^{1/3}  \, \lambda^3(\dint(x_1,x_2,y)) \\
& \quad \quad + \sum_{i,j=1}^m  \int_{\mathbb{X}^3} \big( \E \big( \|D_{x_1}F\|^2+\|D^2_{x_1,y}F\|^2\big)^{3/2} \big( \|D_{x_2}F\|^2+\|D^2_{x_2,y}F\|^2\big)^{3/2} \big)^{1/3}  \\
& \hskip 3cm \times \bigg( \frac{135}{8} \big(\E |D^2_{x_1,y}F_i|^3 |D^2_{x_2,y}F_i|^3 \big)^{1/3} \big(\E |D_{x_1}F_j|^3 |D_{x_2}F_j|^3 \big)^{1/3} \\
& \hskip 3.5cm  + \frac{27}{8} \big(\E |D^2_{x_1,y}F_i|^3 \, |D^2_{x_2,y}F_i|^3 \big)^{1/3} \big(\E |D^2_{x_1,y}F_j|^3 |D^2_{x_2,y}F_j|^3 \big)^{1/3} \bigg) \\
& \hskip 4cm \lambda^3(\dint(x_1,x_2,y))\bigg)^{1/4} ,
\end{align*}
where $\0$ stands for the origin in $\R^m$.

The following multivariate second order Poincar\'e inequality  for the $d_{convex}$-distance
constitutes our main finding. The inequality is the multivariate counterpart to the bound for the Kolmogorov distance at \eqref{eqn:PoincaredK} established in \cite{LPS} and it closely resembles
those for the $d_2$- and $d_3$-distances at \eqref{GMB} and \eqref{GMBd2}. For a positive definite matrix $\Sigma\in\mathbb{R}^{m \times m}$ let $\Sigma^{1/2}$ be the positive definite matrix in $\R^{m\times m}$ such that $\Sigma^{1/2}\Sigma^{1/2}=\Sigma$ and let $\Sigma^{-1/2}:=(\Sigma^{1/2})^{-1}$.

\begin{theo}\label{thm:Generaldconvex}
Let $F=(F_1,\hdots,F_m)$, $m\in\N$, be a vector of Poisson functionals $F_1,\hdots,F_m\\ \in\operatorname{dom} D$ with $\E F_i=0$, $i\in\{1,\hdots,m\}$, and let $\Sigma=(\sigma_{ij})_{i,j\in\{1,\hdots,m\}}\in\R^{m\times m}$ be positive definite. Then
\begin{equation} \label{GMBdHl}
\begin{split}
d_{convex}(F,N_\Sigma) & \leq 941 m^5 \max\{\|\Sigma^{-1/2}\|_{op}, \|\Sigma^{-1/2}\|_{op}^{3}\} \\
& \quad \times \max\bigg\{\sum_{i,j=1}^m |\sigma_{ij}-\Cov(F_i,F_j)|, \gamma_1,\gamma_2, \gamma_3, \gamma_4, \gamma_5, \gamma_6\bigg\}
\end{split}
\end{equation}
with $\gamma_1$, $\gamma_2$, and $\gamma_3$ as in Theorem \ref{thm:GeneralMultivariateBound} and with
$\gamma_4$, $\gamma_5$, and $\gamma_6$ defined as above.
\end{theo}

Several existing results for the multivariate normal approximation of general random vectors in the $d_{convex}$-distance or generalizations of it \cite{CGS,Fang,FangRoellin,RR} all require some almost sure boundedness assumptions; in our set-up this would amount to requiring that $|D_xF_i|$ is almost surely bounded for $x\in\mathbb{X}$ and $i\in\{1,\hdots,m\}$. One of the main achievements of Theorem \ref{thm:Generaldconvex} is that no such assumption is required. For results without almost sure boundedness assumption we refer to \cite[Chapter 3]{FangPhD} and, with weaker rates of convergence, to \cite[Corollary 3.1]{ReinertRoellin2009}.

A second main achievement of  Theorem \ref{thm:Generaldconvex} is that there are no logarithmic terms in the bound \eqref{GMBdHl} (see the discussion at the end of Subsection \ref{subsec:overview}).
The Malliavin-Stein method is used in \cite{NPR} to establish bounds in the $d_W$-distance for the multivariate normal approximation of functionals of Gaussian processes. In \cite{KimPark}, a similar bound with  an additional logarithm is derived for the $d_{convex}$-distance. As with Theorem \ref{thm:Generaldconvex}, the latter result does not require any boundedness assumptions. 
  Moreover, we expect that one can use our proof technique to remove the logarithm from the result in \cite{KimPark}. For a subclass of functionals of Gaussian processes,  namely multiple Wiener-It\^o integrals, one may even establish rates of multivariate normal approximation with respect to the total variation distance \cite{NourdinPeccatiSwan}. This bound also involves additional logarithmic factors and its proof relies on controlling the relative entropy, an approach which differs from Stein's method.

Clearly, if the random vector $N_{\Sigma}$ is replaced by a normal random vector whose covariance matrix consists of entries $\Cov(F_i, F_j)$, then the term $\sum_{i,j=1}^m |\sigma_{ij}- \Cov(F_i,F_j)|$ in the  bounds of our main theorems disappears.

 In Theorem \ref{thm:Generaldconvex} we require that the covariance matrix $\Sigma$ of the approximating Gaussian random vector $N_\Sigma$ is positive definite. Otherwise, $N_{\Sigma}$ would be concentrated on some lower-dimensional linear subspace of $\R^m$. If now $F$ were to belong to any given lower dimensional subspace of $\R^m$ with probability zero, then we would have $d_{convex}(F,N_{\Sigma})\geq 1$. In such situations, one could have weak convergence without convergence in $d_{convex}$.

\subsection{Examples and applications}

At first sight, the bounds in our general results appear unwieldy. However for many functionals of interest, we may readily bound the integrated moments of difference operators and the terms $\gamma_1,\hdots,\gamma_6$ simplify. We illustrate this by four examples, which indicate that our bounds yield presumably optimal rates of convergence.

We start with the following analog to the classical central limit theorem for sums of i.i.d.\ random vectors, where we consider the sum of a Poisson distributed number of i.i.d.\ random vectors.  Here, as in Theorems \ref{thm:GeneralMultivariateBound} and \ref{thm:Generaldconvex}, we implicitly assume that the normal approximation bounds all involve finite quantities, as otherwise there is nothing to prove.  The proof of the following result is postponed to Subsection \ref{subsec:I1}.

\begin{coro}\label{cor:AnalogCLT}
Given  a Poisson distributed random variable $Y$ with mean $s>0$ and a sequence of i.i.d.\ centered random vectors $(X_n)_{n\in\N}$ in $\R^m$, which are independent of $Y$, define
$$
Z_s:=\frac{1}{\sqrt{s}}\sum_{n=1}^Y X_n \quad  \text{and} \quad \Sigma:=(\Cov(X_1^{(i)},X_1^{(j)}))_{i,j\in\{1,\hdots,m\}}.
$$
\begin{itemize}
\item [(a)] It is the case that
$$
d_3(Z_s,N_\Sigma) \leq \frac{m^2}{4} \sum_{i=1}^m \E |X_1^{(i)}|^3 \frac{1}{\sqrt{s}}.
$$
\item [(b)] When $\Sigma$ is positive definite we have
$$
d_2(Z_s,N_\Sigma) \leq \frac{\sqrt{2\pi}m^2}{8} \|\Sigma^{-1}\|_{op}^{3/2} \|\Sigma\|_{op} \sum_{i=1}^m \E |X_1^{(i)}|^3 \frac{1}{\sqrt{s}}.
$$
\item [(c)] When $\Sigma$ is positive definite we have
\begin{equation}
\begin{split}
d_{convex}(F,N_\Sigma) & \leq 941 m^{11/2} \max\{\|\Sigma^{-1/2}\|_{op}, \|\Sigma^{-1/2}\|_{op}^{3}\} \\
& \quad \times \max \bigg\{ \sum_{i = 1}^m \E |X^{(i)}_1|^3,  \sqrt{ \sum_{i = 1}^m \E (X^{(i)}_1)^4 } \bigg\}  \frac{1}{\sqrt{s}}.
\end{split}
\end{equation}
\end{itemize}
\end{coro}

Since one can rewrite $Z_s$ as a sum of a fixed number of i.i.d.\ random vectors, one can also apply the classical multivariate central limit theorem. In \cite{BhattacharyaHolmes2010,Goetze1991,Sazonov1968} corresponding Berry-Esseen inequalities for the $d_{convex}$-distance are derived, which provide in the case of Corollary \ref{cor:AnalogCLT} rates of convergence of the order $1/\sqrt{s}$ as well. These findings are even stronger since they require for the $d_{convex}$-distance only finite third moments, while we require finite fourth moments. The stricter assumptions in Corollary \ref{cor:AnalogCLT} might come from the fact that the proofs of the underlying results for more general Poisson functionals are not optimized for the considered special case. Since $Z_s$ is a vector of first order Poisson integrals,  Corollary \ref{cor:AnalogCLT} follows from a more general theorem in Subsection \ref{subsec:I1}, which is obtained by applying our main results to first order Poisson integrals.

As a second example we consider for fixed $m\in\N$ a family of vectors $F_s=(F_{1,s},\hdots,F_{m,s})$, $s>0$, of square integrable Poisson functionals $F_{1,s},\hdots,F_{m,s}$ with underlying Poisson processes $\eta_s$, $s>0$, having intensity measures $\mu_s$, $s>0$, of the form $\mu_s=s\mu$ with a fixed finite measure $\mu$, e.g., homogenous Poisson processes on the $d$-dimensional unit cube $[0,1]^d$ with increasing intensity. Moreover, we denote by $\Sigma_s$ the covariance matrix of $F_s$ and assume that $(\Sigma_s)_{s>0}$ converges to a matrix $\overline{\Sigma}\in\R^{m\times m}$. Under some additional assumptions on the difference operators our main results imply the following result, proved in Subsection \ref{subsec:Marks}.

\begin{coro}\label{coro:BoundedDifferenceOperator}
Let $F_s$, $s>0$, be as above and assume that $\overline{\Sigma}$ is positive definite and that there are constants $a,b,\varepsilon\in(0,\infty)$ such that, for $i\in\{1,\hdots,m\}$ and $s>0$,
\be \label{As11}
\E |D_xF_{i,s}|^{6+\varepsilon} \leq \frac{a}{s^{3+\varepsilon/2}}, \quad \mu \text{-a.e.}\ x\in \mathbb{X},
\ee
\be \label{As12}
\E |D^2_{x_1,x_2}F_{i,s}|^{6+\varepsilon}\leq \frac{a}{s^{3+\varepsilon/2}}, \quad \mu^2 \text{-a.e.}\ (x_1,x_2)\in \mathbb{X}^2,
\ee
and
\be \label{As2}
s\int_{\mathbb{\mathbb{X}}} \mathbb{P}(D^2_{x,y}F_{i,s}\neq 0)^{\frac{\varepsilon}{36+6\varepsilon}} \, \mu (\dint y)\leq b, \quad \mu \text{-a.e.}\ x \in \mathbb{X}.
\ee
Then there exist constants $s_0,C_{3},C_2, C_{convex}\in(0,\infty)$ depending on $a$, $b$, $\varepsilon$, $m$, $\mu(\mathbb{X})$, $\overline{\Sigma}$, and $(\Sigma_s)_{s>0}$ such that
$$
d_3(F_s,N_{\Sigma_s}) \leq \frac{C_3}{\sqrt{s}}, \quad d_2(F_s,N_{\Sigma_s}) \leq \frac{C_2}{\sqrt{s}}, \quad \text{and} \quad d_{convex}(F_s,N_{\Sigma_s}) \leq \frac{C_{convex}}{\sqrt{s}}
$$
for $s\geq s_0$.
\end{coro}

The rates of convergence in Corollary \ref{coro:BoundedDifferenceOperator} are of the order $s^{-1/2}$ for all distances.
The set-up of Corollary \ref{coro:BoundedDifferenceOperator}, in which one re-scales by the square root of the intensity parameter and in which the $(6+\varepsilon)$-th moments of the un-rescaled difference operators are bounded, frequently occurs in problems in stochastic geometry; see e.g. \cite{LSY,LPS}.

The third example is the situation where, before centering, the components of $F$ have representations $ s^{-1/2} \sum_{x\in\eta_{sg}\cap A_i} \xi^{(i)}_s(x,\eta_{sg})$, $i\in\{1,\hdots,m\}$, with $s\in[1,\infty)$, where $A_i$, $i\in\{1,\hdots,m\}$, are bounded subsets of $\R^d$, $\eta_{sg}$ is a Poisson process in $\R^d$ whose intensity measure has density $sg$ with respect to the Lebesgue measure,  and where $\xi^{(i)}_s$, $i\in\{1,\hdots,m\}$, are stabilizing score functions. Then the companion paper \cite{SchulteYukich2017}, which can be seen as a multivariate counterpart to some of the findings in \cite{LSY}, shows that the right-hand sides of \eqref{GMB}, \eqref{GMBd2}, and \eqref{GMBdHl} reduce to $O(\sum_{i,j=1}^m |\sigma_{ij}-\Cov(F_i,F_j)|) + O(s^{-1/2})$ under some assumptions on $(\xi_s^{(i)})_{s\geq 1}$, $i\in\{1,\hdots,m\}$, $A_i$, $i\in\{1,\hdots,m\}$, and $g$. This means that the approximation error consists of a term taking into account the difference of the covariances and a term of order $s^{-1/2}$, which also occurs in the univariate case (see \cite{LSY}). In Section 3 of \cite{SchulteYukich2017}, these findings are applied to obtain quantitative multivariate central limit theorems for statistics of $k$-nearest neighbors graphs and random geometric graphs as well as for statistics arising in topological data analysis and entropy estimation.

A fourth example concerns the intrinsic volumes of Boolean models, a prominent problem from stochastic geometry. Let $V_d(W)$ be the volume of the compact convex observation window $W\subset\mathbb{R}^d$. If one compares the vector of intrinsic volumes of the Boolean model in $W$ with a centered Gaussian random vector having exactly the same covariance matrix and if one increases the inradius of $W$, then our main results lead to the rate of normal convergence $V_d(W)^{-1/2}$; see Subsection \ref{subsec:BM}.

In the last three examples the rates of convergence $s^{-1/2}$ and $V_d(W)^{-1/2}$, respectively, are comparable to $n^{-1/2}$ in the uni- and multivariate central limit theorems for the i.i.d.\ case and, thus, presumably optimal.

Among these examples, we will consider the first order Poisson integrals generalizing the situation of Corollary \ref{cor:AnalogCLT} and the intrinsic volumes of Boolean models in more detail in Subsections \ref{subsec:I1} and \ref{subsec:BM}, while Corollary \ref{coro:BoundedDifferenceOperator} is a consequence of a theorem derived in Subsection \ref{subsec:Marks}.

\subsection{Proof techniques} \label{prooftech}

Let us now informally comment on the method of proof. The proofs of Theorems \ref{thm:GeneralMultivariateBound} and \ref{thm:Generaldconvex} are based on the Malliavin calculus on the Poisson space and Stein's method for multivariate normal approximation. In particular we apply a smoothing technique, which we discuss in this subsection. Assume we aim to compare an $m$-dimensional random vector $Y=(Y_1,\hdots,Y_m)$ with an $m$-dimensional centered Gaussian random vector $N_I$ with the identity matrix $I\in\R^{m\times m}$ as covariance matrix (we assume $\Sigma=I$ for simplicity) in terms of a measurable test function $h: \R^m\to\R$. The idea of Stein's method for multivariate normal approximation (see e.g.\ \cite{CGS, Goetze1991}) is now to use the identity
$$
\E h(Y) - \E h(N_I) = \E \sum_{i=1}^m Y_i \frac{\partial f_h}{\partial y_i}(Y) - \frac{\partial^2f_h}{\partial y_i^2}(Y) ,
$$
where $f_h: \R^m \to \R$ is a solution of the multivariate Stein equation
\begin{equation}\label{eqn:MultivariateSteinEquation}
\sum_{i=1}^m y_i \frac{\partial f}{\partial y_i}(y) - \frac{\partial^2f}{\partial y_i^2}(y) = h(y) - \E h(N_I), \quad y\in\R^m.
\end{equation}
Under some smoothness assumptions on $h$ one can give formulas for $f_h$ (see, for example, Lemma 2.6 in \cite{CGS}). However for non-smooth $h$ such as indicator functions of convex sets, it appears unclear how to deal with $f_h$. This problem is resolved by considering instead of $h$ some smoothed $C^\infty$ version $h_{t,I}$ of $h$, which depends on a smoothing parameter $t\in(0,1)$. Of course one makes some error by replacing the test functions defining the  $d_{convex}$-distance by their smoothed versions, but a smoothing lemma allows us to bound this error by some constant multiple of $\sqrt{t}$.

Thus it remains to find upper bounds for $|\E h_{t,I}(Y) - \E h_{t,I}(N_I)|$ as a function of $t\in(0,1)$. We sketch how this goes as follows.
Given $h: \R^m\to\R$ measurable and bounded and $t\in(0,1)$ we introduce the smoothed function
\be \label{smoothh}
h_{t,I}(y):= \int_{\R^m} h(\sqrt{t} z +\sqrt{1-t}y) \, \varphi_I(z) \, \dint z, \quad y\in\R^m,
\ee
where $\varphi_I$ denotes the density of $N_I$. The function $f_{t,h,I}: \R^m\to\R$ given by
\be \label{Steinsol}
f_{t,h,I}(y):= \frac{1}{2} \int_t^1 \frac{1}{1-s} \int_{\R^m} (h(\sqrt{s}z+\sqrt{1-s}y)-h(z)) \, \varphi_I(z) \, \dint z\, \dint s, \quad y\in\R^m,
\ee
is a solution of the Stein equation \eqref{eqn:MultivariateSteinEquation} with $h$ replaced by $h_{t,I}$; see \cite[p.\ 726]{Goetze1991} and \cite[p.\ 337]{CGS}.  Moreover, when $\|h\|_\infty := \sup_{x \in \R^m} |h(x)|   \leq 1$, it follows  (see e.g. the first display on p.\ 1498 in \cite{PeccatiZheng}) that, for a vector $F=(F_1,\hdots,F_m)$, $m\in\N$, of Poisson functionals $F_1,\hdots,F_m\in\operatorname{dom} D$ with $\E F_i=0$, $i\in\{1,\hdots,m\}$,
$$
|\E h_{t,I}(F) - \E h_{t,I}(N_{I})|
= \bigg|\sum_{i=1}^m \E\frac{\partial^2 f_{t,h,I}}{\partial y_i^2}(F) - \sum_{k=1}^m \E \int_{\mathbb{X}} D_x \frac{\partial f_{t,h,I}}{\partial y_k}(F) (-D_xL^{-1}F_k) \, \lambda(\dint x)\bigg|,
$$
where $D_x$ is the difference operator given in \eqref{eqn:DifferenceOperator} and $L^{-1}$ is the inverse Ornstein-Uhlenbeck generator defined in the Appendix. A main idea behind the proof of Theorem  \ref{thm:Generaldconvex} is to show that the bound
for the right-hand side of the above involves $$\sqrt{\sum_{i,j=1}^m \E \bigg( \frac{\partial^2 f_{t,h,I}}{\partial y_i\partial y_j}(F)\bigg)^2}$$ and then to use
\begin{equation}\label{eqn:BoundPartialDerivatives}
\sup_{h\in\mathcal{I}_m } \E \sum_{i,j=1}^m \bigg( \frac{\partial^2 f_{t,h, I}}{\partial y_i \partial y_j}(F) \bigg)^2 \leq   M_2 (\log t)^2 d_{convex}(F,N_I) + 530 m^{17/6}
\end{equation}
for all $t\in(0,1)$ and $i,j\in\{1,\hdots,m\}$ where $M_2\leq m^2$. By choosing $t$ appropriately we may deduce Theorem \ref{thm:Generaldconvex}. The inequality \eqref{eqn:BoundPartialDerivatives} is not restricted to a vector $F$ of Poisson functionals, but holds for arbitrary random vectors $Y$ in $\R^m$, as described in Proposition \ref{lem:BoundD2F}. Thus, we expect that it might be helpful for other applications of Stein's method for multivariate normal approximation.

In our main results we provide explicit constants, which are sometimes very large. In part, this is caused by some generous estimates in our proofs, used to obtain relatively short bounds valid for all choices of $m$  and to simplify the proofs. We expect that one could obtain better constants for many instances if one goes back to our proofs and uses the particular stucture of the functionals and the choice of $m$.

\subsection{Structure of the paper}

This paper is organized as follows. The next section provides a smoothing lemma and bounds on solutions of the multivariate Stein equation,  including the afore-mentioned Proposition \ref{lem:BoundD2F}. Section \ref{sec:Proofs}, which draws on the auxiliary results of Section \ref{sec:SteinsMethod}, is devoted to the proofs of our main results. Section \ref{sec:Applications} deals with the application of our findings to first order Poisson integrals and intrinsic volumes of Boolean models. Moreover, we further evaluate our results for the case of marked Poisson processes - a result which will be used in the companion paper \cite{SchulteYukich2017}. In the Appendix we recall the definitions of the Malliavin operators as well as some results from Malliavin calculus on the Poisson space that are used in Section \ref{sec:Proofs}.

\section{Smoothing and the multivariate Stein equation}\label{sec:SteinsMethod}

\subsection{A smoothing lemma for the $d_{convex}$-distance}\label{subsec:Smoothing}

Let $m\in\N$ be fixed in the sequel. Let $\varphi_\Sigma$ denote the density of an $m$-dimensional centered Gaussian random vector $N_\Sigma$ having a positive definite covariance matrix $\Sigma=(\sigma_{ij})_{i,j\in\{1,\hdots,m\}}\in\R^{m\times m}$. Recall that $\Sigma^{1/2}$ and $\Sigma^{-1/2}$ are the positive definite matrices in $\R^{m\times m}$ such that $\Sigma^{1/2}\Sigma^{1/2}=\Sigma$ and $\Sigma^{-1/2}=(\Sigma^{1/2})^{-1}$.

The following result from \cite[p.\ 725]{Goetze1991} (see also \cite[Corollary 3.2]{BhattacharyaRao1976}) is used repeatedly.
For $x \in \R^m$ and a Borel set $B \subseteq \R^m$ we define
$d(x,B):= \inf_{y \in B} \|x - y\|$.

\begin{lemm}\label{lem:Goetze}
For $A\subseteq\R^m$ convex and $r>0$,
$$
\mathbb{P}(d(N_I,\partial A)\leq r) \leq 2 \sqrt{m} r.
$$
\end{lemm}

Given measurable and bounded $h: \R^m\to\R$, positive definite $\Sigma\in\R^{m\times m}$, and $t\in(0,1)$ we introduce the smoothed version
$$
h_{t,\Sigma}(y):= \int_{\R^m} h(\sqrt{t} z +\sqrt{1-t}y) \, \varphi_\Sigma(z) \, \dint z =\E h(\sqrt{t}N_\Sigma+\sqrt{1-t}y), \quad y\in\R^m,
$$
of $h$, extending  \eqref{smoothh} to general $\Sigma$.
The following so-called smoothing lemma (see Lemma 2.11 in \cite{Goetze1991}, Lemma 11.4 in \cite{BhattacharyaRao1976}, or Lemma 12.1 of  \cite{CGS})
allows one to bound the $d_{convex}$-distance to the $m$-dimensional centered Gaussian random vector $N_\Sigma$ with positive definite covariance matrix $\Sigma\in\R^{m\times m}$ in terms of smooth test functions.  Lemma \ref{lem:SmoothingStandard} is the starting point for proving  \eqref{GMBdHl}.

\begin{lemm}\label{lem:SmoothingStandard}
For an $m$-dimensional random vector $Y$, $t\in(0,1)$, and positive definite $\Sigma\in\R^{m\times m}$ we have
$$
d_{convex}(Y, N_\Sigma) \leq \frac{4}{3} \sup_{h\in \mathcal{I}_m } |\E h_{t,\Sigma}(Y)- \E h_{t, \Sigma}(N_\Sigma)|+ \frac{20}{\sqrt{2}} m \frac{\sqrt{t}}{1-t}.
$$
\end{lemm}

\begin{proof}  We first establish that the asserted bound holds when $\Sigma$ is replaced by $I$.  Indeed this is the statement of \cite[Lemma 2.11]{Goetze1991} with $\varepsilon=\sqrt{t}$, $\Delta=2\sqrt{m}$ (see \cite[p.\ 725]{Goetze1991} as well as \cite[Corollary 3.2]{BhattacharyaRao1976}) and $a_m\leq 2\sqrt{2m}$ (which follows from Markov's inequality) there.

Next, to show that this bound holds for positive definite $\Sigma\in\R^{m\times m}$, it suffices to notice that we have
$$
d_{convex}(Y,N_\Sigma)=d_{convex}(Y,\Sigma^{1/2} N_I) = d_{convex}(\Sigma^{-1/2} Y, N_I)
$$
and
$$
\sup_{h\in \mathcal{I}_m  }|\E h_{t,\Sigma}(Y)-\E h_{t,\Sigma}(N_\Sigma)|=\sup_{h\in \mathcal{I}_m  }|\E h_{t,I}(\Sigma^{-1/2}Y)-\E h_{t,I}(N_I)|.
$$
To verify the second identity, notice that for any $h\in\mathcal{I}_m$ the functions $h\circ\Sigma^{1/2}: \R^m\ni x\mapsto h(\Sigma^{1/2} x)$ and $h\circ\Sigma^{-1/2}: \R^m\ni x\mapsto h(\Sigma^{-1/2} x)$ also belong to $\mathcal{I}_m$,
$$
h_{t,\Sigma}(x)=\E h(\sqrt{t}N_\Sigma+\sqrt{1-t}x)=\E  h\circ\Sigma^{1/2}(\sqrt{t}N_I+\sqrt{1-t} \Sigma^{-1/2}x)=( h\circ\Sigma^{1/2})_{t,I}(\Sigma^{-1/2}x),
$$
and similarly $(h\circ\Sigma^{-1/2})_{t,\Sigma}(x)=h_{t,I}(\Sigma^{-1/2}x)$.
\end{proof}

\subsection{Bounds on the derivatives of the solution to Stein's equation for multivariate normal approximation}

We extend the definition of  $f_{t,h, I}$  at  \eqref{Steinsol} to include indices with general covariance matrix $\Sigma$.
This goes as follows. For $h: \R^m \to\R$ measurable and bounded, $\Sigma=(\sigma_{ij})_{i,j\in\{1,\hdots,m\}}\in\R^{m\times m}$ positive definite, and $t\in(0,1)$, the function $f_{t,h,\Sigma}: \R^m\to\R$ given by
$$
f_{t,h,\Sigma}(y):=\frac{1}{2} \int_t^1 \frac{1}{1-s} \int_{\R^m} (h(\sqrt{s}z+\sqrt{1-s}y)-h(z)) \, \varphi_\Sigma(z) \, \dint z\, \dint s, \quad y\in\R^m,
$$
is a solution of the Stein equation
$$
h_{t,\Sigma}(y) - \E h_{t,\Sigma}(N_\Sigma) = \sum_{i=1}^m y_i \frac{\partial f}{\partial y_i}(y) - \sum_{i,j=1}^m \sigma_{ij} \frac{\partial^2f}{\partial y_i \partial y_j}(y), \quad y\in\R^m,
$$
see \cite[p.\ 726]{Goetze1991} and \cite[p.\ 337]{CGS} for $\Sigma=I$ as well as \cite[Lemma 1]{Meckes2009} and \cite[Lemma 3.3]{NPR} for general $\Sigma$. Some calculations show that, for $i,j,k\in\{1,\hdots,m\}$ and $y\in\R^m$,
$$
\frac{\partial f_{t,h,\Sigma}}{\partial y_i}(y) = -\frac{1}{2} \int_t^1 \frac{1}{\sqrt{s}\sqrt{1-s}} \int_{\R^m}  h(\sqrt{s}z+\sqrt{1-s}y) \, \frac{\partial \varphi_{\Sigma}}{\partial y_i}(z) \, \dint z \, \dint s,
$$
\begin{equation}\label{eqn:D2ht}
\frac{\partial^2 f_{t,h,\Sigma}}{\partial y_i \partial y_j}(y) = \frac{1}{2} \int_t^1 \frac{1}{s} \int_{\R^m}  h(\sqrt{s}z+\sqrt{1-s}y) \, \frac{\partial^2 \varphi_{\Sigma}}{\partial y_i \partial y_j}(z) \, \dint z \, \dint s,
\end{equation}
and
\begin{equation}\label{eqn:D3ht}
\frac{\partial^3 f_{t,h,\Sigma}}{\partial y_i \partial y_j \partial y_k}(y) = - \frac{1}{2} \int_t^1 \frac{\sqrt{1-s}}{s^{3/2}} \int_{\R^m}  h(\sqrt{s}z+\sqrt{1-s}y) \, \frac{\partial^3 \varphi_{\Sigma}}{\partial y_i \partial y_j \partial y_k}(z) \, \dint z \, \dint s.
\end{equation}
By $h\circ\Sigma^{1/2}$ we denote the function $\R^m\ni y \mapsto h(\Sigma^{1/2}y)$. It follows from the definition of $f_{t,h,\Sigma}$ that, for $y\in\R^m$,
\begin{equation}\label{eqn:IdentityfthSigma}
\begin{split}
f_{t,h,\Sigma}(y) & = \frac{1}{2} \int_t^1 \frac{1}{1-s} \E[h(\sqrt{s}N_\Sigma + \sqrt{1-s}y) - h(N_\Sigma)] \, \dint s \\
& = \frac{1}{2} \int_t^1 \frac{1}{1-s} \E[h\circ\Sigma^{1/2}(\sqrt{s}N_I + \sqrt{1-s}\Sigma^{-1/2}y) - h\circ\Sigma^{1/2}(N_I)] \, \dint s \\
& = f_{t,h\circ\Sigma^{1/2} ,I}(\Sigma^{-1/2}y).
\end{split}
\end{equation}
Since $\varphi_\Sigma(z)=\varphi_I(\Sigma^{-1/2}z)/\sqrt{\det(\Sigma)}$ for $z\in\R^m$, we have that, for $i,j,k\in\{1,\hdots,m\}$ and $z\in\R^m$,
$$
\frac{\partial^3\varphi_{\Sigma}}{\partial y_i \partial y_j \partial y_k}(z) =  \frac{1}{\sqrt{\det(\Sigma)}} \sum_{u,v,w=1}^m (\Sigma^{-1/2})_{ui} (\Sigma^{-1/2})_{vj} (\Sigma^{-1/2})_{wk} \frac{\partial^3\varphi_{I}}{\partial y_u \partial y_v \partial y_w}(\Sigma^{-1/2}z),
$$
which yields together with a short computation
\begin{equation}\label{eqn:BoundD3Sigma}
\sum_{i,j,k=1}^m \bigg( \frac{\partial^3\varphi_{\Sigma}}{\partial y_i \partial y_j \partial y_k}(z) \bigg)^2 \leq \frac{\|\Sigma^{-1}\|_{op}^3}{\det(\Sigma)} \sum_{i,j,k=1}^m \bigg( \frac{\partial^3\varphi_I}{\partial y_i \partial y_j \partial y_k}(\Sigma^{-1/2}z) \bigg)^2.
\end{equation}

From the above formulas for the derivatives of $f_{t,h,\Sigma}$ one can deduce that
$$
\sup_{y\in\R^m} \bigg|\frac{\partial^2f_{t,h,\Sigma}(y)}{\partial y_i \partial y_j}\bigg| \leq m^2 \|\Sigma^{-1}\|_{op} \|h\|_\infty |\log t|, \quad t\in(0,1),
$$
and
\begin{equation}\label{eqn:BoundPartial3}
\sup_{y\in\R^m} \bigg|\frac{\partial^3f_{t,h,\Sigma}}{\partial y_i \partial y_j \partial y_k}(y)\bigg| \leq 6m^3 \|\Sigma^{-1}\|_{op}^{3/2}  \|h\|_\infty  \frac{1}{\sqrt{t}}, \quad t\in(0,1).
\end{equation}
Sup norm bounds on the derivatives of $f_{t,h,\Sigma}$ go hand-in-hand with the following more useful second moment bound. It is a key to controlling the right-hand side of the smoothing inequality in Lemma \ref{lem:SmoothingStandard},
an essential part of the proof of Theorem  \ref{thm:Generaldconvex}.

\begin{prop}\label{lem:BoundD2F}
Let $Y$ be an $m$-dimensional random vector, let $\Sigma\in\R^{m\times m}$ be positive definite, and define
\begin{equation}\label{eqn:M2}
M_2:= \frac{1}{4} \sum_{i,j=1}^m \bigg(\int_{\R^m} \bigg|\frac{\partial^2 \varphi_{I}}{\partial y_i \partial y_j}(z)\bigg| \, \dint z \bigg)^2 \leq m^2.
\end{equation}
Then
$$
\sup_{h\in\mathcal{I}_m} \E \sum_{i,j=1}^m \bigg( \frac{\partial^2 f_{t,h,\Sigma}}{\partial y_i \partial y_j}(Y) \bigg)^2 \leq \|\Sigma^{-1}\|^2_{op} (M_2 (\log t)^2 d_{convex}(Y,N_\Sigma) +  530m^{17/6}  )
$$
for all $t\in(0,1)$.
\end{prop}

We prepare the proof of Proposition \ref{lem:BoundD2F} with the following lemmas.

\begin{lemm}\label{lem:DistanceGaussian}
For any $\alpha\in(0,1)$,
$$
\sup_{A\subseteq \R^m \text{ convex}} \E \frac{1}{d(N_I,\partial A)^\alpha} \leq 1 + 2 \sqrt{m} \frac{\alpha}{1-\alpha}.
$$
\end{lemm}

\begin{proof}
For any convex $A \subseteq \R^m$ we have that
\begin{align*}
\E \frac{1}{d(N_I,\partial A)^\alpha} & = \int_0^\infty \mathbb{P}(d(N_I,\partial A)^{-\alpha} \geq u) \, \dint u  = \int_0^\infty \mathbb{P}(d(N_I,\partial A) \leq u^{-1/\alpha}) \, \dint u \\
& \leq 1 + \int_1^\infty \mathbb{P}(d(N_I,\partial A) \leq u^{-1/\alpha}) \, \dint u \allowdisplaybreaks \\
& \leq 1 + 2 \sqrt{m} \int_1^\infty u^{-1/\alpha} \, \dint u = 1 + 2\sqrt{m} \frac{\alpha}{1-\alpha},
\end{align*}
where we used Lemma \ref{lem:Goetze} for the last inequality.
\end{proof}

\begin{lemm}\label{lem:HessianDensity}
For any positive definite $\Sigma\in\R^{m\times m}$ and $i,j\in \{1,\hdots,m\}$,
$$
\int_{\R^m} \frac{\partial^2 \varphi_{\Sigma}}{\partial y_i \partial y_j}(z) \, \dint z = 0.
$$
\end{lemm}

\begin{proof}  As noted at display (12.72) of \cite{CGS} we have that the integral of the mixed derivative
$\frac{\partial^2 \varphi_{\Sigma}}{\partial y_i \partial y_j}(z)$
is the mixed derivative of  $x\mapsto\int_{\R^m} \varphi_\Sigma(z + x) \, \dint z$ evaluated
at $x = \0$. The integral is one, so the derivative vanishes.
\end{proof}

\begin{lemm}\label{lem:fNSigma}
For all $h\in\mathcal{I}_m$ and $t\in(0,1)$,
$$
\max_{i,j\in\{1,\hdots,m\}} \E \bigg(\frac{\partial^2 f_{t,h,I}}{\partial y_i \partial y_j}(N_I)\bigg)^2 \leq  530 m^{5/6}.
$$
\end{lemm}

\begin{proof}
Put  $h:={\mathbf 1}\{\cdot \in A\}$ for some measurable convex set $A\subseteq\R^m$. Then, for $i,j\in\{1,\hdots,m\}$ and $y\in\R^m$, it follows from \eqref{eqn:D2ht} that
\begin{align*}
\frac{\partial^2 f_{t,h,I}}{\partial y_i \partial y_j}(y) & = \frac{1}{2} \int_t^1 \frac{1}{s} \int_{\R^m} {\bf 1}\{\sqrt{s} z +\sqrt{1-s} y\in A\} \frac{\partial^2 \varphi_{I} }{\partial y_i \partial y_j}(z) \, \dint z \, \dint s \\
& = \frac{1}{2} \int_t^1 \frac{1}{s} \int_{\R^m} {\bf 1}\{ z \in \frac{1}{\sqrt{s}}(A - \sqrt{1-s} y)\} \frac{ \partial^2 \varphi_{I} } {\partial y_i \partial y_j} (z) \, \dint z \, \dint s.
\end{align*}
For $s\in (0,1)$ and  $y\in \R^m$ let $r_{s,y}:=d(\0, \partial\big(\frac{1}{\sqrt{s}}(A-\sqrt{1-s} y)\big))=\frac{1}{\sqrt{s}} d(\sqrt{1-s}y,\partial A)$. If $\0\notin \frac{1}{\sqrt{s}}(A-\sqrt{1-s} y)$, we have
$$
\bigg| \int_{\R^m} {\bf 1}\{z \in \frac{1}{\sqrt{s}}(A-\sqrt{1-s} y)\} \frac{\partial^2 \varphi_I}{\partial y_i \partial y_j}(z) \, \dint z \bigg| \leq \int_{\R^m\setminus B^m(\0,r_{s,y})} \bigg| \frac{\partial^2 \varphi_I}{\partial y_i \partial y_j}(z) \bigg| \, \dint z,
$$
where $B^m(x,r)$ denotes the closed ball with center $x\in\R^m$ and radius $r\geq 0$. If $\0\in  \frac{1}{\sqrt{s}}(A-\sqrt{1-s} y)$, Lemma \ref{lem:HessianDensity}
implies that
\begin{align*}
& \bigg| \int_{\R^m} {\bf 1}\{z \in \frac{1}{\sqrt{s}}(A-\sqrt{1-s} y)\} \frac{\partial^2 \varphi_I}{\partial y_i \partial y_j}(z) \, \dint z \bigg|\\
& = \bigg| \int_{\R^m} {\bf 1}\{z \notin \frac{1}{\sqrt{s}}(A-\sqrt{1-s} y)\} \frac{\partial^2 \varphi_I}{\partial y_i \partial y_j}(z) \, \dint z \bigg| \leq \int_{\R^m\setminus B^m(\0,r_{s,y})} \bigg| \frac{\partial^2 \varphi_I}{\partial y_i \partial y_j}(z) \bigg| \, \dint z.
\end{align*}
Letting $\phi$ be the density of a standard Gaussian random variable, we have, for all $a\in\R$,
$$
|\phi'(a)|=\frac{1}{\sqrt{2\pi}} |a| e^{-a^2/2} \leq \frac{1}{\sqrt{2\pi}} \underbrace{|a e^{-a^2/4}|}_{\leq 1} e^{-a^2/4} \leq \frac{\sqrt{2}}{\sqrt{4\pi}} e^{-a^2/4}
$$
and
$$
|\phi''(a)|=\frac{1}{\sqrt{2\pi}} |a^2-1| e^{-a^2/2} \leq \frac{1}{\sqrt{2\pi}} \underbrace{|(a^2-1) e^{-a^2/4}|}_{\leq 2} e^{-a^2/4} \leq \frac{2^{3/2}}{\sqrt{4\pi}} e^{-a^2/4}.
$$
We obtain
$$
\bigg| \frac{\partial^2 \varphi_I}{\partial y_i \partial y_j}(z) \bigg| \leq 2^{3/2} \varphi_{I_{i,j}}(z), \quad z\in\R^m,
$$
where $I_{i,j}$ is the identity matrix $I$ where the $i$-th and the $j$-th diagonal element are replaced by $2$. Consequently, we have
$$
\bigg| \int_{\R^m} {\bf 1}\{z \in \frac{1}{\sqrt{s}}(A-\sqrt{1-s} y)\} \frac{\partial^2 \varphi_I}{\partial y_i \partial y_j}(z) \, \dint z \bigg| \leq 2^{3/2} \mathbb{P}(\|N_{I_{i,j}}\|\geq r_{s,y}).
$$
The Markov inequality yields
$$
\mathbb{P}(\|N_{I_{i,j}}\|\geq r_{s,y}) \leq \frac{\E \|N_{I_{i,j}}\|^{1/3}}{r_{s,y}^{1/3}} \leq \frac{s^{1/6} (\E \|N_{I_{i,j}}\|^2)^{1/6}}{d(\sqrt{1-s}y,\partial A)^{1/3}} \leq  \frac{ 2^{1/6} m^{1/6} s^{1/6} }{(1-s)^{1/6} d(y,\partial A/\sqrt{1-s})^{1/3}}.
$$
Hence, we obtain
$$
\bigg|\frac{\partial^2 f_{t,h,I}}{\partial y_i \partial y_j}(y) \bigg| \leq 2^{2/3}m^{1/6} \int_t^1 \frac{1}{s^{5/6}(1-s)^{1/6}} \frac{1}{d(y,\partial A/\sqrt{1-s})^{1/3}} \, \dint s, \quad y\in\R^m.
$$
The Cauchy-Schwarz inequality leads to
$$
\bigg(\frac{\partial^2 f_{t,h,I}}{\partial y_i \partial y_j}(y) \bigg)^2 \leq 2^{4/3} m^{1/3} \int_t^1 \frac{1}{s^{5/6} (1-s)^{1/3}} \, \dint s \int_t^1 \frac{1}{s^{5/6}} \frac{1}{d(y,\partial A/\sqrt{1-s})^{2/3}} \, \dint s, \quad y\in\R^m.
$$
Numerical integration shows that the first integral may be generously bounded by $7$ 
so that we obtain, together with Lemma \ref{lem:DistanceGaussian},
\begin{align*}
\E \bigg(\frac{\partial^2 f_{t,h,I}}{\partial y_i \partial y_j}(N_I) \bigg)^2
 & \leq 7 \cdot 2^{4/3}  m^{1/3} \int_t^1 \frac{1}{s^{5/6}} \E \frac{1}{d(N_{I},\partial A/\sqrt{1-s})^{2/3}} \, \dint s \allowdisplaybreaks\\
& \leq 7 \cdot 2^{4/3}  m^{1/3} \int_t^1 \frac{1}{s^{5/6}} \, \dint s \sup_{A'\subseteq \R^m\text{ convex}}\E \frac{1}{d(N_{I},\partial A')^{2/3}} \allowdisplaybreaks\\
&\leq 42\cdot  2^{4/3}  m^{1/3}  (1+4\sqrt{m}) \\
& \leq 530 m^{5/6},
\end{align*}
which completes the proof of Lemma \ref{lem:fNSigma}.
\end{proof}

\begin{proof}[Proof of Proposition \ref{lem:BoundD2F}]
First we prove the assertion for the special case $\Sigma=I$. For $i,j\in\{1,\hdots,m\}$ we have
\begin{align*}
& \E \bigg( \frac{\partial^2 f_{t,h,I}}{\partial y_i \partial y_j}(Y) \bigg)^2 \\
 & =\E \bigg( \frac{1}{2} \int_t^1 \frac{1}{s} \int_{\R^m} h(\sqrt{s} z + \sqrt{1-s} Y)  \frac{\partial^2 \varphi_{I} }{\partial y_i \partial y_j}(z)  \, \dint z \, \dint s \bigg)^2 \allowdisplaybreaks\\
& = \frac{1}{4} \int_t^1 \int_t^1 \frac{1}{s_1s_2} \int_{\R^m} \int_{\R^m} \E h(\sqrt{s_1}z_1+\sqrt{1-s_1}Y) h(\sqrt{s_2}z_2+\sqrt{1-s_2}Y)\\
& \hskip 5cm \times  \frac{\partial^2 \varphi_{I}}{\partial y_i \partial y_j}(z_1) \frac{\partial^2 \varphi_{I} }{\partial y_i \partial y_j}(z_2)  \, \dint z_2 \, \dint z_1 \, \dint s_2 \, \dint s_1 \allowdisplaybreaks\\
& = \frac{1}{4} \int_t^1 \int_t^1 \frac{1}{s_1s_2} \int_{\R^m} \int_{\R^m} \E h(\sqrt{s_1}z_1+\sqrt{1-s_1}N_I) h(\sqrt{s_2}z_2+\sqrt{1-s_2}N_I)\\
& \hskip 5cm \times  \frac{\partial^2 \varphi_{I}}{\partial y_i \partial y_j}(z_1) \frac{\partial^2 \varphi_{I} }{\partial y_i \partial y_j}(z_2)  \, \dint z_2 \, \dint z_1 \, \dint s_2 \, \dint s_1\\
& \quad + \frac{1}{4} \int_t^1 \int_t^1 \frac{1}{s_1s_2} \int_{\R^m} \int_{\R^m} \bigg(\E h(\sqrt{s_1}z_1+\sqrt{1-s_1}Y) h(\sqrt{s_2}z_2+\sqrt{1-s_2}Y)\\
& \hskip 5.3cm - \E h(\sqrt{s_1}z_1+\sqrt{1-s_1}N_I) h(\sqrt{s_2}z_2+\sqrt{1-s_2}N_I)\bigg) \\
& \hskip 4.8cm  \quad \times  \frac{\partial^2 \varphi_{I} }{\partial y_i \partial y_j}(z_1) \frac{\partial^2 \varphi_{I} }{\partial y_i \partial y_j}(z_2) \, \dint z_2 \, \dint z_1 \, \dint s_2 \, \dint s_1 \allowdisplaybreaks\\
& = \E\bigg( \frac{\partial^2 f_{t,h,I} }{\partial y_i \partial y_j}(N_I) \bigg)^2+ R_{ij},
\end{align*}
where $R_{ij}$ denotes the second four-fold integral in the penultimate equation. It follows from Lemma \ref{lem:fNSigma} that
$$
\E\sum_{i,j=1}^m \bigg( \frac{\partial^2 f_{t,h,I} }{\partial y_i \partial y_j}(N_I) \bigg)^2 \leq  530 m^{17/6}.
$$
For $h\in \mathcal{I}_m$ we have that
$$
h_{z_1,z_2,s_1,s_2}: \R^m\ni y\mapsto h(\sqrt{s_1}z_1+\sqrt{1-s_1}y) h(\sqrt{s_2}z_2+\sqrt{1-s_2}y)
$$
is the indicator function of a measurable convex set, whence
$$
\sum_{i,j=1}^m |R_{ij}| \leq M_2 (\log t)^2 d_{convex}(Y,N_I).
$$
Combining the previous estimates completes the proof of Proposition \ref{lem:BoundD2F} for the special case $\Sigma = I.$

For a positive definite $\Sigma\in\R^{m\times m}$ it follows from \eqref{eqn:IdentityfthSigma} that, for $y\in\R^m$,
$$
\operatorname{Hess} f_{t,h,\Sigma}(y) = \Sigma^{-1/2} \operatorname{Hess} f_{t,h\circ \Sigma^{1/2},I}(\Sigma^{-1/2}y) \Sigma^{-1/2}.
$$
Using the Hilbert-Schmidt norm $\|A\|_{H.S.}:=\sqrt{\sum_{i,j=1}^m a_{ij}^2}$ of a matrix $A=(a_{ij})_{i,j\in\{1,\hdots,m\}}\in\R^{m\times m}$ and the relation
$\|AB\|_{H.S.}\leq \|A\|_{op} \|B\|_{H.S.}$ for $A,B\in\R^{m\times m}$, we obtain
\begin{align*}
\E \sum_{i,j=1}^m \bigg( \frac{\partial^2 f_{t,h,\Sigma}}{\partial y_i \partial y_j}(Y) \bigg)^2 & = \E \| \operatorname{Hess} f_{t,h,\Sigma}(Y) \|^2_{H.S.} \\
& = \E \| \Sigma^{-1/2} \operatorname{Hess} f_{t,h\circ \Sigma^{1/2},I}(\Sigma^{-1/2} Y) \Sigma^{-1/2} \|^2_{H.S.} \allowdisplaybreaks\\
& \leq \|\Sigma^{-1/2}\|_{op}^4 \E \| \operatorname{Hess} f_{t,h\circ \Sigma^{1/2},I}(\Sigma^{-1/2} Y) \|^2_{H.S.} \\
& = \|\Sigma^{-1}\|_{op}^2 \E \sum_{i,j=1}^m \bigg( \frac{\partial^2 f_{t,h\circ\Sigma^{1/2},I}}{\partial y_i \partial y_j}(\Sigma^{-1/2}Y) \bigg)^2.
\end{align*}
Now the special case proven above (for $\Sigma=I$) and the observation that $d_{convex}(\Sigma^{-1/2}Y,N_I)=d_{convex}(Y,N_\Sigma)$ complete the proof
of Proposition \ref{lem:BoundD2F}.
\end{proof}

\section{Proofs of the main results}\label{sec:Proofs}

Throughout this section we assume that the reader is familiar with Malliavin calculus on the Poisson space. The Appendix provides the essential definitions and properties of Malliavin operators needed in the sequel.

\subsection{Proof of Theorem \ref{thm:GeneralMultivariateBound}}\label{subsec:ProofsSmooth}

The starting point for the proofs for the $d_3$- and the $d_2$-distance are the following quantitative bounds for the normal approximation of Poisson functionals, which were derived in \cite[Theorem 4.2]{PeccatiZheng} and \cite[Theorem 3.3]{PeccatiZheng} by a combination of Malliavin calculus with the interpolation method and Stein's method, respectively (see also \cite[Section 6]{BourguinPeccatiChapter}). For a definition of the inverse Ornstein-Uhlenbeck generator $L^{-1}$ we refer to \cite{LPS,PeccatiZheng} or the Appendix.

\begin{prop}\label{prop:PeccatiZheng}
Let $F=(F_1,\hdots,F_m)$, $m\in\N$, be a vector of Poisson functionals $F_1,\hdots,F_m\in\operatorname{dom}D$ with $\E F_i=0$, $i\in\{1,\hdots,m\}$, let $\Sigma=(\sigma_{ij})_{i,j\in\{1,\hdots,m\}}\in\R^{m\times m}$ be positive semi-definite, and put
\begin{align*}
\beta_1 & := \sqrt{\sum_{i,j=1}^m \E\bigg(\sigma_{ij} - \int_{\mathbb{X}} D_xF_i (-D_xL^{-1}F_j) \, \lambda(\dint x) \bigg)^2 } \allowdisplaybreaks  \\
\beta_2 & := \int_{\mathbb{X}} \E \bigg(\sum_{i=1}^m |D_xF_i| \bigg)^2 \sum_{j=1}^m |D_xL^{-1}F_j| \, \lambda(\dint x). \label{beta2}
\end{align*}
Then
\begin{align*}
d_3(F,N_{\Sigma}) & \leq \frac{m}{2} \beta_1  +\frac{1}{4} \beta_2.
\end{align*}
If, additionally, $\Sigma$ is positive definite, then
\begin{align*}
d_2(F,N_\Sigma) & \leq \|\Sigma^{-1}\|_{op} \|\Sigma\|_{op}^{1/2}\beta_1  +\frac{\sqrt{2\pi}}{8} \|\Sigma^{-1}\|_{op}^{3/2} \|\Sigma\|_{op} \beta_2.
\end{align*}
\end{prop}

The main difficulty in evaluating these bounds is to control the behavior of the terms involving $L^{-1}$, which will be done in the same way as in \cite{LPS}. The following proposition collects two estimates from \cite[Lemma 3.4 and Proposition 4.1]{LPS}, which will play a crucial role in the sequel. This proposition and Proposition \ref{prop:Mehler} are consequences of Mehler's formula (see \cite[Section 3]{LPS}).

\begin{prop}\label{prop:LPS}
\begin{itemize}
\item [(a)] For a square integrable Poisson functional $F$ and $p \geq 1$,
$$
\E |D_xL^{-1}F|^p \leq \E |D_xF|^p, \quad \lambda\text{-a.e. } x\in\mathbb{X}
$$
and
$$
\E |D^2_{x,y}L^{-1}F|^p \leq \E |D^2_{x,y}F|^p, \quad \lambda^2\text{-a.e. } (x,y)\in\mathbb{X}^2.
$$
\item [(b)] For $F,G\in\operatorname{dom}D$ with $\E F=\E G=0$,
\begin{align*}
& \E \bigg(\Cov(F,G)-\int_{\mathbb{X}} D_xF (-D_xL^{-1}G) \, \lambda(\dint x) \bigg)^2\\
 & \leq 3 \int_{\mathbb{X}^3} \big[ \E (D^2_{x_1,x_3}F)^2 (D^2_{x_2,x_3}F)^2 \big]^{1/2} \big[ \E (D_{x_1}G)^2 (D_{x_2}G)^2 \big]^{1/2} \, \lambda^3(\dint(x_1,x_2,x_3)) \\
& \quad  + \int_{\mathbb{X}^3} \big[ \E (D_{x_1}F)^2 (D_{x_2}F)^2 \big]^{1/2} \big[ \E (D^2_{x_1,x_3}G)^2 (D^2_{x_2,x_3}G)^2 \big]^{1/2} \, \lambda^3(\dint(x_1,x_2,x_3))\\
& \quad  + \int_{\mathbb{X}^3} \big[ \E (D^2_{x_1,x_3}F)^2 (D^2_{x_2,x_3}F)^2 \big]^{1/2} \big[ \E (D^2_{x_1,x_3}G)^2 (D^2_{x_2,x_3}G)^2 \big]^{1/2} \, \lambda^3(\dint (x_1,x_2,x_3)).
\end{align*}
\end{itemize}
\end{prop}

Combining Proposition \ref{prop:PeccatiZheng} and Proposition \ref{prop:LPS} yields the proof of Theorem \ref{thm:GeneralMultivariateBound}, which goes as follows.
\vskip.3cm

\noindent {\em Proof of Theorem \ref{thm:GeneralMultivariateBound}.}
From the triangle inequality we obtain
$$
\beta_1 \leq \sum_{i,j=1}^m |\sigma_{ij}-\Cov(F_i,F_j)|+\sqrt{\sum_{i,j=1}^m \E\bigg(\Cov(F_i,F_j) - \int_{\mathbb{X}} D_xF_i (-D_xL^{-1}F_j) \, \lambda(\dint x) \bigg)^2}.
$$
An application of Proposition \ref{prop:LPS}(b) yields that, for $i,j\in\{1,\hdots,m\}$,
\begin{align*}
& \E\bigg(\Cov(F_i,F_j) - \int_{\mathbb{X}} D_xF_i (-D_xL^{-1}F_j) \, \lambda(\dint x) \bigg)^2 \\
& \leq 3 \int_{\mathbb{X}^3} \big[ \E (D^2_{x_1,x_3}F_i)^2 (D^2_{x_2,x_3}F_i)^2 \big]^{1/2} \big[ \E (D_{x_1}F_j)^2 (D_{x_2}F_j)^2 \big]^{1/2} \, \lambda^3(\dint(x_1,x_2,x_3)) \\
& \quad  + \int_{\mathbb{X}^3} \big[ \E (D_{x_1}F_i)^2 (D_{x_2}F_i)^2 \big]^{1/2} \big[ \E (D^2_{x_1,x_3}F_j)^2 (D^2_{x_2,x_3}F_j)^2 \big]^{1/2} \, \lambda^3(\dint(x_1,x_2,x_3))\\
& \quad  + \int_{\mathbb{X}^3} \big[ \E (D^2_{x_1,x_3}F_i)^2 (D^2_{x_2,x_3}F_i)^2 \big]^{1/2} \big[ \E (D^2_{x_1,x_3}F_j)^2 (D^2_{x_2,x_3}F_j)^2 \big]^{1/2} \, \lambda^3(\dint(x_1,x_2,x_3))
\end{align*}
so that
\begin{equation}\label{eqn:Boundbeta1}
\beta_1 \leq \sum_{i,j=1}^m |\sigma_{ij}-\Cov(F_i,F_j)|+ 2\gamma_1+\gamma_2.
\end{equation}
It follows from H\"older's inequality and Proposition \ref{prop:LPS}(a) that
\begin{align*}
\beta_2  & \leq  m\int_{\mathbb{X}} \sum_{i=1}^m \big(\E |D_xF_i|^3 \big)^{2/3} \sum_{j=1}^m\big( \E|D_xL^{-1}F_j|^3 \big)^{1/3} \, \lambda(\dint x) \allowdisplaybreaks\\
& \leq m\int_{\mathbb{X}} \sum_{i=1}^m \big(\E|D_xF_i|^3 \big)^{2/3} \sum_{j=1}^m\big(\E |D_xF_j|^3 \big)^{1/3} \, \lambda(\dint x) \allowdisplaybreaks\\
& \leq m\int_{\mathbb{X}} m^{1/3} \bigg(\sum_{i=1}^m \E|D_xF_i|^3 \bigg)^{2/3} m^{2/3} \bigg(\sum_{j=1}^m\E |D_xF_j|^3 \bigg)^{1/3} \, \lambda(\dint x) \\
& = m^2 \int_{\mathbb{X}} \sum_{i=1}^m \E|D_xF_i|^3 \, \lambda(\dint x) = m^2 \gamma_3.
\end{align*}
Now Proposition \ref{prop:PeccatiZheng} completes the proof of Theorem \ref{thm:GeneralMultivariateBound}.
\qed

\vskip.3cm

\subsection{Proof of Theorem \ref{thm:Generaldconvex}}\label{subsec:ProofsNonSmooth}

Throughout this subsection we use several Malliavin operators, namely the already introduced difference operator $D$, the inverse Ornstein-Uhlenbeck generator $L^{-1}$, and the Skorohod integral $\delta$. Recall that we denote the domain of $D$ by $\operatorname{dom}D$ and we define  $\operatorname{dom}\delta$ similarly. For definitions we refer to the Appendix.

We prepare the proof of Theorem \ref{thm:Generaldconvex} by the following lemma.

\begin{lemm}\label{lem:DistanceBoundary}
For an $m$-dimensional random vector $Y$, a measurable convex set $A\subseteq\mathbb{R}^m$, a positive definite matrix $\Sigma\in\mathbb{R}^{m\times m}$, and $w\geq 0$,
$$
\mathbb{P}(d(Y,\partial A)\leq w) \leq 2 \sqrt{m} \|\Sigma^{-1/2}\|_{op} w + 2d_{convex}(Y,N_\Sigma).
$$
\end{lemm}

\begin{proof}
Using the abbreviations $A^w:=\{ y\in\R^m: d(y, A)\leq w \}$ and $A^{-w}:=\{ y\in A: d(y, \partial A)> w \}$, we obtain
\begin{align*}
\mathbb{P}(d(Y,\partial A)\leq w) & = \mathbb{P}(Y\in A^w) - \mathbb{P}(Y\in A^{-w})\\
& = \mathbb{P}(N_\Sigma \in A^w) - \mathbb{P}(N_\Sigma \in A^{-w}) + \mathbb{P}(Y \in A^w) - \mathbb{P}(N_\Sigma \in A^w)\\
& \quad + \mathbb{P}(N_\Sigma \in A^{-w})  - \mathbb{P}(Y \in A^{-w}).
\end{align*}
Since $A^w$ and $A^{-w}$ are measurable and convex, we have
$$
\max_{B \in \{ A^w, A^{-w} \} }   \big| \mathbb{P}(Y \in B) - \mathbb{P}(N_\Sigma \in B) \big| \leq d_{convex}(Y,N_{\Sigma})
$$
so that
\begin{align*}
\mathbb{P}(d(Y,\partial A)\leq w) & \leq \mathbb{P}(N_\Sigma\in A^w \setminus A^{-w}) + 2d_{convex}(Y,N_\Sigma) \\
& = \mathbb{P}(d(N_\Sigma,\partial A)\leq w) + 2d_{convex}(Y,N_\Sigma).
\end{align*}
Note that
\begin{align*}
d(N_I,\Sigma^{-1/2}\partial A) & = \sup_{y\in \Sigma^{-1/2}\partial A} \|N_I - y\| = \sup_{y\in \Sigma^{-1/2}\partial A} \|\Sigma^{-1/2} (\Sigma^{1/2}N_I - \Sigma^{1/2} y)\| \\
& = \sup_{y\in \partial A} \|\Sigma^{-1/2} (\Sigma^{1/2}N_I - y)\|\\
& \leq \|\Sigma^{-1/2}\|_{op} \sup_{y\in \partial A} \|\Sigma^{1/2}N_I - y\| = \|\Sigma^{-1/2}\|_{op} d(\Sigma^{1/2}N_I,\partial A).
\end{align*}
Together with Lemma \ref{lem:Goetze}, we see that
$$
\mathbb{P}(d(N_\Sigma,\partial A)\leq w) \leq \mathbb{P}( d(N_I,\Sigma^{-1/2}\partial A) \leq \|\Sigma^{-1/2}\|_{op} w) \leq 2 \sqrt{m} \|\Sigma^{-1/2}\|_{op} w,
$$
which completes the proof.
\end{proof}

The next proposition is an abstract formulation of one of the main ideas of the proof of Proposition \ref{prop:LPS}(b) (see also \cite[Lemma 21.4]{LastPenroseBook}).

\begin{prop}\label{prop:Mehler}
For a measurable function $h: \mathbb{X}^2\times\mathbf{N}\to[0,\infty)$, a Poisson functional $G\in\operatorname{dom} D$, and $p,q\in(0,\infty)$ with $1/p+1/q=1$,
\begin{align*}
& \int_{\mathbb{X}} \E \bigg( \int_{\mathbb{X}} h(x,y,\eta) |D_xL^{-1} G|  \, \lambda(\dint x) \bigg)^2 \, \lambda(\dint y) \\
& \leq \int_{\mathbb{X}^3} \big( \E h(x_1,y,\eta)^{p} h(x_2,y,\eta)^{p} \big)^{1/p} \, \big( \E |D_{x_1}G|^{q} |D_{x_2}G|^q \big)^{1/q} \, \lambda^3(\dint (x_1,x_2,y))
\end{align*}
and
\begin{align*}
& \int_{\mathbb{X}} \E \bigg( \int_{\mathbb{X}} h(x,y,\eta) |D_{x,y}^2L^{-1} G|  \, \lambda(\dint x) \bigg)^2 \, \lambda(\dint y) \\
& \leq \frac{1}{4} \int_{\mathbb{X}^3} \big( \E h(x_1,y,\eta)^{p} h(x_2,y,\eta)^{p} \big)^{1/p} \, \big( \E |D^2_{x_1,y}G|^{q} |D^2_{x_2,y}G|^q \big)^{1/q} \, \lambda^3(\dint (x_1,x_2,y)).
\end{align*}
\end{prop}

\begin{proof}
It follows from \cite[Corollary 3.3]{LPS} that
$$
|D_xL^{-1}G| \leq \int_0^1 |P_sD_xG| \, \dint s, \quad \lambda\text{-a.e. } x\in\mathbb{X}, \ \mathbb{P}\text{-a.s.,}
$$
and
$$
|D^2_{x,y}L^{-1}G| \leq \int_0^1 s \, |P_sD^2_{x,y}G| \, \dint s,  \quad \lambda^2\text{-a.e. } (x,y)\in\mathbb{X}^2, \ \mathbb{P}\text{-a.s.}
$$
For the definition of the operator $P_s$ we refer to \cite[Equation (3.1)]{LPS} or \cite[Equation (20.2)]{LastPenroseBook}. Now \cite[Lemma 21.4]{LastPenroseBook} yields the desired inequalities. Actually \cite[Lemma 21.4]{LastPenroseBook} deals only with $p=q=2$, but by using H\"older's inequality instead of the Cauchy-Schwarz inequality in the last two steps of its proof, one can extend it to $p,q\in(0,\infty)$ with $1/p+1/q=1$.
\end{proof}

\begin{proof}[Proof of Theorem \ref{thm:Generaldconvex}]
In the following five-part proof we may assume that $\gamma_1,...,\gamma_6<\infty$ since otherwise there is nothing to prove.
Throughout let $h: \R^{m}\to\R$ be the indicator function of a measurable convex set $K\subseteq \mathbb{R}^m$.

The idea of the proof goes as follows.  Put
\be \label{defgamma}
\gamma:= \|\Sigma^{-1}\|_{op} \max\bigg\{ \sum_{i,j=1}^m |\sigma_{ij} - \Cov(F_i,F_j)|, \gamma_1,\gamma_2, \gamma_3, \gamma_4,\gamma_5,\gamma_6 \bigg\}.
\ee
We first establish the bound
\be \label{maininequality}
|\E h_{t,\Sigma}(F) - \E h_{t,\Sigma}(N_{\Sigma})| \leq \bigg|J_1 -  \sum_{i,j=1}^m \sigma_{ij} \E\frac{\partial^2 f_{t,h,\Sigma}}{\partial y_i\partial y_j}(F)  \bigg| + |J_{2,1}| +|J_{2,2}|
\ee
where $J_1$, $J_{2,1}$, and $J_{2,2}$ are given below. We then show that the three terms on the right hand side of \eqref{maininequality} are each bounded by products of powers of $\gamma$ and factors such as $1/\sqrt{t}$, $|\log t| \sqrt{d_{convex}(F,N_\Sigma)}$, or $1/\sqrt{t} \cdot d_{convex}(F,N_\Sigma),$ and then choose $t$ appropriately.  Put
$$
\tilde{J}:= J_1 -  \sum_{i,j=1}^m \sigma_{ij} \E\frac{\partial^2 f_{t,h,\Sigma}}{\partial y_i\partial y_j}(F).
$$
 An intermediate step shows that the terms $|\tilde{J}|$ and $|J_{21}|$ are each bounded by a product involving  $\sqrt{\E \sum_{i,j=1}^m \big( \frac{\partial^2 f_{t,h,\Sigma}}{\partial y_i\partial y_j}(F)\big)^2}$, which, after applying Proposition \ref{lem:BoundD2F}, leads to the previously mentioned bounds.

\vskip.3cm
\noindent{\em Part (i): \  A key decomposition.} As noted in Subsection \ref{prooftech}, it follows from p.\ 1498 in \cite{PeccatiZheng} that
\begin{align*}
& |\E h_{t,\Sigma}(F) - \E h_{t,\Sigma}(N_{\Sigma})|\\
& = \bigg| \sum_{i,j=1}^m \sigma_{ij} \E\frac{\partial^2 f_{t,h,\Sigma}}{\partial y_i\partial y_j}(F) - \sum_{k=1}^m \E \int_{\mathbb{X}} D_x \frac{\partial f_{t,h,\Sigma}}{\partial y_k}(F) (-D_xL^{-1}F_k) \, \lambda(\dint x) \bigg|.
\end{align*}
The fundamental theorem of calculus yields
\begin{align*}
& \sum_{k=1}^m \E \int_{\mathbb{X}} D_x \frac{\partial f_{t,h,\Sigma}}{\partial y_k}(F) (-D_xL^{-1}F_k) \, \lambda(\dint x) \\
& =  \sum_{k=1}^m \E \int_{\mathbb{X}} \int_0^1 \sum_{j=1}^m \frac{\partial^2 f_{t,h,\Sigma}}{\partial y_j \partial y_k}(F+u D_xF) D_xF_j (-D_xL^{-1}F_k) \, \dint u\, \lambda(\dint x)  \allowdisplaybreaks\\
& = \sum_{j,k=1}^m  \E \int_{\mathbb{X}} \frac{\partial^2 f_{t,h,\Sigma}}{\partial y_j \partial y_k}(F) D_xF_j (-D_xL^{-1}F_k) \, \lambda(\dint x)  \\
& \quad + \sum_{j,k=1}^m  \E \int_{\mathbb{X}} \int_0^1 \bigg(\frac{\partial^2 f_{t,h,\Sigma}}{\partial y_j \partial y_k}(F+u D_xF)-\frac{\partial^2 f_{t,h,\Sigma}}{\partial y_j \partial y_k}(F)\bigg) D_xF_j (-D_xL^{-1}F_k) \, \dint u\, \lambda(\dint x)\\
& =: J_1 + J_2.
\end{align*}
Further applications of the fundamental theorem of calculus lead to
\begin{align*}
J_2 & = \sum_{j,k=1}^m \E \int_{\mathbb{X}} \int_0^1 \bigg( \frac{\partial^2 f_{t,h,\Sigma} }{\partial y_j \partial y_k}(F+u D_xF) - \frac{\partial^2 f_{t,h,\Sigma} }{\partial y_j \partial y_k}(F) \bigg) D_xF_j (-D_xL^{-1}F_k) \, \dint u \, \lambda(\dint x) \allowdisplaybreaks\\
& = \sum_{i,j,k=1}^m \E \int_{\mathbb{X}} \int_0^1 \int_0^1 \frac{\partial^3 f_{t,h,\Sigma} }{\partial y_i \partial y_j \partial y_k}(F+v u D_xF) u D_xF_i D_xF_j (-D_xL^{-1}F_k) \, \dint v \, \dint u \, \lambda(\dint x) \allowdisplaybreaks\\
& = \sum_{i,j,k=1}^m \E \int_{\mathbb{X}} \int_0^1 \int_0^1 \frac{\partial^3 f_{t,h,\Sigma} }{\partial y_i \partial y_j \partial y_k}(F+v D_xF) u D_xF_i D_xF_j (-D_xL^{-1}F_k) \, \dint v \, \dint u \, \lambda(\dint x)\\
& \quad + \sum_{i,j,k=1}^m \E \int_{\mathbb{X}} \int_0^1 \int_0^1 \bigg( \frac{\partial^3 f_{t,h,\Sigma} }{\partial y_i \partial y_j \partial y_k}(F+v u D_xF) - \frac{\partial^3 f_{t,h,\Sigma} }{\partial y_i \partial y_j \partial y_k}(F+v D_xF)\bigg)\\
& \hskip 7cm \times u D_xF_i D_xF_j (-D_xL^{-1}F_k) \, \dint v \, \dint u \, \lambda(\dint x) \allowdisplaybreaks \\
& = \frac{1}{2} \sum_{j,k=1}^m \E \int_{\mathbb{X}} \bigg(\frac{\partial^2 f_{t,h,\Sigma} }{ \partial y_j \partial y_k}(F+D_xF) -\frac{\partial^2 f_{t,h,\Sigma} }{ \partial y_j \partial y_k}(F) \bigg)  D_xF_j (-D_xL^{-1}F_k)  \, \lambda(\dint x)\\
& \quad + \sum_{i,j,k=1}^m \E \int_{\mathbb{X}} \int_0^1 \int_0^1 \bigg( \frac{\partial^3 f_{t,h,\Sigma} }{\partial y_i \partial y_j \partial y_k}(F+v u D_xF) - \frac{\partial^3 f_{t,h,\Sigma} }{\partial y_i \partial y_j \partial y_k}(F+v D_xF)\bigg)\\
& \hskip 7cm  \times u D_xF_i D_xF_j (-D_xL^{-1}F_k) \, \dint v \, \dint u \, \lambda(\dint x)\\
& =: J_{2,1}+J_{2,2},
\end{align*}
which gives \eqref{maininequality}.

\vskip.3cm

\noindent {\em Part (ii): \  A bound for $\tilde{J}$.}
Recalling the definition of $\beta_1$ in Proposition \ref{prop:PeccatiZheng} and applying the
Cauchy-Schwarz inequality, we obtain
\begin{equation}\label{eqn:FirstBoundJ1}
\begin{split}
|\tilde{J}| & \leq  \sqrt{\E \sum_{i,j=1}^m \bigg(\sigma_{ij}- \int_{\mathbb{X}} D_xF_j (-D_xL^{-1}F_k) \, \lambda(\dint x) \bigg)^2} \sqrt{\E \sum_{i,j=1}^m \bigg( \frac{\partial^2 f_{t,h,\Sigma}}{\partial y_i\partial y_j}(F)\bigg)^2} \\
& = \beta_{1} \sqrt{\E \sum_{i,j=1}^m \bigg( \frac{\partial^2 f_{t,h,\Sigma}}{\partial y_i\partial y_j}(F)\bigg)^2}.
\end{split}
\end{equation}
Now Proposition \ref{lem:BoundD2F} leads to
\begin{equation}\label{eqn:BoundD2fF}
\sqrt{\E \sum_{i,j=1}^m \bigg( \frac{\partial^2 f_{t,h,\Sigma}}{\partial y_i\partial y_j}(F)\bigg)^2} \leq \| \Sigma^{-1} \|_{op} \big(\sqrt{M_2} |\log t| \sqrt{d_{convex}(F,N_\Sigma)} + 24 m^{17/12}\big).
\end{equation}
Combining inequalities \eqref{eqn:Boundbeta1}, \eqref{eqn:FirstBoundJ1}, and \eqref{eqn:BoundD2fF} yields
\begin{equation}\label{eqn:BoundJ1}
\begin{split}
|\tilde{J}|  \leq &  \|\Sigma^{-1}\|_{op} \big(\sqrt{M_2} |\log t| \sqrt{d_{convex}(F,N_\Sigma)} + 24 m^{17/12} \big)\\
& \times \bigg( \sum_{i,j=1}^m |\sigma_{ij}-\Cov(F_i,F_j)|+2\gamma_1+\gamma_2 \bigg).
\end{split}
\end{equation}

\vskip.3cm
\noindent{ \em Part (iii): \  A bound for $J_{2,1}$.}
We start by rewriting $J_{2,1}$ as
$$
J_{2,1} = \frac{1}{2} \sum_{j,k=1}^m \E \int_{\mathbb{X}} D_x\frac{\partial^2 f_{t,h,\Sigma} }{ \partial y_j \partial y_k}(F) D_xF_j (-D_xL^{-1}F_k)  \, \lambda(\dint x).
$$
All third partial derivatives of $f_{t,h,\Sigma}$ are bounded by some constant (recall \eqref{eqn:BoundPartial3}), and thus
$$
\frac{\partial^2 f_{t,h,\Sigma} }{ \partial y_j \partial y_k}(F) \in\operatorname{dom} D, \quad j,k\in\{1,\hdots,m\}.
$$
From Lemma \ref{lem:BoundDelta} and the computation for $\E\delta( DF_j (-DL^{-1}F_k))^2$ below, one deduces that $DF_j (-DL^{-1}F_k)\in\operatorname{dom}\delta$.
It follows from integration by parts (see Lemma \ref{lem:PartialIntegration}) and the Cauchy-Schwarz inequality that
\begin{align*}
2 |J_{2,1}|
& =  \bigg| \sum_{j,k=1}^m \E  \frac{\partial^2 f_{t,h,\Sigma} }{ \partial y_j \partial y_k}(F) \delta( DF_j (-DL^{-1}F_k)) \, \lambda(\dint x) \bigg| \allowdisplaybreaks \\
& \leq  \bigg(\E \sum_{j,k=1}^m  \bigg(\frac{\partial^2 f_{t,h,\Sigma} }{ \partial y_j \partial y_k}(F)\bigg)^2\bigg)^{1/2} \bigg( \sum_{j,k=1}^m \E\delta( DF_j (-DL^{-1}F_k))^2\bigg)^{1/2}.
\end{align*}
By Proposition \ref{lem:BoundD2F} the first factor is bounded by
$$
\bigg(\E \sum_{j,k=1}^m  \bigg(\frac{\partial^2 f_{t,h,\Sigma} }{ \partial y_j \partial y_k}(F)\bigg)^2\bigg)^{1/2} \leq \|\Sigma^{-1}\|_{op} \big( \sqrt{M_2} |\log t| \sqrt{d_{convex}(F,N_\Sigma)} + 24 m^{17/12} \big).
$$
For the summands in the second factor it follows from Lemma \ref{lem:BoundDelta}  that
\begin{align*}
& \E\delta( DF_j (-DL^{-1}F_k))^2 \\
& \leq \int_{\mathbb{X}} \E (D_xF_j)^2 (-D_xL^{-1}F_k)^2 \, \lambda(\dint x) + \int_{\mathbb{X}^2} \E\big(D_y(D_xF_j (-D_xL^{-1}F_k)) \big)^2 \, \lambda^2(\dint (x,y)) \allowdisplaybreaks\\
& \leq \frac{1}{2} \int_{\mathbb{X}} \E (D_xF_j)^4 + \E (-D_xL^{-1}F_k)^4 \, \lambda(\dint x) \\
& \quad  + 3 \int_{\mathbb{X}^2} \E (D^2_{x,y}F_j)^2 (-D_xL^{-1}F_k)^2 + \E (D_xF_j)^2 (-D^2_{x,y}L^{-1}F_k)^2\\
& \hskip 2cm + \E (D^2_{x,y}F_j)^2 (-D^2_{x,y}L^{-1}F_k)^2  \, \lambda^2(\dint (x,y)),
\end{align*}
where we used the arithmetic geometric mean inequality $a_1 a_2 \leq \frac{1}{2}(a_1^2 + a_2^2)$ for $a_1, a_2 \in (0, \infty)$ as well as Lemma \ref{lem:DProduct} and Jensen's inequality. It follows from Proposition \ref{prop:LPS}(a) and the Cauchy-Schwarz inequality that
\begin{align*}
& \E\delta( DF_j (-DL^{-1}F_k))^2 \\
& \leq \frac{1}{2} \int_{\mathbb{X}} \E (D_xF_j)^4 + \E (D_xF_k)^4 \, \lambda(\dint x) \\
& \quad  + 3 \int_{\mathbb{X}^2} \big(\E (D^2_{x,y}F_j)^4\big)^{1/2}  \big(\E (D_xF_k)^4\big)^{1/2} + \big(\E (D_xF_j)^4\big)^{1/2} \big(\E (D^2_{x,y}F_k)^4\big)^{1/2}\\
& \hskip 2cm + \big(\E (D^2_{x,y}F_j)^4\big)^{1/2} \big(\E (D^2_{x,y}F_k)^4 \big)^{1/2}  \, \lambda^2(\dint (x,y)).
\end{align*}
Since $\gamma_4<\infty$, the right-hand side is finite, which implies that assumptions \eqref{eqn:LPlambda} and \eqref{eqn:AssumptionDg} are satisfied and, thus, justifies the previous applications of Lemma \ref{lem:PartialIntegration} and Lemma \ref{lem:BoundDelta}. Finally, combining the previous estimates yields
\begin{equation}\label{eqn:BoundJ21}
|J_{2,1}| \leq \frac{1}{2} \|\Sigma^{-1}\|_{op} (\sqrt{M_2} |\log t| \sqrt{d_{convex}(F,N_\Sigma)} + 24 m^{17/12} )  \gamma_4.
\end{equation}

\vskip.3cm
\noindent{\em Part (iv): \  A bound for $J_{2,2}$.} The bound for $|J_{2,2}|$ is more involved and goes as follows.  First, note that the triangle inequality and \eqref{eqn:D3ht} imply that
\begin{align*}
|J_{2,2}| & \leq \sum_{i,j,k=1}^m \E \int_{\mathbb{X}} \int_0^1 \int_0^1 \bigg| \frac{\partial^3 f_{t,h,\Sigma} }{\partial y_i \partial y_j \partial y_k}(F+v u D_xF) - \frac{\partial^3 f_{t,h,\Sigma} }{\partial y_i \partial y_j \partial y_k}(F+v D_xF)\bigg|\\
& \hskip 7cm \times u |D_xF_i \, D_xF_j \, D_xL^{-1}F_k| \, \dint v \, \dint u \, \lambda(\dint x) \allowdisplaybreaks\\
& \leq \sum_{i,j,k=1}^m \E \int_{\mathbb{X}} \int_0^1 \int_0^1 \frac{1}{2} \int_t^1 \int_{\R^m} \frac{\sqrt{1-s}}{s^{3/2}}\\
& \hskip 3cm \times \big|h(\sqrt{s}z+\sqrt{1-s} (F+v u D_xF)) - h(\sqrt{s}z+\sqrt{1-s}(F+v D_xF))\big|\\
& \hskip 3cm \times \bigg|\frac{\partial^3 \varphi_{\Sigma}}{\partial y_i \partial y_j \partial y_k}(z) \bigg| u |D_xF_i \, D_xF_j \, D_xL^{-1}F_k| \, \dint z \, \dint s\, \dint v \, \dint u \, \lambda(\dint x).
\end{align*}
Using the abbreviation
\begin{align*}
U_{ijk} &:= \sup_{\substack{z\in\R^m,\\ s,u\in[0,1]}} \E \int_{\mathbb{X}} \int_0^1 \big|h(\sqrt{s}z+\sqrt{1-s} (F+v u D_xF)) - h(\sqrt{s}z+\sqrt{1-s}(F+v D_xF))\big| \\
& \hskip 5cm \times |D_xF_i \, D_xF_j \, D_xL^{-1}F_k| \, \dint v  \, \lambda(\dint x)
\end{align*}
for $i,j,k\in\{1,\hdots,m\}$ and the Cauchy-Schwarz inequality, we obtain
\begin{align*}
|J_{2,2}| & \leq \frac{1}{2\sqrt{t}}\sum_{i,j,k=1}^m \int_{\R^m}  \bigg|\frac{\partial^3 \varphi_{\Sigma}}{\partial y_i \partial y_j \partial y_k}(z) \bigg| \, \dint z \,  U_{ijk} \\
& \leq \frac{1}{2\sqrt{t}} \int_{\R^m} \bigg( \sum_{i,j,k=1}^m \bigg(\frac{\partial^3 \varphi_{\Sigma}}{\partial y_i \partial y_j \partial y_k}(z) \bigg)^2 \bigg)^{1/2} \, \dint z \ \bigg( \sum_{i,j,k=1}^m U_{ijk}^2 \bigg)^{1/2}.
\end{align*}
By \eqref{eqn:BoundD3Sigma} and substitution the first integral satisfies the bound
\begin{align*}
& \int_{\R^m} \bigg( \sum_{i,j,k=1}^m \bigg(\frac{\partial^3 \varphi_{\Sigma}}{\partial y_i \partial y_j \partial y_k}(z) \bigg)^2 \bigg)^{1/2} \, \dint z \\
& \leq \frac{\|\Sigma^{-1}\|_{op}^{3/2}}{\sqrt{\det(\Sigma)}}\int_{\R^m} \bigg( \sum_{i,j,k=1}^m \bigg(\frac{\partial^3 \varphi_{I}}{\partial y_i \partial y_j \partial y_k}(\Sigma^{-1/2}z) \bigg)^2 \bigg)^{1/2} \, \dint z = M_3 \|\Sigma^{-1}\|_{op}^{3/2}
\end{align*}
with
$$
M_3 := \int_{\R^m} \bigg( \sum_{i,j,k=1}^m \bigg(\frac{\partial^3 \varphi_{I}}{\partial y_i \partial y_j \partial y_k}(z) \bigg)^2 \bigg)^{1/2} \, \dint z
$$
so that
\begin{equation}\label{eqn:BoundJ22M3}
|J_{2,2}| \leq  M_3 \|\Sigma^{-1}\|_{op}^{3/2} \frac{1}{2\sqrt{t}}  \bigg( \sum_{i,j,k=1}^m U_{ijk}^2 \bigg)^{1/2}.
\end{equation}
The Cauchy-Schwarz inequality yields that
\begin{align*}
M_3 &  = \int_{\R^m} \bigg( \sum_{i,j,k=1}^m \bigg(\frac{\partial^3 \varphi_{I}}{\partial y_i \partial y_j \partial y_k}(z) \frac{1}{\varphi_I(z)} \bigg)^2 \bigg)^{1/2} \varphi_I(z) \, \dint z \\
& \leq \bigg( \sum_{i,j,k=1}^m \int_{\R^m} \bigg(\frac{\partial^3 \varphi_{I}}{\partial y_i \partial y_j \partial y_k}(z) \frac{1}{\varphi_I(z)} \bigg)^2 \varphi_I(z) \, \dint z  \bigg)^{1/2}.
\end{align*}
Together with the observation that, for a standard univariate Gaussian random variable $N$ with density $\phi$,
\begin{align*}
\E[ (\phi'(N)/\phi(N))^2 ] & = \E[N^2]=1 \allowdisplaybreaks\\
\E[(\phi''(N)/\phi(N))^2] & =\E[(N^2-1)^2]=\E[N^4-2N^2+1]=2 \allowdisplaybreaks\\
\E[(\phi'''(N)/\phi(N))^2] & = \E[(N^3-3N)^2] = \E[N^6-6 N^4 + 9N^2] = 6
\end{align*}
this implies that
\begin{equation}\label{eqn:BoundM3}
M_3 \leq \sqrt{6} m^{3/2}.
\end{equation}

Next we bound $U_{ijk}$ for fixed $i,j,k\in\{1,\hdots,m\}$. We define $r(D_xF):=\frac{1}{\|D_xF\|} D_xF$. Using the substitution $w=v \|D_xF\|$ for the first term, we obtain
\begin{align*}
U_{ijk} & \leq	 \sup_{\substack{z\in\R^m,\\ s,u\in[0,1]}} \E \int_{\mathbb{X}} \int_0^{\|D_xF\|} \big| h(\sqrt{s} z+ \sqrt{1-s}(F+uwr(D_xF))) \\
& \hskip 5cm- h(\sqrt{s} z+ \sqrt{1-s}(F+wr(D_xF))) \big| \\
& \hskip 4cm \times \mathbf{1}\{\|D_xF\|\leq 1\} \frac{|D_xF_i|}{\|D_xF\|} |D_xF_j| \, |D_xL^{-1}F_k|  \, \dint w \, \lambda(\dint x) \\
& \quad + \sup_{\substack{z\in\R^m,\\ s,u\in[0,1]}} \E \int_{\mathbb{X}} \int_0^1 \big| h(\sqrt{s} z+ \sqrt{1-s}(F+uv D_xF )) \\
& \hskip 5cm- h(\sqrt{s} z+ \sqrt{1-s}(F+v D_xF )) \big| \\
& \hskip 3.5cm \times \mathbf{1}\{\|D_xF\|\geq 1\} |D_xF_i| \, \, |D_xF_j| \, |D_xL^{-1}F_k|  \, \dint v \, \lambda(\dint x) \\
& =: U_{ijk}^{(1)} + U_{ijk}^{(2)}.
\end{align*}
Recall that $h(\cdot)=\mathbf{1}\{\cdot\in K\}$ for a measurable convex set $K\subseteq\R^m$. We have that
\begin{align*}
U_{ijk}^{(2)} & \leq \E \int_{\mathbb{X}} \mathbf{1}\{\|D_xF\|\geq 1\} |D_xF_i| \, \, |D_xF_j| \, |D_xL^{-1}F_k|  \, \lambda(\dint x) \\
& \leq \E \int_{\mathbb{X}} \|D_xF\| \, |D_xF_i| \, \, |D_xF_j| \, |D_xL^{-1}F_k|  \, \lambda(\dint x) \allowdisplaybreaks\\
& \leq \sum_{\ell=1}^m \E \int_{\mathbb{X}} |D_xF_\ell| \, |D_xF_i| \, \, |D_xF_j| \, |D_xL^{-1}F_k|  \, \lambda(\dint x) \\
& \leq \frac{1}{4} \bigg( \sum_{\ell=1}^m \int_{\mathbb{X}} \E (D_xF_\ell)^4 \, \lambda(\dint x) + m \int_{\mathbb{X}} \E (D_xF_i)^4 + \E (D_xF_j)^4 + \E (D_xF_k)^4\, \lambda(\dint x) \bigg),
\end{align*}
where we used the arithmetic geometric mean inequality and Proposition \ref{prop:LPS}(a) in the last step. This implies
\begin{equation}\label{eqn:SumUijk2}
\sqrt{ \sum_{i,j,k=1}^m (U_{ijk}^{(2)})^2 } \leq \sum_{i,j,k=1}^m U_{ijk}^{(2)} \leq m^2 \gamma_4^2.
\end{equation}

Next we bound $\sum_{i,j,k=1}^m (U_{ijk}^{(1)})^2$.  We shall do this with the aid of Proposition \ref{prop:Mehler} and the Poincar\'e inequality. By way of preparation, define $K_{s,z}:=\frac{1}{\sqrt{1-s}} (K-\sqrt{s}z)$. Then
\begin{align*}
& \big| h(\sqrt{s} z+ \sqrt{1-s}(F+uwr(D_xF))) - h(\sqrt{s} z+ \sqrt{1-s}(F+wr(D_xF))) \big| \\
& = \big| \mathbf{1}\{ F+uwr(D_xF) \in K_{s,z}\}  - \mathbf{1}\{ F+wr(D_xF) \in K_{s,z} \} \big| \\
& \leq \mathbf{1}\{ d(F,\partial K_{s,z}) \leq w \}.
\end{align*}
Thus, we have that
\begin{align*}
U^{(1)}_{ijk} & \leq \sup_{z\in\R^m, s\in[0,1]} \E \int_{\mathbb{X}} \int_0^1 \mathbf{1}\{d(F,\partial K_{s,z})\leq w\} \mathbf{1}\{w\leq \|D_xF\|\} |D_xF_j \, D_xL^{-1}F_k| \, \dint w \, \lambda(\dint x) \\
& \leq \sup_{z\in\R^m, s\in[0,1]} \int_0^1 \mathbb{P}(d(F,\partial K_{s,z})\leq w) \E \int_{\mathbb{X}} \mathbf{1}\{w\leq \|D_xF\|\} |D_xF_j \, D_xL^{-1}F_k| \, \lambda(\dint x) \, \dint w \\
& \quad + \sup_{z\in\R^m, s\in[0,1]} \int_0^1 \E \mathbf{1}\{d(F,\partial K_{s,z})\leq w\}\\
& \hskip 3.5cm \times \bigg| \int_{\mathbb{X}} \mathbf{1}\{w\leq \|D_xF\|\} |D_xF_j \, D_xL^{-1}F_k|  \, \lambda(\dint x) \\
& \hskip 4cm - \E\int_{\mathbb{X}} \mathbf{1}\{w\leq \|D_xF\|\} |D_xF_j \, D_xL^{-1}F_k|  \, \lambda(\dint x)\bigg|  \, \dint w \\
& =: R^{(1)}_{jk} + R^{(2)}_{jk}.
\end{align*}
Now Lemma \ref{lem:DistanceBoundary}, the arithmetic geometric mean inequality, and Proposition \ref{prop:LPS}(a) imply that
\begin{align*}
R^{(1)}_{jk} & \leq \int_0^1 \bigg(  2\sqrt{m} \|\Sigma^{-1/2}\|_{op} w + 2d_{convex}(F,N_\Sigma) \bigg)\\
& \hskip 1.25cm \times \E \int_{\mathbb{X}} \mathbf{1}\{w\leq \|D_xF\|\} |D_xF_j \, D_xL^{-1}F_k|  \, \lambda(\dint x)  \, \dint w\\
& \leq 2\sqrt{m} \|\Sigma^{-1/2}\|_{op}  \int_{\mathbb{X}} \E \int_0^1 w \mathbf{1}\{w\leq \|D_xF\|\} \, \dint w \, |D_xF_j \, D_xL^{-1}F_k| \, \lambda(\dint x) \\
& \quad + 2 d_{convex}(F,N_\Sigma) \int_{\mathbb{X}} \E \int_0^1 \mathbf{1}\{w\leq \|D_xF\|\} \, \dint w \, |D_xF_j \, D_xL^{-1}F_k| \, \lambda(\dint x) \allowdisplaybreaks\\
& \leq \sqrt{m} \|\Sigma^{-1/2}\|_{op}  \int_{\mathbb{X}} \E \|D_xF\|^2 \, |D_xF_j \, D_xL^{-1}F_k| \, \lambda(\dint x) \\
& \quad + 2d_{convex}(F,N_\Sigma) \int_{\mathbb{X}} \E \|D_xF\| \, |D_xF_j \, D_xL^{-1}F_k|  \, \lambda(\dint x)\\
& \leq \sqrt{m} \|\Sigma^{-1/2}\|_{op} \bigg(\frac{1}{2} \sum_{\ell=1}^m \int_{\mathbb{X}} \E (D_xF_\ell)^4 \, \lambda(\dint x) + \frac{m}{4} \int_{\mathbb{X}} \E (D_xF_j)^4 + \E (D_xF_k)^4 \, \lambda(\dint x)\bigg) \\
& \quad +2d_{convex}(F,N_\Sigma) \frac{1}{3} \bigg( \sum_{\ell=1}^m \int_{\mathbb{X}} \E |D_xF_\ell|^3 \, \lambda(\dint x) + m \int_{\mathbb{X}} \E |D_xF_j|^3 + \E |D_xF_k|^3 \, \lambda(\dint x) \bigg).
\end{align*}
Consequently, we have that
\begin{equation}\label{eqn:SumRjk1}
\sqrt{m\sum_{j,k=1}^m (R_{jk}^{(1)})^2} \leq \sqrt{m} \sum_{j,k=1}^m R_{jk}^{(1)} \leq  m^2 \|\Sigma^{-1/2}\|_{op} \gamma_4^2 + 2 d_{convex}(F,N_\Sigma) m^{5/2} \gamma_3.
\end{equation}

For $R^{(2)}_{jk}$ we obtain by the Cauchy-Schwarz inequality and Lemma \ref{lem:DistanceBoundary} that
\begin{align*}
R^{(2)}_{jk} & \leq \sup_{z\in\R^m, s\in[0,1]} \int_0^1 \mathbb{P}(d(F,\partial K_{s,z})\leq w)^{1/2}\\
& \hskip 2.8cm \times \bigg(\E\bigg| \int_{\mathbb{X}} \mathbf{1}\{w\leq \|D_xF\|\} |D_xF_j \, D_xL^{-1}F_k|  \, \lambda(\dint x) \\
& \hskip 3.1cm \quad - \E\int_{\mathbb{X}} \mathbf{1}\{w\leq \|D_xF\|\} |D_xF_j \, D_xL^{-1}F_k|  \, \lambda(\dint x)\bigg|^2\bigg)^{1/2}  \, \dint w \allowdisplaybreaks\\
& \leq \int_0^1 \bigg( 2d_{convex}(F,N_\Sigma) + 2\sqrt{m} \|\Sigma^{-1/2}\|_{op} w \bigg)^{1/2}\\
& \hskip 2.5cm \times \bigg(\Var \bigg( \int_{\mathbb{X}} \mathbf{1}\{w\leq \|D_xF\|\} |D_xF_j \, D_xL^{-1}F_k|  \, \lambda(\dint x) \bigg)\bigg)^{1/2} \, \dint w \allowdisplaybreaks\\
& \leq \bigg( 2 d_{convex}(F,N_\Sigma) V_{jk}^{(1)} + 2\sqrt{m} \|\Sigma^{-1/2}\|_{op} V_{jk}^{(2)} \bigg)^{1/2}
\end{align*}
with
\begin{align*}
V_{jk}^{(1)} & := \int_0^1 \Var\bigg( \int_{\mathbb{X}} \mathbf{1}\{ w\leq \|D_xF\|\} |D_xF_j \, D_xL^{-1}F_k| \, \lambda(\dint x) \bigg) \, \dint w, \\
V_{jk}^{(2)} & := \int_0^1 w \Var\bigg( \int_{\mathbb{X}} \mathbf{1}\{ w\leq \|D_xF\|\} |D_xF_j \, D_xL^{-1}F_k| \, \lambda(\dint x) \bigg) \, \dint w.
\end{align*}
The existence of the variances in the definitions of $V_{jk}^{(1)}$ and $V_{jk}^{(2)}$ will be discussed below.

To  further bound $V_{jk}^{(1)}$ and $V_{jk}^{(2)}$ we will apply the Poincar\'e inequality (see Theorem \ref{thm:Poincare}). We prepare this by computing difference operators. We have that
\begin{align*}
|D_y\mathbf{1}\{w\leq\|D_xF\|\}| & = |\mathbf{1}\{w\leq\|D_xF\|+D_y\|D_xF\|\}-\mathbf{1}\{w\leq\|D_xF\|\}|\\
& \leq \mathbf{1}\{w\leq \|D_xF\|+\big|D_y\|D_xF\|\big|\}
\end{align*}
and
$$
\big|D_y\|D_xF\|\big| = \big| \|D_xF+D^2_{x,y}F\|-\|D_xF\| \big| \leq \|D^2_{x,y}F\|,
$$
whence
$$
|D_y\mathbf{1}\{w\leq\|D_xF\|\}| \leq \mathbf{1}\{w\leq \|D_xF\|+\|D^2_{x,y}F\|\} \mathbf{1}\{D^2_{x,y}F\neq \0\}.
$$
Together with Lemma \ref{lem:DProduct}, we obtain
\begin{align*}
& \bigg( D_y \bigg( \int_{\mathbb{X}} \mathbf{1}\{w\leq \|D_xF\|\} |D_xF_j \, D_xL^{-1}F_k| \, \lambda(\dint x) \bigg) \bigg)^2 \\
& \leq \bigg( \int_{\mathbb{X}} \mathbf{1}\{w\leq \|D_xF\|+\|D^2_{x,y}F\|\} \mathbf{1}\{D^2_{x,y}F\neq \0\}
\big( |D_xF_j \, D_xL^{-1}F_k| + \big|D_y |D_xF_j \, D_xL^{-1}F_k| \big| \big)  \\
& \quad \quad \quad \quad + \mathbf{1}\{w\leq \|D_xF\|\} \big|D_y |D_xF_j \, D_xL^{-1}F_k| \big| \, \lambda(\dint x)  \bigg)^2 \allowdisplaybreaks\\
& \leq 3 \int_{\mathbb{X}^2} \mathbf{1}\{w\leq \min\{\|D_{x_1}F\|+\|D^2_{x_1,y}F\|, \|D_{x_2}F\|+\|D^2_{x_2,y}F\| \}\} \mathbf{1}\{D^2_{x_1,y}F\neq \0, D^2_{x_2,y}F\neq \0\} \\
& \hskip 1.5cm \times \big( |D_{x_1}F_j \, D_{x_1}L^{-1}F_k| \, |D_{x_2}F_j \, D_{x_2}L^{-1}F_k|\\
& \hskip 2cm + \big|D_y|D_{x_1}F_j \, D_{x_1}L^{-1}F_k| \big| \, \big|D_y|D_{x_2}F_j \, D_{x_2}L^{-1}F_k|\big| \big) \, \lambda^2(\dint(x_1,x_2))\\
& \quad + 3 \int_{\mathbb{X}^2}  \mathbf{1}\{ w\leq \min\{\|D_{x_1}F\|,\|D_{x_2}F\|\} \} \big|D_y|D_{x_1}F_j \, D_{x_1}L^{-1}F_k|\big| \\
& \hskip 3cm \times \big|D_y|D_{x_2}F_j \, D_{x_2}L^{-1}F_k|\big| \, \lambda^2(\dint(x_1,x_2)).
\end{align*}
Now it follows from the Poincar\'e inequality that
\begin{align*}
V_{jk}^{(1)} & \leq 3 \int_{\mathbb{X}^3} \E \mathbf{1}\{D^2_{x_1,y}F\neq \0, D^2_{x_2,y}F\neq \0\} \min\{\|D_{x_1}F\|+\|D^2_{x_1,y}F\|, \|D_{x_2}F\|+\|D^2_{x_2,y}F\| \} \\
& \hskip 1.5cm \times \big( |D_{x_1}F_j \, D_{x_1}L^{-1}F_k| \, |D_{x_2}F_j \, D_{x_2}L^{-1}F_k|\\
& \hskip 2cm + \big|D_y|D_{x_1}F_j \, D_{x_1}L^{-1}F_k|\big| \, \big|D_y|D_{x_2}F_j \, D_{x_2}L^{-1}F_k|\big| \big) \, \lambda^3(\dint(x_1,x_2,y))\\
& \quad + 3 \int_{\mathbb{X}^3} \E \min\{\|D_{x_1}F\|,\|D_{x_2}F\|\} \big|D_y|D_{x_1}F_j \, D_{x_1}L^{-1}F_k|\big|  \\
& \hskip 3cm \times \big|D_y|D_{x_2}F_j \, D_{x_2}L^{-1}F_k|\big| \, \lambda^3(\dint(x_1,x_2,y))
\end{align*}
and
\begin{align*}
V_{jk}^{(2)} & \leq 3 \int_{\mathbb{X}^3} \E \mathbf{1}\{D^2_{x_1,y}F\neq \0, D^2_{x_2,y}F\neq \0\} \min\{\|D_{x_1}F\|^2+\|D^2_{x_1,y}F\|^2, \|D_{x_2}F\|^2+\|D^2_{x_2,y}F\|^2 \} \\
& \hskip 1.5cm \times \big( |D_{x_1}F_j \, D_{x_1}L^{-1}F_k| \, |D_{x_2}F_j \, D_{x_2}L^{-1}F_k|\\
& \hskip 2cm + \big|D_y|D_{x_1}F_j \, D_{x_1}L^{-1}F_k|\big| \, \big|D_y|D_{x_2}F_j \, D_{x_2}L^{-1}F_k|\big| \big) \, \lambda^3(\dint(x_1,x_2,y))\\
& \quad + \frac{3}{2} \int_{\mathbb{X}^3} \E \min\{\|D_{x_1}F\|^2,\|D_{x_2}F\|^2\} \big|D_y|D_{x_1}F_j \, D_{x_1}L^{-1}F_k|\big| \\
& \hskip 3cm \times \big|D_y|D_{x_2}F_j \, D_{x_2}L^{-1}F_k|\big| \, \lambda^3(\dint(x_1,x_2,y)).
\end{align*}
This implies that
\begin{align*}
V_{jk}^{(\ell)} & \leq 3 \int_{\mathbb{X}} \E \bigg( \int_{\mathbb{X}} \mathbf{1}\{D^2_{x,y}F\neq \0\} \sqrt{\|D_{x}F\|^\ell+\|D^2_{x,y}F\|^\ell} \, |D_{x}F_j \, D_{x}L^{-1}F_k| \, \lambda(\dint x)\bigg)^2\\
& \quad \quad \quad + \E \bigg( \int_{\mathbb{X}} \mathbf{1}\{D^2_{x,y}F\neq \0\} \sqrt{\|D_{x}F\|^\ell+\|D^2_{x,y}F\|^\ell} \, \big|D_y|D_{x}F_j \, D_{x}L^{-1}F_k|\big| \, \lambda(\dint x)\bigg)^2 \, \lambda(\dint y)\\
& \quad + \frac{3}{\ell}  \int_{\mathbb{X}} \E \bigg(\int_{\mathbb{X}} \sqrt{\|D_xF\|^\ell} \, \big|D_y|D_xF_j \, D_xL^{-1}F_k|\big| \, \lambda(\dint x) \bigg)^2 \, \lambda(\dint y) \\
& \leq 3 \int_{\mathbb{X}} \E \bigg( \int_{\mathbb{X}} \mathbf{1}\{D^2_{x,y}F\neq \0\} \sqrt{\|D_{x}F\|^\ell+\|D^2_{x,y}F\|^\ell} \, |D_{x}F_j \, D_{x}L^{-1}F_k| \, \lambda(\dint x)\bigg)^2 \lambda(\dint y)\\
& \quad  + \bigg(3 + \frac{3}{\ell}\bigg) \int_{\mathbb{X}} \E \bigg( \int_{\mathbb{X}} \sqrt{\|D_{x}F\|^\ell+\|D^2_{x,y}F\|^\ell} \, \big|D_y|D_{x}F_j \, D_{x}L^{-1}F_k|\big| \, \lambda(\dint x)\bigg)^2 \, \lambda(\dint y)
\end{align*}
for $\ell\in\{1,2\}$. Lemma \ref{lem:DProduct} yields
\begin{equation}\label{eqn:IteratedD}
\begin{split}
\big|D_y|D_xF_j D_xL^{-1}F_k|\big| & \leq |D_y(D_xF_j D_xL^{-1}F_k)|\\
& = |D^2_{x,y}F_j  D_xL^{-1}F_k + D_xF_j  D^2_{x,y}L^{-1}F_k + D^2_{x,y}F_j  D^2_{x,y}L^{-1}F_k|
\end{split}
\end{equation}
so that
\begin{align*}
V_{jk}^{(\ell)} & \leq 3 \int_{\mathbb{X}} \E \bigg( \int_{\mathbb{X}} \mathbf{1}\{D^2_{x,y}F\neq \0\} \sqrt{\|D_{x}F\|^\ell+\|D^2_{x,y}F\|^\ell} \, |D_{x}F_j \, D_{x}L^{-1}F_k| \, \lambda(\dint x)\bigg)^2  \, \lambda(\dint y)\\
& \quad + \bigg(9 + \frac{9}{\ell}\bigg) \int_{\mathbb{X}} \E \bigg( \int_{\mathbb{X}} \sqrt{\|D_{x}F\|^\ell+\|D^2_{x,y}F\|^\ell} \, |D^2_{x,y}F_j \, D_{x}L^{-1}F_k| \, \lambda(\dint x)\bigg)^2 \\
& \hskip 2.5cm + \E \bigg( \int_{\mathbb{X}} \sqrt{\|D_{x}F\|^\ell+\|D^2_{x,y}F\|^\ell} \, |D_{x}F_j \, D^2_{x,y}L^{-1}F_k| \, \lambda(\dint x)\bigg)^2 \\
& \hskip 2.5cm + \E \bigg( \int_{\mathbb{X}} \sqrt{\|D_{x}F\|^\ell+\|D^2_{x,y}F\|^\ell} \, |D^2_{x,y}F_j \, D^2_{x,y}L^{-1}F_k| \, \lambda(\dint x)\bigg)^2 \, \lambda(\dint y).
\end{align*}
Now it follows from Proposition \ref{prop:Mehler} with $p=3/2$ and $q=3$ that
\begin{align*}
V_{jk}^{(\ell)} & \leq 3 \int_{\mathbb{X}^3} \big( \E \mathbf{1}\{D^2_{x_1,y}F\neq \0, D^2_{x_2,y}F\neq \0\} \, \big(\|D_{x_1}F\|^\ell+\|D^2_{x_1,y}F\|^\ell\big)^{3/4} \\
& \hskip 2cm \times \big(\|D_{x_2}F\|^\ell+\|D^2_{x_2,y}F\|^\ell\big)^{3/4} |D_{x_1}F_j|^{3/2} \, |D_{x_2}F_j|^{3/2} \big)^{2/3}\\
& \hskip 1.5cm \times \big( \E |D_{x_1}F_k|^3 \, |D_{x_2}F_k|^3 \big)^{1/3} \, \lambda^3(\dint(x_1,x_2,y)) \\
& \quad + \bigg(9 + \frac{9}{\ell}\bigg) \int_{\mathbb{X}^3} \big( \E \big(\|D_{x_1}F\|^\ell+\|D^2_{x_1,y}F\|^\ell\big)^{3/4} \big(\|D_{x_2}F\|^\ell+\|D^2_{x_2,y}F\|^\ell\big)^{3/4} \\
& \hskip 3.5cm \times |D^2_{x_1,y}F_j|^{3/2} \, |D^2_{x_2,y}F_j|^{3/2} \big)^{2/3} \\
& \hskip 2.85cm \times \big( \E |D_{x_1}F_k|^3 \, |D_{x_2}F_k|^3 \big)^{1/3} \, \lambda^3(\dint(x_1,x_2,y)) \\
& \quad + \frac{1}{4} \bigg(9 + \frac{9}{\ell}\bigg) \int_{\mathbb{X}^3} \big( \E \big(\|D_{x_1}F\|^\ell+\|D^2_{x_1,y}F\|^\ell\big)^{3/4} \big(\|D_{x_2}F\|^\ell+\|D^2_{x_2,y}F\|^\ell\big)^{3/4} \\
& \hskip 3.5cm \times |D_{x_1}F_j|^{3/2} \, |D_{x_2}F_j|^{3/2} \big)^{2/3}\\
& \hskip 2.85cm \times \big( \E |D^2_{x_1,y}F_k|^3 \, |D^2_{x_2,y}F_k|^3 \big)^{1/3} \, \lambda^3(\dint(x_1,x_2,y)) \\
& \quad + \frac{1}{4} \bigg(9 + \frac{9}{\ell}\bigg) \int_{\mathbb{X}^3} \big( \E \big(\|D_{x_1}F\|^\ell+\|D^2_{x_1,y}F\|^\ell\big)^{3/4} \big(\|D_{x_2}F\|^\ell+\|D^2_{x_2,y}F\|^\ell\big)^{3/4} \\
& \hskip 3.5cm \times |D^2_{x_1,y}F_j|^{3/2} \, |D^2_{x_2,y}F_j|^{3/2} \big)^{2/3}\\
& \hskip 2.85cm \times \big( \E |D^2_{x_1,y}F_k|^3 \, |D^2_{x_2,y}F_k|^3 \big)^{1/3} \, \lambda^3(\dint(x_1,x_2,y)).
\end{align*}
A short computation using H\"older's inequality shows that
\begin{equation}\label{eqn:SumRjk2}
\begin{split}
\sum_{j,k=1}^m (R_{jk}^{(2)})^2 & \leq  2 d_{convex}(F,N_\Sigma) \sum_{j,k=1}^m  V_{jk}^{(1)} + 2\sqrt{m} \|\Sigma^{-1/2}\|_{op} \sum_{j,k=1}^m  V_{jk}^{(2)} \\
& \leq  2d_{convex}(F,N_\Sigma) \gamma_5^3 + 2\sqrt{m} \|\Sigma^{-1/2}\|_{op} \gamma_6^4.
\end{split}
\end{equation}

By the Poincar\'e inequality (see Theorem \ref{thm:Poincare}), we have that
\begin{align*}
& \E \bigg( \int_{\mathbb{X}} |D_xF_j| \, |D_xL^{-1}F_k| \, \lambda(\dint x) \bigg)^2 \\
& \leq \bigg(\E  \int_{\mathbb{X}} |D_xF_j| \, |D_xL^{-1}F_k| \, \lambda(\dint x) \bigg)^2 + \E \int_{\mathbb{X}} \bigg( \int_{\mathbb{X}} D_y\big(|D_xF_j| \, |D_xL^{-1}F_k|  \big) \, \lambda(\dint x)\bigg)^2 \, \lambda(\dint y).
\end{align*}
Here, the first term is bounded because $F_j,F_k\in \operatorname{dom}D$. Using \eqref{eqn:IteratedD} and Proposition \ref{prop:Mehler} in a similar way as above, one obtains that the second term can be bounded by $3(\gamma_1+\gamma_2)<\infty$. This guarantees that the variances in the definitions of $V_{jk}^{(1)}$ and $V_{jk}^{(2)}$ exist.

Combining
\begin{align*}
\sqrt{\sum_{i,j,k=1}^m U_{ijk}^2} & \leq \sqrt{\sum_{i,j,k=1}^m (U_{ijk}^{(1)})^2} + \sqrt{\sum_{i,j,k=1}^m (U_{ijk}^{(2)})^2}\\
& \leq  \sqrt{m\sum_{j,k=1}^m (R_{jk}^{(1)})^2} + \sqrt{m\sum_{j,k=1}^m (R_{jk}^{(2)})^2} + \sqrt{\sum_{i,j,k=1}^m (U_{ijk}^{(2)})^2}
\end{align*}
with \eqref{eqn:BoundJ22M3}, \eqref{eqn:SumRjk1}, \eqref{eqn:SumRjk2}, and \eqref{eqn:SumUijk2} leads to
\begin{equation} \label{eqn:BoundJ22}
\begin{split}
|J_{2,2}| & \leq \frac{M_3 \|\Sigma^{-1}\|_{op}^{3/2}}{2\sqrt{t}}\bigg( m^2 \|\Sigma^{-1/2}\|_{op}  \gamma_4^2 + 2 d_{convex}(F,N_\Sigma) m^{5/2} \gamma_3 \\
& \hskip 2.75cm +\sqrt{ 2m d_{convex}(F,N_\Sigma) \gamma_5^3 + 2  m^{3/2}  \|\Sigma^{-1/2}\|_{op} \gamma_6^4} + m^2\gamma_4^2 \bigg).
\end{split}
\end{equation}

\vskip.3cm
\noindent{\em Part (v): \  Putting the pieces together and choosing $t$.} Finally, we may evaluate the right-hand side of \eqref{maininequality}. Recalling the definition of $\gamma$ at \eqref{defgamma}, we may simplify \eqref{eqn:BoundJ1}, \eqref{eqn:BoundJ21}, and \eqref{eqn:BoundJ22} to
\begin{align*}
| \tilde{J}| &  \leq  4 \big(\sqrt{M_2} |\log t| \sqrt{d_{convex}(F,N_\Sigma)} + 24 m^{17/12} \big) \gamma,
\end{align*}
$$
|J_{2,1}| \leq \frac{1}{2} (\sqrt{M_2} |\log t| \sqrt{d_{convex}(F,N_\Sigma)} + 24 m^{17/12} )  \gamma,
$$
and
\begin{align*}
|J_{2,2}| & \leq \frac{M_3 }{2\sqrt{t}}\bigg(  m^{2} \gamma^2 + 2 m^{5/2} \|\Sigma^{-1/2}\|_{op}  d_{convex}(F,N_\Sigma) \gamma + \sqrt{2} \sqrt{m} \sqrt{d_{convex}(F,N_\Sigma)} \gamma^{3/2}\\
& \hskip 1.5cm  + \sqrt{2} m^{3/4}  \|\Sigma^{-1/2}\|_{op}^{-1/2} \gamma^{2}  +\|\Sigma^{-1/2}\|_{op}^{-1} m^2 \gamma^2 \bigg),
\end{align*}
where we used $\|\Sigma^{-1}\|_{op}=\|\Sigma^{-1/2}\|_{op}^2$ for the last inequality.

In view of \eqref{maininequality} and Lemma \ref{lem:SmoothingStandard},  we have that
\begin{align*}
& d_{convex}(F,N_\Sigma) \\
& \leq 6 \big( \sqrt{M_2} |\log t| \sqrt{d_{convex}(F,N_\Sigma)} +  24 m^{17/12} \big) \gamma\\
& \quad + \frac{4M_3}{3\sqrt{t}} \bigg( \frac{1}{ 2} m^{2} \gamma^2 + m^{5/2} \|\Sigma^{-1/2}\|_{op} d_{convex}(F,N_\Sigma) \gamma + \frac{1}{\sqrt{2}} \sqrt{m} \sqrt{d_{convex}(F,N_\Sigma)} \gamma^{3/2} \\
& \hskip 3cm + \frac{1}{ \sqrt{2} }  m^{3/4} \|\Sigma^{-1/2}\|_{op}^{-1/2} \gamma^{2} + \frac{1}{2} m^2 \|\Sigma^{-1/2}\|_{op}^{-1} \gamma^2 \bigg)\\
& \quad +  \frac{20}{\sqrt{2}} m  \frac{\sqrt{t}}{1-t}
\end{align*}
for $t\in(0,1)$. For $t\in(0,1/2)$ the inequalities $t^{1/4} |\log t|\leq 2$ and $1-t\geq 1/2$ yield that
\begin{align*}
& d_{convex}(F,N_\Sigma) \\
& \leq \frac{12\sqrt{M_2}}{t^{1/4}}\sqrt{d_{convex}(F,N_\Sigma)} \gamma + 144 m^{17/12}  \gamma\\
& \quad + \frac{4M_3}{3\sqrt{t}}\bigg( \frac{1}{2} m^{2}  \gamma^2 +  m^{5/2} \|\Sigma^{-1/2}\|_{op} d_{convex}(F,N_\Sigma) \gamma + \frac{1}{\sqrt{2}}\sqrt{m} \sqrt{d_{convex}(F,N_\Sigma)} \gamma^{3/2} \\
& \hskip 3cm + \frac{1}{\sqrt{2}} m^{3/4}  \|\Sigma^{-1/2}\|_{op}^{-1/2} \gamma^{2} + \frac{1}{2} m^2 \|\Sigma^{-1/2}\|_{op}^{-1} \gamma^2 \bigg)\\
& \quad + \frac{40}{\sqrt{2}} m \sqrt{t}.
\end{align*}
Assume that $\gamma<1/\sqrt{2}$ (otherwise \eqref{GMBdHl} is obviously true). Then, the choice $\sqrt{t}=\max\big\{ \frac{  \sqrt{2}}{80 m } d_{convex}(F,N_\Sigma),\gamma\big\}$ leads to
\begin{align*}
& d_{convex}(F,N_\Sigma) \\
& \leq \frac{12 \sqrt{80} \sqrt{M_2}} {2^{1/4}} \sqrt{m}  \gamma + 144 m^{17/12} \gamma\\
& \quad + \frac{4M_3}{3} \bigg( \frac{1}{ 2} m^{2} \gamma +  \frac{80}{\sqrt{2}} m^{7/2} \|\Sigma^{-1/2}\|_{op} \gamma \\
& \hskip 3cm + \frac{\sqrt{40}}{2^{1/4}}  m   \gamma + \frac{1}{ \sqrt{ 2} } m^{3/4} \|\Sigma^{-1/2}\|_{op}^{-1/2} \gamma + \frac{1}{2} m^2 \|\Sigma^{-1/2}\|_{op}^{-1} \gamma \bigg)\\
& \quad + \frac{40}{\sqrt{2} } m \gamma + \frac{1}{2} d_{convex}(F,N_\Sigma).
\end{align*}
Together with \eqref{eqn:M2} and \eqref{eqn:BoundM3} we obtain
\begin{align*}
d_{convex}(F,N_\Sigma)
& \leq 2 \bigg( \frac{48\sqrt{5}}{ 2^{1/4}} + 144 + \frac{4}{\sqrt{6}}+ \frac{320}{\sqrt{3}} + \frac{16\sqrt{5}}{\sqrt{3} \cdot 2^{1/4}}+\frac{4}{ \sqrt{3} } + \frac{2\sqrt{6}}{3} + \frac{40}{\sqrt{2}}\bigg)\\
& \quad \quad \times m^5 \max\{\|\Sigma^{-1/2}\|_{op}^{-1},\|\Sigma^{-1/2}\|_{op}\} \gamma\\
& \leq 941 m^5  \max\{\|\Sigma^{-1/2}\|_{op}^{-1},\|\Sigma^{-1/2}\|_{op}\} \gamma,
\end{align*}
which completes the proof.
\end{proof}

\section{Applications}\label{sec:Applications}

\subsection{Multivariate normal approximation of first order Poisson integrals}\label{subsec:I1}

In this subsection we apply our main results to first order Poisson integrals with respect to the Poisson process $\eta$ (as considered before). For $f\in L^1(\lambda)\cap L^2(\lambda)$ we define $I_1(f)$ to be the Poisson integral of $f$ (also called the Wiener-It\^o integral of $f$  in \cite{LastPenroseBook}), namely
$$
I_1(f) := \int_{\mathbb{X}} f(x) \, \eta(\dint x) - \int_{\mathbb{X}} f(x) \, \lambda(\dint x).
$$
If $\eta$ is a proper Poisson process, i.e., it has almost surely a representation $\eta=\sum_{i\in I} \delta_{X_i}$ with a countable collection $(X_i)_{i\in I}$ of random elements of $\mathbb{X}$, this can be rewritten as
$$
I_1(f)= \sum_{i\in I} f(X_i) - \int_{\mathbb{X}} f(x) \, \lambda(\dint x).
$$
Using approximation arguments in $L^2(\mathbb{P})$, one can extend the above definition to integrands $f\in L^2(\lambda)$. Note that, for all $f,g\in L^2(\lambda)$,
\begin{equation}\label{eqn:PropertiesI1}
\E I_1(f) = 0 \quad \text{and} \quad \E I_1(f) I_1(g) = \int_{\mathbb{X}} f(x) g(x) \, \lambda(\dint x).
\end{equation}
For an exact definition and more details on first order Poisson integrals with respect to Poisson processes we refer to \cite[Subsection 12.1]{LastPenroseBook}.

\begin{coro}\label{cor:FirstOrderIntegrals}
Let $F=(I_1(f_1),\hdots,I_1(f_m))$ with $f_1,\hdots,f_m\in L^2(\lambda)$ and $m\in\N$ and let $\Sigma=(\sigma_{ij})_{i,j\in\{1,\hdots,m\}}\in\R^{m\times m}$ be positive semi-definite.
\begin{itemize}
\item [(a)] It is the case that
$$
d_3(F,N_\Sigma)\leq \frac{m}{2} \sum_{i,j=1}^m \big| \sigma_{ij} - \int_{\mathbb{X}} f_i(x) f_j(x) \, \lambda(\dint x) \big| + \frac{m^2}{4} \sum_{i=1}^m \int_{\mathbb{X}} |f_i(x)|^3 \, \lambda(\dint x).
$$
\item [(b)] If $\Sigma$ is positive definite,
\begin{align*}
d_2(F,N_\Sigma) & \leq \|\Sigma^{-1}\|_{op} \|\Sigma\|_{op}^{1/2} \sum_{i,j=1}^m \big| \sigma_{ij} - \int_{\mathbb{X}} f_i(x) f_j(x) \, \lambda(\dint x) \big|\\
&\quad + \frac{\sqrt{2\pi}m^2}{8} \|\Sigma^{-1}\|_{op}^{3/2} \|\Sigma\|_{op} \sum_{i=1}^m \int_{\mathbb{X}} |f_i(x)|^3 \, \lambda(\dint x).
\end{align*}
\item [(c)] If $\Sigma$ is positive definite, then
\begin{align*}
d_{convex}(F,N_\Sigma) \leq & 941 m^{11/2} \max\{\|\Sigma^{-1/2}\|_{op}, \|\Sigma^{-1/2}\|_{op}^{3}\} \\
&\times \max\bigg\{ \sum_{i,j=1}^m \big| \sigma_{ij} - \int_{\mathbb{X}} f_i(x) f_j(x) \, \lambda(\dint x) \big|,
\sum_{i=1}^m \int_{\mathbb{X}} |f_i(x)|^3 \, \lambda(\dint x), \\
& \quad \quad \quad \quad \quad
\bigg(\sum_{i=1}^m \int_{\mathbb{X}} f_i(x)^4 \, \lambda(\dint x) \bigg)^{1/2} \bigg\}.
\end{align*}
\end{itemize}
\end{coro}

\begin{proof}
It follows from \eqref{eqn:PropertiesI1} that, for $i,j\in \{1,\hdots,m\}$,
$$
\Cov(I_1(f_i),I_1(f_j)) = \int_{\mathbb{X}} f_i(x) f_j(x) \, \lambda(\dint x).
$$
Moreover, it is well-known (see, for example, Eqn.\ (2.6) in \cite{LPS}) that, for $f\in L^2(\lambda)$ and $x,x_1,x_2\in\mathbb{X}$,
$$
D_xI_1(f)=f(x) \quad \text{and} \quad D^2_{x_1,x_2}I_1(f) =0.
$$
This implies that $\gamma_1=\gamma_2=\gamma_5 =\gamma_6=0$, $\gamma_3 = \sum_{i=1}^m \int_{\mathbb{X}} |f_i(x)|^3 \, \lambda(\dint x)$, and
$$
\gamma_4 = \sqrt{m} \bigg(\sum_{i=1}^m \int_{\mathbb{X}} f_i(x)^4 \, \lambda(\dint x) \bigg)^{1/2}.
$$
Now (a) and (b) are immediate consequences of Theorem \ref{thm:GeneralMultivariateBound}, while (c) follows from Theorem \ref{thm:Generaldconvex}.
One final technical remark is in order.  To ensure that $\gamma_5$ and  $\gamma_6$  vanish in the event that $\int_{\mathbb{X}} f_i(x)^6 \lambda(\dint x)$ is not finite for some $i \in \{1,...,m \}$, we use the convention that $0 \cdot \infty = 0$.  This convention is supported by the technical details of the proof of Theorem \ref{thm:Generaldconvex}; in particular we have $V_{jk}^{(1)} = V_{jk}^{(2)} = 0$ because the difference operator is a deterministic function.
\end{proof}

The idea of the following proof of Corollary \ref{cor:AnalogCLT} is to show that it is only a special case of Corollary \ref{cor:FirstOrderIntegrals}.

\begin{proof}[Proof of Corollary \ref{cor:AnalogCLT}.]
Let $\mathbb{X}=\R^m$ (equipped with its Borel $\sigma$-field) and $\lambda(\cdot)=s\mathbb{P}(X_1\in\cdot)$, i.e., $\lambda$ is $s$ times the probability measure of $X_1$. For $i\in\{1,\hdots,m\}$ let us denote by $\pi_i$ the projection $\R^m\ni (y_1,\hdots,y_m)\mapsto y_i$. Then we have that
$$
Z_s =(I_1(\pi_1/\sqrt{s}),\hdots,I_1(\pi_m/\sqrt{s})).
$$
Together with the observation that, for $i\in \{1,\hdots,m\}$, $p\in(0,\infty)$, and $s>0$,
$$
\int_{\mathbb{X}} |\pi_i(x)/\sqrt{s}|^p \, \lambda(\dint x) = \E |X_1^{(i)}|^p s^{1-p/2},
$$
we see that conclusions (a) and (b) of Corollary \ref{cor:AnalogCLT} follow from conclusions (a) and (b) of Corollary \ref{cor:FirstOrderIntegrals}, with $p = 3$,  whereas  conclusion (c) follows from its counterpart in  Corollary \ref{cor:FirstOrderIntegrals} with $p \in\{3, 4\}$.
\end{proof}

\subsection{Multivariate central limit theorems for intrinsic volumes of Boolean models}\label{subsec:BM}

In the following, we derive quantitative multivariate central limit theorems for Boolean models, extending previous findings in \cite{HLS2016} and \cite[Chapter 22]{LastPenroseBook}. Our proofs rely on the general bounds from Subsection \ref{sec:MainResults} as well as arguments from \cite{HLS2016} and \cite[Chapter 22]{LastPenroseBook}.

We denote by $\mathcal{K}^d$ the set of compact convex sets in $\R^d$, $d \geq 1$.  For a probability measure $\mathbb{Q}$ on $\mathcal{K}^d$ such that $\mathbb{Q}(\{\emptyset\})=0$ and $\gamma>0$ let $\eta$ be a Poisson process on $\R^d\times\mathcal{K}^d$ with intensity measure $\gamma \lambda_d\otimes \mathbb{Q}$, where $\lambda_d$ is Lebesgue measure on $\R^d$. Note that  $\eta$ is a stationary Poisson process in $\R^d$ with independent marks in $\mathcal{K}^d$ distributed according to $\mathbb{Q}$.
A random compact convex set $Z_0$ distributed according to $\mathbb{Q}$ is called the typical grain.  From $\eta$ we construct the random closed set
$$
Z:=\bigcup_{(x,K)\in\eta} (x+K),
$$
which is called the Boolean model. For more details on Boolean models and further references we refer to \cite{SW}.

By the convex ring $\mathcal{R}^d$ we mean the set of all finite unions of elements from $\mathcal{K}^d$.
Let $V_0,V_1,\hdots,V_d:\mathcal{R}^d\to\R$ be the intrinsic volumes (see, for example, \cite[Section 14.2]{SW} for a definition via the Steiner formula and additive extensions). In particular, for $A\in\mathcal{R}^d$, $V_d(A)$ is the volume of $A$, $V_{d-1}(A)$ is half the surface area of $A$ (if $A$ is the closure of its interior), and $V_0(A)$ is the Euler characteristic of $A$.

In the sequel we study the intersection of the Boolean model $Z$ with a compact convex observation window $W\in\mathcal{K}^d$. Note that $Z\cap W$ almost surely belongs to $\mathcal{R}^d$ if $\E V_i(Z_0)<\infty$ for $i\in\{1,\hdots,d\}$. Questions of interest include finding the fraction of $W$ covered by $Z$ and the surface area of $Z\cap W$.
We address both problems simultaneously by considering
$$
\mathcal{V}(Z\cap W):=(V_0(Z\cap W),V_1(Z\cap W),\hdots, V_d(Z\cap W)).
$$

Denote by $r(K)$ the inradius of $K\in\mathcal{K}^d$.  In \cite[Theorem 3.1]{HLS2016} it is shown that there exists a matrix $\Sigma=(\sigma_{i,j})_{i,j\in\{0,\hdots,d\}}\in\R^{(d+1)\times (d+1)}$ such that
$$
\Sigma(W):=\frac{1}{V_d(W)} (\Cov(V_i(Z\cap W),V_j(Z\cap W)))_{i,j\in\{0,\hdots,d\}} \to \Sigma \quad \text{as} \quad r(W)\to\infty
$$
if $\E V_i(Z_0)^2<\infty$ for $i\in\{1,\hdots,d\}$. If, additionally, $\mathbb{P}(V_d(Z_0)>0)>0$, the asymptotic covariance matrix $\Sigma$ is positive definite (see \cite[Theorem 4.1]{HLS2016}). We describe the asymptotic behavior of $\mathcal{V}(Z\cap W)$ as $r(W)\to\infty$ with respect to
$d_3, d_2,$ and $d_{convex}$.

\begin{theo}\label{thm:CLTBM}
\begin{itemize}
\item [(a)] If $\E V_i(Z_0)^3<\infty$ for $i\in\{1,\hdots,d\}$, there exists a constant $C_1\in(0,\infty)$ depending on $d$, $\gamma$, and $\mathbb{Q}$ such that
$$
d_3\bigg( \frac{\mathcal{V}(Z\cap W)-\E \mathcal{V}(Z\cap W)}{\sqrt{V_d(W)}} ,N_\Sigma\bigg) \leq C_1 \frac{1}{r(W)^{\min\{1,d/2\}}}
$$
for all $W\in\mathcal{K}^d$ with $r(W)\geq 1$.
\item [(b)] If $\E V_i(Z_0)^3<\infty$ for $i\in\{1,\hdots,d\}$ and $\mathbb{P}(V_d(Z_0)>0)>0$, there exists a constant $C_2\in(0,\infty)$ depending on $d$, $\gamma$, and $\mathbb{Q}$ such that
$$
d_2\bigg( \frac{\mathcal{V}(Z\cap W)-\E \mathcal{V}(Z\cap W)}{\sqrt{V_d(W)}} ,N_\Sigma\bigg) \leq C_2 \frac{1}{r(W)^{\min\{1,d/2\}}}
$$
for all $W\in\mathcal{K}^d$ with $r(W)\geq 1$.
\item [(c)] If $\E V_i(Z_0)^4<\infty$ for $i\in\{1,\hdots,d\}$ and $\mathbb{P}(V_d(Z_0)>0)>0$, there exists a constant $C_{3}\in(0,\infty)$ depending on $d$, $\gamma$, and $\mathbb{Q}$ such that
$$
d_{convex}\bigg( \frac{\mathcal{V}(Z\cap W)-\E \mathcal{V}(Z\cap W)}{\sqrt{V_d(W)}} ,N_\Sigma\bigg) \leq C_{3} \frac{1}{r(W)^{\min\{1,d/2\}}}
$$
for all $W\in\mathcal{K}^d$ with $r(W)\geq 1$.
\item [(d)] If $N_\Sigma$ is replaced by $N_{\Sigma(W)}$, the assertions (a)-(c) hold with the rate $1/\sqrt{V_d(W)}$.
\end{itemize}
\end{theo}
Theorem \ref{thm:CLTBM}(a) improves upon the moment assumptions of \cite[Theorem 9.1]{HLS2016} by requiring existence of third moments (i.e., $\E V_i(Z_0)^3<\infty$ for $i\in\{1,\hdots,d\}$) and not fourth moments. Parts (b) and (c) extend  \cite[Theorem 9.1]{HLS2016} to different distances, in particular, the non-smooth $d_{convex}$-distance.  The findings of \cite{HLS2016} as well as the univariate results in \cite{LastPenroseBook} consider so-called geometric functionals, which include intrinsic volumes. Theorem \ref{thm:CLTBM} could be also generalized to these functionals, but for the sake of simplicity we consider only intrinsic volumes. Since our proof of Theorem \ref{thm:CLTBM} is based on second order Poincar\'e inequalities, it does not require dealing with the whole chaos expansion as in \cite{HLS2016}. For previous results on volume and surface area of Boolean models we refer the reader to \cite{HLS2016}. Theorem \ref{thm:CLTBM} indicates that the slow convergence of $\Sigma(W)$ to $W$ weakens the rate of convergence for $d\geq 3$ (see also \cite[Remark 9.5]{HLS2016}). The rate of convergence $1/\sqrt{V_d(W)}$ for the distance to $N_{\Sigma(W)}$ is comparable to $1/\sqrt{n}$ in the classical central limit theorem for sums of $n$  i.i.d.\ random vectors and, thus, presumably optimal.

We prepare the proof of Theorem \ref{thm:CLTBM} by two lemmas. In the sequel, we use the Wills functional $\overline{V}(K):=\sum_{i=0}^d \kappa_{d-i} V_i(K)$ for $K\in\mathcal{K}^d$, where $\kappa_{d-i}$ is the volume of the $(d-i)$-dimensional unit ball. We write the difference operator $D$ with respect to the pair $(x, K)$, with $x \in \R^d, K\in\mathcal{K}^d$.

\begin{lemm}\label{lem:BoundDiffsBoolean}
There exists a constant $C\in(0,\infty)$ only depending on $d$, $\gamma$, and $\mathbb{Q}$ such that, for $x,x_1,x_2\in\R^d$, $K,K_1,K_2\in\mathcal{K}^d$, $i,j\in\{0,\hdots,d\}$, and $m,m_1,m_2\in\{1,\hdots,6\}$,
$$
\E |D_{(x,K)}V_i(Z\cap W)|^m \leq C^m \overline{V}((x+K)\cap W)^m,
$$
$$
\E |D^2_{(x_1,K_1),(x_2,K_2)}V_i(Z\cap W)|^m \leq C^m \overline{V}((x_1+K_1)\cap (x_2+K_2)\cap W)^m,
$$
\begin{align*}
& \E |D_{(x_1,K_1)}V_i(Z\cap W)|^{m_1} |D_{(x_2,K_2)}V_j(Z\cap W)|^{m_2}\\
& \leq C^{m_1+m_2} \overline{V}((x_1+K_1)\cap W)^{m_1} \overline{V}((x_2+K_2)\cap W)^{m_2},
\end{align*}
and
\begin{align*}
& \E |D^2_{(x_1,K_1),(x,K)}V_i(Z\cap W)|^{m_1} |D^2_{(x_2,K_2),(x,K)}V_j(Z\cap W)|^{m_2}\\
& \leq C^{m_1+m_2} \overline{V}((x_1+K_1)\cap (x+K)\cap  W)^{m_1} \overline{V}((x_2+K_2)\cap (x+K)\cap W)^{m_2}.
\end{align*}
\end{lemm}

\begin{proof}
For $m\in\{2,3\}$ or $i=j$ and $m_1=m_2=2$ this is shown in \cite{LastPenroseBook} in Proposition 22.4 in connection with (22.30) and (22.31) (see also \cite[Lemma 3.3]{HLS2016}), but the proof can be extended to $i\neq j$ and the other choices for $m,m_1,m_2$.
\end{proof}

Moreover, we will use the following translative integral formula from \cite[Proposition 22.5]{LastPenroseBook} and \cite[Lemma 3.4]{HLS2016}.

\begin{lemm}\label{lem:Translation}
For all $K,L\in \mathcal{K}^d$,
$$
\int_{\R^d} \overline{V}((x+K)\cap L) \, \dint x \leq \overline{V}(K) \overline{V}(L).
$$
\end{lemm}

\begin{proof}[Proof of Theorem \ref{thm:CLTBM}]
We deduce Theorem \ref{thm:CLTBM} from Theorem \ref{thm:GeneralMultivariateBound} and Theorem \ref{thm:Generaldconvex} by bounding $\gamma_1,\hdots,\gamma_6$ from Subsection \ref{sec:MainResults} as follows.  We denote by $\tilde{\gamma}_1,\hdots,\tilde{\gamma}_6$ the corresponding terms without the normalization $1/\sqrt{V_d(W)}$ of the functionals. Without loss of generality we can assume that $\gamma=1$. In the sequel let $(Z_n)_{n\in\N}$ be independent copies of the typical grain $Z_0$. It follows from Lemma \ref{lem:BoundDiffsBoolean}, the monotonicity and the translation invariance of the Wills functional (i.e., $\overline{V}(K)\leq \overline{V}(L)$ for $K,L\in\mathcal{K}^d$ with $K \subseteq L$ and $\overline{V}(K+x)=\overline{V}(K)$ for $K\in\mathcal{K}^d$ and $x\in\R^d$), and Lemma \ref{lem:Translation} that
\begin{align*}
\tilde{\gamma}_1^2 & \leq (d+1)^2 C^4 \E \int_{(\R^d)^3} \overline{V}((x_1+Z_1)\cap (x_3+Z_3)\cap W) \overline{V}((x_2+Z_1)\cap (x_3+Z_3)\cap W)\\
& \hskip 4cm \overline{V}((x_1+Z_1)\cap W) \overline{V}((x_2+Z_2)\cap W) \, \dint(x_1,x_2,x_3) \\
& \leq (d+1)^2 C^4 \E \int_{(\R^d)^3} \overline{V}((x_1+Z_1)\cap (x_3+Z_3)\cap W) \overline{V}((x_2+Z_1)\cap (x_3+Z_3)\cap W)\\
& \hskip 4cm \overline{V}(Z_1) \overline{V}(Z_2) \, \dint(x_1,x_2,x_3) \allowdisplaybreaks\\
& \leq (d+1)^2 C^4 \E  \int_{\R^d} \overline{V}(Z_1)^2 \overline{V}(Z_2)^2 \overline{V}((x+Z_3)\cap W)^2 \, \dint x \allowdisplaybreaks\\
& \leq (d+1)^2 C^4 \E  \int_{\R^d} \overline{V}(Z_1)^2 \overline{V}(Z_2)^2 \overline{V}(Z_3) \overline{V}((x+Z_3)\cap W) \, \dint x \\
& \leq (d+1)^2 C^4 \E \overline{V}(Z_1)^2 \E \overline{V}(Z_2)^2 \E \overline{V}(Z_3)^2 \overline{V}(W) \\
& \leq (d+1)^2 C^4 (\E \overline{V}(Z_0)^2)^3 \overline{V}(W)
\end{align*}
and
\begin{align*}
\tilde{\gamma}_2^2 & \leq (d+1)^2 C^4 \E \int_{(\R^d)^3} \overline{V}((x_1+Z_1)\cap (x_3+Z_3)\cap W)^2 \overline{V}((x_2+Z_2)\cap (x_3+Z_3)\cap W)^2 \\
& \hskip 4cm \dint (x_1,x_2,x_3)\\
& \leq (d+1)^2 C^4 \E \int_{\R^d}  \overline{V}(Z_1)^2 \overline{V}(Z_2)^2 \overline{V}((x+Z_3)\cap W)^2 \, \dint x \allowdisplaybreaks\\
& \leq (d+1)^2 C^4 \E \overline{V}(Z_1)^2 \overline{V}(Z_2)^2 \overline{V}(Z_3)^2 \overline{V}(W)\\
& = (d+1)^2 C^4  (\E \overline{V}(Z_0)^2)^3 \overline{V}(W).
\end{align*}
Hence, we see that $\gamma_1$ and $\gamma_2$ are at most of the order $\sqrt{\overline{V}(W)}/V_d(W)$. From the same arguments as above we obtain that, for $k\in\N$,
\begin{equation}\label{eqn:BoundD1k}
\E \int_{\R^d} \overline{V}((x+Z_0)\cap W)^k \, \dint x
\leq \E \overline{V}(Z_0)^{k-1} \int_{\R^d} \overline{V}((x+Z_0)\cap W) \, \dint x \\
\leq \E \overline{V}(Z_0)^{k} \overline{V}(W),
\end{equation}
whence $\gamma_3$ is at most of order $\overline{V}(W)/V_d(W)^{3/2}$. We can also show that
\begin{align*}
& \E \int_{(\R^d)^2} \overline{V}((x_1+Z_1 ) \cap (x_2+Z_2)\cap W)^2\\
& \hskip1.5cm  \big(\overline{V}((x_1+Z_1 ) \cap (x_2+Z_2)\cap W)^2  + \overline{V}((x_1+Z_1 ) \cap W)^2  \big) \dint(x_1,x_2)\\
& \leq 2 \E \int_{(\R^d)^2} \overline{V}((x_1+Z_1 ) \cap (x_2+Z_2)\cap W) \overline{V}(Z_2) \overline{V}(Z_1)^2 \dint(x_1,x_2)\\
& \leq 2 \E \overline{V}(Z_1)^3 \overline{V}(Z_2)^2 \overline{V}(W)
\end{align*}
so that together with \eqref{eqn:BoundD1k}, we deduce that $\gamma_4$ is at most of order $\sqrt{\overline{V}(W)}/V_d(W)$.
Jensen's inequality and Lemma \ref{lem:BoundDiffsBoolean} lead to
\begin{equation}\label{eqn:NormDBM}
\E \| D_{(x,K)}\mathcal{V}(Z\cap W) \|^6 \leq (d+1)^3 C^6 \overline{V}((x+K)\cap W)^6
\end{equation}
for $x\in\mathbb{R}^d$ and $K\in\mathcal{K}^d$ and
\begin{equation}\label{eqn:NormD2BM}
\begin{split}
\E \| D^2_{(x_1,K_1),(x_2,K_2)}\mathcal{V}(Z\cap W) \|^6 & \leq (d+1)^3 C^6 \overline{V}((x_1+K_1)\cap (x_2+K_2)\cap W)^6\\
& \leq (d+1)^3 C^6 \overline{V}((x_1+K_1)\cap W)^6
\end{split}
\end{equation}
for $x_1,x_2\in\mathbb{R}^d$ and $K_1,K_2\in\mathcal{K}^d$. This implies that, for $x,y\in\mathbb{R}^d$, $K,L\in\mathcal{K}^d$, and $\ell\in\{1,2\}$,
\begin{align*}
& \E \big( \|D_{(x,K)}\mathcal{V}(Z\cap W)\|^\ell + \|D^2_{(x,K),(y,L)}\mathcal{V}(Z\cap W)\|^\ell \big)^3 \\
& \leq 4 \E \|D_{(x,K)}\mathcal{V}(Z\cap W)\|^{3\ell} + 4 \E \|D^2_{(x,K),(y,L)}\mathcal{V}(Z\cap W)\|^{3\ell} \\
& \leq 4 \big( \E \|D_{(x,K)}\mathcal{V}(Z\cap W)\|^{6}\big)^{\ell/2} + 4 \big(\E \|D^2_{(x,K),(y,L)}\mathcal{V}(Z\cap W)\|^6\big)^{\ell/2} \\
& \leq 8 (d+1)^{3\ell/2} C^{3\ell} \overline{V}((x+K)\cap W)^{3\ell}.
\end{align*}
From \cite[Lemma 3.2]{HLS2016} or \cite[Lemma 22.6]{LastPenroseBook} it follows that, for $i\in\{0,\hdots,d\}$, $x_1,x_2\in\R^d$, and $K_1,K_2\in\mathcal{K}^d$,
$$
D^2_{(x_1,K_1),(x_2,K_2)}V_i(Z\cap W) = V_i(Z\cap (x_1+K_1)\cap(x_2+K_2)\cap W) - V_i((x_1+K_1)\cap(x_2+K_2)\cap W).
$$
Consequently, we have that
\begin{align*}
{\bf 1}\{D^2_{(x_1,K_1),(x_2,K_2)}\mathcal{V}(Z\cap W)\neq \0\} & \leq {\bf 1}\{(x_1+K_1)\cap(x_2+K_2)\cap W\neq \emptyset\}\\
& \leq \overline{V}((x_1+K_1)\cap(x_2+K_2)\cap W).
\end{align*}
From H\"older's inequality, \eqref{eqn:NormDBM}, and \eqref{eqn:NormD2BM}, we obtain that, for $x_1,x_2,y\in\mathbb{R}^d$, $K_1,K_2,L\in\mathcal{K}^d$, and $\ell\in\{1,2\}$,
\begin{align*}
& \E {\bf 1}\{D^2_{(x_1,K_1),(y,L)}\mathcal{V}(Z\cap W)\neq \0, D^2_{(x_2,K_2),(y,L)}\mathcal{V}(Z\cap W)\neq \0 \}\\
&\quad \times \big( \|D_{(x_1,K_1)}\mathcal{V}(Z\cap W)\|^\ell + \|D^2_{(x_1,K_1),(y,L)}\mathcal{V}(Z\cap W)\|^\ell \big)^{3/4}\\
&\quad \times \big( \|D_{(x_2,K_2)}\mathcal{V}(Z\cap W)\|^\ell + \|D^2_{(x_2,K_2),(y,L)}\mathcal{V}(Z\cap W)\|^\ell \big)^{3/4}\\
&\quad \times |D_{(x_1,K_1)}V_i(Z\cap W)|^{3/2} \, |D_{(x_2,K_2)}V_i(Z\cap W)|^{3/2} \\
& \leq {\bf 1}\{(x_1+K_1)\cap(y+L)\cap W\neq \emptyset\} {\bf 1}\{(x_2+K_2)\cap(y+L)\cap W\neq \emptyset\} \\
&\quad  \times \E \big( \|D_{(x_1,K_1)}\mathcal{V}(Z\cap W)\|^{3\ell/4} + \|D^2_{(x_1,K_1),(y,L)}\mathcal{V}(Z\cap W)\|^{3\ell/4} \big)\\
&\quad \quad \quad \times \big( \|D_{(x_2,K_2)}\mathcal{V}(Z\cap W)\|^{3\ell/4} + \|D^2_{(x_2,K_2),(y,L)}\mathcal{V}(Z\cap W)\|^{3\ell/4} \big)\\
&\quad \quad \quad \times |D_{(x_1,K_1)}V_i(Z\cap W)|^{3/2} \, |D_{(x_2,K_2)}V_i(Z\cap W)|^{3/2} \allowdisplaybreaks\\
& \leq {\bf 1}\{(x_1+K_1)\cap(y+L)\cap W\neq \emptyset\} {\bf 1}\{(x_2+K_2)\cap(y+L)\cap W\neq \emptyset\} \\
&\quad \times \big( \big(\E \|D_{(x_1,K_1)}\mathcal{V}(Z\cap W)\|^{6} \big)^{\ell/8} + \big( \E \|D^2_{(x_1,K_1),(y,L)}\mathcal{V}(Z\cap W)\|^{6}\big)^{\ell/8} \big)\\
&\quad \times \big( \big(\E \|D_{(x_2,K_2)}\mathcal{V}(Z\cap W)\|^{6} \big)^{\ell/8} + \big( \E \|D^2_{(x_2,K_2),(y,L)}\mathcal{V}(Z\cap W)\|^{6}\big)^{\ell/8} \big)\\
& \quad \times \big( \E |D_{(x_1,K_1)}V_i(Z\cap W)|^3 \, |D_{(x_2,K_2)}V_i(Z\cap W)|^3 \big)^{1/2} \\
& \leq {\bf 1}\{(x_1+K_1)\cap(y+L)\cap W\neq \emptyset\} {\bf 1}\{(x_2+K_2)\cap(y+L)\cap W\neq \emptyset\}\\
& \quad \times 4 (d+1)^{3\ell/4} C^{3\ell/2} \overline{V}((x_1+K_1)\cap W)^{3\ell/4}  \overline{V}((x_2+K_2)\cap W)^{3\ell/4} \\
& \quad \times C^3 \overline{V}((x_1+K_1)\cap W)^{3/2} \overline{V}((x_2+K_2)\cap W)^{3/2}\\
& = {\bf 1}\{(x_1+K_1)\cap(y+L)\cap W\neq \emptyset\} {\bf 1}\{(x_2+K_2)\cap(y+L)\cap W\neq \emptyset\}\\
& \quad \times 4 (d+1)^{3\ell/4} C^{3+3\ell/2} \overline{V}((x_1+K_1)\cap W)^{3/2+3\ell/4}  \overline{V}((x_2+K_2)\cap W)^{3/2+3\ell/4}.
\end{align*}
Combining the previous estimates with Lemma \ref{lem:BoundDiffsBoolean} yields
\begin{align*}
\tilde{\gamma}_5^3 & \leq 3(d+1)^2 \E \int_{(\mathbb{R}^d)^3} \overline{V}((x_1+Z_1)\cap (x_3+Z_3)\cap W) \overline{V}((x_2+Z_2)\cap (x_3+Z_3)\cap W) \\
& \hskip 2cm \times 4^{2/3} (d+1)^{1/2} C^3 \overline{V}((x_1+Z_1)\cap W)^{3/2} \overline{V}((x_2+Z_2)\cap W)^{3/2}\\
& \hskip 2cm \times C^2 \overline{V}((x_1+Z_1)\cap W)  \overline{V}((x_2+Z_2)\cap W) \, \dint(x_1,x_2,x_3) \\
& \quad + 27 (d+1)^2 \E \int_{(\mathbb{R}^d)^3} 2 (d+1)^{1/2} C \overline{V}((x_1+Z_1)\cap W)^{1/2} \overline{V}((x_2+Z_2)\cap W)^{1/2} \\
& \hskip 2.5cm \times C^2 \overline{V}((x_1+Z_1)\cap(x_3+Z_3)\cap W) \overline{V}((x_2+Z_2)\cap(x_3+Z_3)\cap W)\\
& \hskip 2.5cm \times C^2 \overline{V}((x_1+Z_1)\cap W) \overline{V}((x_2+Z_2)\cap W) \, \dint(x_1,x_2,x_3)
\end{align*}
and
\begin{align*}
\tilde{\gamma}_6^4 & \leq 3(d+1)^2 \E \int_{(\mathbb{R}^d)^3} \overline{V}((x_1+Z_1)\cap (x_3+Z_3)\cap W) \overline{V}((x_2+Z_2)\cap (x_3+Z_3)\cap W) \\
& \hskip 2cm \times 4^{2/3} (d+1) C^4 \overline{V}((x_1+Z_1)\cap W)^{2} \overline{V}((x_2+Z_2)\cap W)^{2}\\
& \hskip 2cm \times C^2 \overline{V}((x_1+Z_1)\cap W)  \overline{V}((x_2+Z_2)\cap W) \, \dint(x_1,x_2,x_3) \\
& \quad + \frac{81}{4} (d+1)^2 \E \int_{(\mathbb{R}^d)^3} 2 (d+1) C^2 \overline{V}((x_1+Z_1)\cap W) \overline{V}((x_2+Z_2)\cap W) \\
& \hskip 2.5cm \times C^2 \overline{V}((x_1+Z_1)\cap(x_3+Z_3)\cap W) \overline{V}((x_2+Z_2)\cap(x_3+Z_3)\cap W)\\
& \hskip 2.5cm \times C^2 \overline{V}((x_1+Z_1)\cap W) \overline{V}((x_2+Z_2)\cap W) \, \dint(x_1,x_2,x_3).
\end{align*}
Monotonicity and translation invariance of the Wills functional and Lemma \ref{lem:Translation} imply
\begin{align*}
\tilde{\gamma}_5^3 & \leq 3\cdot 4^{2/3} (d+1)^{5/2} C^5 \E \overline{V}(Z_1)^{7/2} \overline{V}(Z_2)^{7/2} \overline{V}(Z_3)^2 \overline{V}(W)\\
& \quad + 54 (d+1)^{5/2} C^5 \E \overline{V}(Z_1)^{5/2} \overline{V}(Z_2)^{5/2} \overline{V}(Z_3)^2 \overline{V}(W)
\end{align*}
and
\begin{align*}
\tilde{\gamma}_6^4 & \leq 3\cdot 4^{2/3} (d+1)^3 C^6 \E \overline{V}(Z_1)^{4} \overline{V}(Z_2)^{4} \overline{V}(Z_3)^2 \overline{V}(W) \\
& \quad + \frac{81}{2} (d+1)^{3} C^6 \E \overline{V}(Z_1)^{3} \overline{V}(Z_2)^{3} \overline{V}(Z_3)^2 \overline{V}(W).
\end{align*}
Thus, $\gamma_5$ and $\gamma_6$ are at most of the orders $\overline{V}(W)^{1/3}/V_d(W)^{5/6}$ and $\overline{V}(W)^{1/4}/V_d(W)^{3/4}$, respectively. By \cite[Lemma 3.7]{HLS2016}, there exists a dimension dependent constant $C_d\in(0,\infty)$ such that
$$
\frac{\overline{V}(W)}{V_d(W)} \leq C_d \quad \text{for all} \quad W\in\mathcal{K}^d \quad \text{with} \quad r(W)\geq 1.
$$
This implies that $\gamma_1$, $\gamma_2$, $\gamma_3$, $\gamma_4$, $\gamma_5$, and $\gamma_6$ have at most the order $1/\sqrt{V_d(W)}$. It is known \cite[Theorem 3.1]{HLS2016} that there exists a constant $C_\Sigma\in(0,\infty)$ only depending on $d$, $\gamma$, and $\mathbb{Q}$ such that
$$
\bigg| \frac{\Cov(V_i(Z\cap W),V_j(Z\cap W))}{V_d(W)} - \sigma_{i,j} \bigg| \leq C_\Sigma \frac{1}{r(W)}
$$
for $i,j\in\{0,\hdots,d\}$ and $W\in\mathcal{K}^d$ with $r(W)\geq 1$. Now Theorem \ref{thm:GeneralMultivariateBound} and Theorem \ref{thm:Generaldconvex} complete the proof.
\end{proof}

\vskip.3cm

\subsection{Multivariate normal approximation for functionals of marked Poisson processes}\label{subsec:Marks}

In this subsection we establish a consequence of Theorem \ref{thm:GeneralMultivariateBound} and Theorem \ref{thm:Generaldconvex}, which can be seen as a multivariate version of Proposition 1.4 and Theorem 6.1 in \cite{LPS}. This result will be used heavily in the companion paper \cite{SchulteYukich2017}, in order to deduce rates of normal approximation for Poisson functionals which may be expressed as sums of stabilizing score functions.
We work in the context of marked Poisson processes, where $(\MM,\mathcal{F}_\MM,\la_\MM)$ denotes the probability space of marks.  Let $\widehat{\XX}:=\XX\times \MM$, put $\widehat{\mathcal{F}}$ to be the product $\sigma$-field of $\mathcal{F}$ and $\mathcal{F}_\MM$,  and let $\widehat{\la}$ be the product measure of $\la$ and $\la_{\MM}$. Here, $(\XX,\mathcal{F},\lambda)$ is as before.
For a given point $x \in \XX$ we denote by $M_x$ the corresponding random mark, which has distribution $\la_{\MM}$ and which is independent of everything else.

Let $F=(F_1,\hdots,F_m)$, $m\in\N$, be a vector of Poisson functionals $F_1,\hdots,F_m\in\operatorname{dom}D$ with $\E F_i=0$, $i\in\{1,\hdots,m\}$.  Define for all $c,p \in (0, \infty)$,
\begin{align*}
\Gamma_1(c,p) & :=    c^{\frac{2}{4+p}} \bigg(\sum_{i=1}^m\int_{\mathbb{X}} \bigg( \int_{\mathbb{X}} \PP(D_{(x_1,M_{x_1}), (x_2,M_{x_2})}^2 F_i\neq 0)^{\frac{p}{16+4p}} \, \lambda(\dint x_2) \bigg)^2 \, \lambda(\dint x_1) \bigg)^{1/2}  \allowdisplaybreaks \\
\Gamma_2(c,p) & :=  c^{\frac{3}{4+p}} \sum_{i=1}^m \int_{\mathbb{X}} \PP(D_{(x,M_x)} F_i\neq 0)^{\frac{1+p}{4+p}} \, \lambda(\dint x) \allowdisplaybreaks\\
\Gamma_3(c,p) & :=  c^{\frac{2}{4+p}} \bigg(\sum_{i=1}^m 9 \int_{\mathbb{X}^2} \mathbb{P}(D^2_{(x_1,M_{x_1}),(x_2,M_{x_2})}F_i\neq 0)^{\frac{p}{8+2p}} \, \lambda^2(\dint(x_1,x_2)) \\
& \hskip 2.5cm  +\int_{\mathbb{X}} \mathbb{P}(D_{(x,M_x)}F_i\neq 0)^{\frac{p}{4+p}} \, \lambda(\dint x) \bigg)^{1/2} \allowdisplaybreaks\\
\Gamma_4(c,p) & := c^{\frac{5}{3(4+p)}} \bigg(62 \int_{\mathbb{X}} \bigg( \int_{\mathbb{X}} \mathbb{P} (D^2_{(x_1,M_{x_1}),(x_2,M_{x_2})}F\neq\0)^{\frac{p-2}{24+6p}}  \, \lambda(\dint x_2) \bigg)^2  \, \lambda(\dint x_1) \bigg)^{1/3} \allowdisplaybreaks\\
\Gamma_5(c,p) & := c^{\frac{3}{2(4+p)}}\bigg(49 \int_{\mathbb{X}} \bigg( \int_{\mathbb{X}} \mathbb{P}(D^2_{(x_1,M_{x_1}),(x_2,M_{x_2})}F\neq\0)^{\frac{p-2}{24+6p}}  \, \lambda(\dint x_2) \bigg)^2  \, \lambda(\dint x_1)\bigg)^{1/4}.
\end{align*}

\begin{theo}\label{thm:generalStabilization}
Let $F=(F_1,\hdots,F_m)$, $m\in\N$, be a vector of Poisson functionals $F_1,\hdots,F_m\\ \in\operatorname{dom}D$ with $\E F_i=0$, $i\in\{1,\hdots,m\}$, and assume that there are constants $c,p\in (0,\infty)$ such that
\begin{equation}\label{eqn:AssumptionMomentBoundsD}
\E |D_{(x,M_x)} F_i|^{4+p} \leq c, \quad \lambda\text{-a.e. } x\in \mathbb{X},
\end{equation}
and
\begin{equation}\label{eqn:AssumptionMomentBoundsD2}
\E |D^2_{(x_1,M_{x_1}), (x_2,M_{x_2})} F_i|^{4+p} \leq c, \quad \lambda^2\text{-a.e. } (x_1, x_2)\in \mathbb{X}^2,
\end{equation}
for all $i\in\{1,\hdots,m\}$.
\begin{itemize}
\item [(a)] For positive semi-definite $\Sigma=(\sigma_{ij})_{i,j\in\{1,\hdots,m\}}\in\R^{m\times m}$,
$$
d_3(F,N_{\Sigma}) \leq \frac{m}{2} \sum_{i,j=1}^m |\sigma_{ij}-\Cov(F_i,F_j)| + \frac{3m^{3/2}}{2} \Gamma_1(c,p)  +  \frac{m^2}{4} \Gamma_2(c,p).
$$
\item [(b)] For positive definite $\Sigma\in\R^{m\times m}$,
\begin{align*}
d_2(F,N_{\Sigma}) & \leq \|\Sigma^{-1}\|_{op} \|\Sigma\|_{op}^{1/2} \sum_{i,j=1}^m |\sigma_{ij}-\Cov(F_i,F_j)| + 3\|\Sigma^{-1}\|_{op} \|\Sigma\|_{op}^{1/2} \sqrt{m}  \Gamma_1(c,p)\\
& \quad  + \frac{\sqrt{2\pi}}{8} \|\Sigma^{-1}\|_{op}^{3/2} \|\Sigma\|_{op}m^2\Gamma_2(c,p).
\end{align*}
\item[(c)] Let $\Sigma\in\R^{m\times m}$ be positive definite and assume that $p>2$. Then
\begin{align*}
d_{convex}(F, N_\Sigma)
 \leq & 941 m^5  \max\{\|\Sigma^{-1/2}\|_{op}, \|\Sigma^{-1/2}\|_{op}^{3} \}\\
& \times
 \max\bigg\{ \sum_{i,j\in\{1,\hdots,m\}} |\sigma_{ij}-\Cov(F_i,F_j)|, \sqrt{m} \Gamma_1(c,p), \Gamma_2(c,p),\\
& \hskip 2.5cm \sqrt{m} \Gamma_3(c,p), m^{5/6} \Gamma_4(c,p), m^{3/4} \Gamma_5(c,p) \bigg\}.
\end{align*}
\end{itemize}
\end{theo}

\begin{proof}
Obviously, Theorem \ref{thm:GeneralMultivariateBound} and Theorem \ref{thm:Generaldconvex} can be also applied to marked Poisson processes. By combining the product form of $\widehat{\lambda}$ with the Cauchy-Schwarz inequality we obtain
\begin{align*}
& \int_{\widehat{\mathbb{X}}^3} \big[ \E (D^2_{\wx_1,\wx_3}F_i)^2 (D^2_{\wx_2,\wx_3}F_i)^2 \big]^{1/2} \big[ \E (D_{\wx_1}F_j)^2 (D_{\wx_2}F_j)^2 \big]^{1/2} \, \widehat{\lambda}^3(\dint(\wx_1,\wx_2,\wx_3)) \\
& = \int_{\mathbb{X}^3} \int_{\MM^3} \big[ \E (D^2_{(x_1,m_1),(x_3,m_3)}F_i)^2 (D^2_{(x_2,m_2),(x_3,m_3)}F_i)^2 \big]^{1/2}\\
& \hskip 1.8cm \times \big[ \E (D_{(x_1,m_1)}F_j)^2 (D_{(x_2,m_2)}F_j)^2 \big]^{1/2}\, \lambda_\MM^3(\dint(m_1,m_2,m_3)) \, \lambda^3(\dint(x_1,x_2,x_3)) \allowdisplaybreaks \\
& \leq \int_{\mathbb{X}^3}  \bigg[ \int_{\MM^3} \E (D^2_{(x_1,m_1),(x_3,m_3)}F_i)^2 (D^2_{(x_2,m_2),(x_3,m_3)}F_i)^2 \, \lambda_\MM^3(\dint(m_1,m_2,m_3))  \bigg]^{1/2}\\
& \quad \quad \quad \quad \times \bigg[\int_{\MM^3} \E (D_{(x_1,m_1)}F_j)^2 (D_{(x_2,m_2)}F_j)^2 \, \lambda_\MM^3(\dint(m_1,m_2,m_3)) \bigg]^{1/2}  \, \lambda^3(\dint(x_1,x_2,x_3)) \allowdisplaybreaks\\
& = \int_{\mathbb{X}^3} \big[ \E (D^2_{(x_1,M_{x_1}),(x_3,M_{x_3})}F_i)^2 (D^2_{(x_2,M_{x_2}),(x_3,M_{x_3})}F_i)^2 \big]^{1/2}\\
& \hskip 1.2cm \times \big[ \E (D_{(x_1,M_{x_1})}F_j)^2 (D_{(x_2,M_{x_2})}F_j)^2 \big]^{1/2} \, \lambda^3(\dint(x_1,x_2,x_3)).
\end{align*}
Since we can apply the same arguments to the other terms, we see that the bounds from Theorem \ref{thm:GeneralMultivariateBound} and Theorem \ref{thm:Generaldconvex} are still valid if we integrate with respect to $\lambda$ and always replace $x_i$ by $(x_i,M_{x_i})$, where $M_{x_i}$ is an independent random mark. We denote the corresponding versions of $\gamma_1,\hdots,\gamma_6$ by $\hat{\gamma}_1,\hdots,\hat{\gamma}_6$.
For $i\in\{1,\hdots,m\}$ and $q\in(0,4+p)$ it follows from \eqref{eqn:AssumptionMomentBoundsD}, \eqref{eqn:AssumptionMomentBoundsD2}, and H\"older's inequality that
$$
\E |D_{(x,M_x)}F_i|^q \leq c^{\frac{q}{4+p}} \PP(D_{(x,M_x)}F_i\neq 0)^{\frac{4+p-q}{4+p}}, \quad \lambda\text{-a.e.}\ x\in\mathbb{X},
$$
and
$$
\E |D^2_{(x_1,M_{x_1}),(x_2,M_{x_2})}F_i|^q \leq c^{\frac{q}{4+p}} \PP(D^2_{(x_1,M_{x_1}),(x_2,M_{x_2})}F_i\neq 0)^{\frac{4+p-q}{4+p}}, \quad \lambda^2\text{-a.e.}\ (x_1,x_2)\in\mathbb{X}^2.
$$
Applying H\"older's inequality to separate expectations of products and using these inequalities, one obtains
\begin{equation}\label{eqn:Gammas1-4}
\hat{\gamma}_1\leq \sqrt{m} \Gamma_1(c,p), \ \hat{\gamma}_2\leq \sqrt{m} \Gamma_1(c,p), \ \hat{\gamma}_3\leq \Gamma_2(c,p), \ \text{and} \ \hat{\gamma}_4\leq \sqrt{m} \Gamma_3(c,p).
\end{equation}
Next we bound $\hat{\gamma}_5$ and $\hat{\gamma}_6$. Combining Jensen's inequality with \eqref{eqn:AssumptionMomentBoundsD} and \eqref{eqn:AssumptionMomentBoundsD2} yields that
\begin{equation}\label{eqn:BoundNormD}
\E \|D_{(x,M_x)}F\|^{4+p} \leq m^{\frac{4+p}{2}} c, \quad \lambda\text{-a.e.}\ x\in\mathbb{X},
\end{equation}
and
\begin{equation}\label{eqn:BoundNormD2}
\E \|D^2_{(x_1,M_{x_1}),(x_2,M_{x_2})}F\|^{4+p} \leq m^{\frac{4+p}{2}} c, \quad \lambda^2\text{-a.e.}\ (x_1,x_2)\in\mathbb{X}^2.
\end{equation}
Consequently, we have that, for $\ell\in\{1,2\}$ and $\lambda^2\text{-a.e.}\ (x_1,x_2)\in\mathbb{X}^2$,
\begin{align*}
& \E \big( \|D_{(x_1,M_{x_1})}F\|^\ell + \| D^2_{(x_1,M_{x_1}),(x_2,M_{x_2})}F \|^\ell \big)^3 \\
& \leq 4 \E \|D_{(x_1,M_{x_1})}F\|^{3\ell} + 4 \E \| D^2_{(x_1,M_{x_1}),(x_2,M_{x_2})}F \|^{3\ell} \\
& \leq 4 \big( \E \|D_{(x_1,M_{x_1})}F\|^{4+p} \big)^{\frac{3\ell}{4+p}} + 4 \big( \E \|D^2_{(x_1,M_{x_1}),(x_2,M_{x_2})}F\|^{4+p} \big)^{\frac{3\ell}{4+p}} \leq 8 m^{\frac{3\ell}{2}} c^{\frac{3\ell}{4+p}}.
\end{align*}
For $p>2$, $\ell\in\{1,2\}$, and $\lambda^3\text{-a.e.}\ (x_1,x_2,y)\in\mathbb{X}^3$, H\"older's inequality with $q_1=\frac{4+p}{p-2}$, $q_2=q_3=\frac{4}{3\ell}(4+p)$, and $q_4=q_5=\frac{2}{3}(4+p)$ (and $q_6=\frac{2}{3}(4+p)$ for the factor one if $\ell=1$) as well as \eqref{eqn:BoundNormD} and \eqref{eqn:BoundNormD2} lead to
\begin{align*}
& \E \mathbf{1}\{D^2_{(x_1,M_{x_1}),(y,M_y)}F\neq \0, D^2_{(x_2,M_{x_2}),(y,M_y)}F\neq \0\} \big(\|D_{(x_1,M_{x_1})}F\|^\ell+\|D^2_{(x_1,M_{x_1}),(y,M_y)}F\|^\ell\big)^{3/4}\\
& \quad \times \big(\|D_{(x_2,M_{x_2})}F\|^\ell+\|D^2_{(x_2,M_{x_2}),(y,M_y)}F\|^\ell\big)^{3/4}  |D_{(x_1,M_{x_1})}F_i|^{3/2} \, |D_{(x_2,M_{x_2})}F_i|^{3/2} \\
& \leq \E \mathbf{1}\{D^2_{(x_1,M_{x_1}),(y,M_y)}F\neq \0, D^2_{(x_2,M_{x_2}),(y,M_y)}F\neq \0\}\\
& \quad \quad \times \big(\|D_{(x_1,M_{x_1})}F\|^{\frac{3\ell}{4}}+\|D^2_{(x_1,M_{x_1}),(y,M_y)}F\|^{\frac{3\ell}{4}}\big)\\
& \quad \quad \times \big(\|D_{(x_2,M_{x_2})}F\|^{\frac{3\ell}{4}}+\|D^2_{(x_2,M_{x_2}),(y,M_y)}F\|^{\frac{3\ell}{4}}\big) |D_{(x_1,M_{x_1})}F_i|^{3/2} \, |D_{(x_2,M_{x_2})}F_i|^{3/2} \allowdisplaybreaks\\
& \leq \mathbb{P}(D^2_{(x_1,M_{x_1}),(y,M_y)}F\neq \0, D^2_{(x_2,M_{x_2}),(y,M_y)}F\neq \0)^{\frac{p-2}{4+p}} \\
& \quad \quad \times \big( \big(\E \|D_{(x_1,M_{x_1})}F\|^{4+p}\big)^{\frac{3\ell}{4(4+p)}}+\big(\E \|D^2_{(x_1,M_{x_1}),(y,M_y)}F\|^{4+p}\big)^{\frac{3\ell}{4(4+p)}}  \big) \\
& \quad \quad \times \big( \big(\E \|D_{(x_2,M_{x_2})}F\|^{4+p}\big)^{\frac{3\ell}{4(4+p)}}+\big(\E \|D^2_{(x_2,M_{x_2}),(y,M_y)}F\|^{4+p}\big)^{\frac{3\ell}{4(4+p)}}  \big) \\
& \quad \quad \times \big( \E |D_{(x_1,M_{x_1})}F_i|^{4+p}\big)^{\frac{3}{2(4+p)}}  \big( \E |D_{(x_2,M_{x_2})}F_i|^{4+p}\big)^{\frac{3}{2(4+p)}} \\
& \leq 4 m^{\frac{3\ell}{4}} c^{\frac{3(\ell+2)}{2(4+p)}} \mathbb{P}(D^2_{(x_1,M_{x_1}),(y,M_y)}F\neq \0, D^2_{(x_2,M_{x_2}),(y,M_y)}F\neq \0)^{\frac{p-2}{4+p}}.
\end{align*}
From H\"older's inequality and the previous estimates, we obtain that, for $p>2$,
\begin{align*}
\hat{\gamma}_5^3 & \leq 3\sum_{i,j=1}^m \int_{\mathbb{X}^3} 4^{2/3} \sqrt{m} c^{\frac{3}{4+p}} \mathbb{P}(D^2_{(x_1,M_{x_1}),(y,M_y)}F\neq \0, D^2_{(x_2,M_{x_2}),(y,M_y)}F\neq \0)^{\frac{2}{3}\frac{p-2}{4+p}} c^{\frac{2}{4+p}} \\
& \hskip 3cm \lambda^3(\dint(x_1,x_2,y))\\
& \quad + \sum_{i,j=1}^m \int_{\mathbb{X}^3} 2\sqrt{m} c^{\frac{1}{4+p}}\\
& \hskip 2cm \times \bigg( \frac{45}{2}  c^{\frac{2}{4+p}} \mathbb{P}(D^2_{(x_1,M_{x_1}),(y,M_y)}F_i\neq 0, D^2_{(x_2,M_{x_2}),(y,M_y)}F_i\neq 0)^{\frac{1}{3}\frac{p-2}{4+p}} c^{\frac{2}{4+p}}\\
& \hskip 2.5cm + \frac{9}{2} c^{\frac{2}{4+p}} \mathbb{P}(D^2_{(x_1,M_{x_1}),(y,M_y)}F_i\neq 0, D^2_{(x_2,M_{x_2}),(y,M_y)}F_i\neq 0)^{\frac{1}{3}\frac{p-2}{4+p}} c^{\frac{2}{4+p}} \bigg) \\
& \hskip 3cm \lambda^3(\dint(x_1,x_2,y)) \\
& \leq m^{5/2} \Gamma_4(c,p)^3
\end{align*}
and
\begin{align*}
\hat{\gamma}_6^4 & \leq 3\sum_{i,j=1}^m \int_{\mathbb{X}^3} 4^{2/3} m c^{\frac{4}{4+p}} \mathbb{P}(D^2_{(x_1,M_{x_1}),(y,M_y)}F\neq \0, D^2_{(x_2,M_{x_2}),(y,M_y)}F\neq \0)^{\frac{2}{3}\frac{p-2}{4+p}} c^{\frac{2}{4+p}} \\
& \hskip 3cm \lambda^3(\dint(x_1,x_2,y))\\
& \quad + \sum_{i,j=1}^m \int_{\mathbb{X}^3} 2m c^{\frac{2}{4+p}}\\
& \hskip 2cm \times \bigg( \frac{135}{8}  c^{\frac{2}{4+p}} \mathbb{P}(D^2_{(x_1,M_{x_1}),(y,M_y)}F_i\neq 0, D^2_{(x_2,M_{x_2}),(y,M_y)}F_i\neq 0)^{\frac{1}{3}\frac{p-2}{4+p}} c^{\frac{2}{4+p}}\\
& \hskip 2.5cm + \frac{27}{8} c^{\frac{2}{4+p}} \mathbb{P}(D^2_{(x_1,M_{x_1}),(y,M_y)}F_i\neq 0, D^2_{(x_2,M_{x_2}),(y,M_y)}F_i\neq 0)^{\frac{1}{3}\frac{p-2}{4+p}} c^{\frac{2}{4+p}} \bigg) \\
& \hskip 2.5cm \lambda^3(\dint(x_1,x_2,y)) \\
& \leq m^3 \Gamma_5(c,p)^4.
\end{align*}
This implies that
\begin{equation}\label{eqn:Gammas56}
\hat{\gamma}_5 \leq m^{5/6} \Gamma_4(c,p) \quad \text{and} \quad \hat{\gamma}_6 \leq m^{3/4} \Gamma_5(c,p).
\end{equation}
Combining the estimates in \eqref{eqn:Gammas1-4} and in \eqref{eqn:Gammas56} with the marked versions of Theorem \ref{thm:GeneralMultivariateBound} and Theorem \ref{thm:Generaldconvex} described at the beginning of this proof completes the proof of Theorem \ref{thm:generalStabilization}.
\end{proof}

\begin{proof}[Proof of Corollary \ref{coro:BoundedDifferenceOperator}]
We aim to apply Theorem \ref{thm:generalStabilization} without marks. We choose $s_0$ such that $\Sigma_s$ is positive definite for $s\geq s_0$ and such that $\|\Sigma_s\|_{op}$ and $\|\Sigma_s^{-1}\|_{op}$ are uniformly bounded for $s\geq s_0$. For $\lambda=s\mu$, the assumptions \eqref{eqn:AssumptionMomentBoundsD} and \eqref{eqn:AssumptionMomentBoundsD2} of Theorem \ref{thm:generalStabilization} are satisfied with $c=a/s^{3+\varepsilon/2}$ and $p=2+\varepsilon$. The assumptions \eqref{As11}, \eqref{As12}, and \eqref{As2} show that $\Gamma_j(c,p)$, $j\in\{1,\hdots,5\}$, are all of order $s^{-1/2}$. Together with $\Sigma=\Sigma_s$, this yields the conclusion of Corollary \ref{coro:BoundedDifferenceOperator}.
\end{proof}

\section*{Acknowledgements}

The authors are very thankful to Xiao Fang for bringing to their attention his PhD thesis \cite{FangPhD}, which inspired the part of the proof of Theorem \ref{thm:Generaldconvex} that deals with $J_{2,2}$. This led to a significant improvement of the previous version \cite{ArXivVersion}, where a weaker distance $d_{\mathbb{H}_\ell}$ based on the intersection of half-spaces was considered and where bounds for the $d_{convex}$-distance could  only be given under more restrictive assumptions. The second author gratefully acknowledges the generous support of the University of Bern, where some of this research was completed. Finally, we thank the referee for an attentive reading and for pointing out a typo in the statement of Theorem \ref{thm:Generaldconvex}.



\appendix

\section{Appendix: Malliavin calculus on the Poisson space}\label{sec:MalliavinCalculus}

We recall the definitions of the Malliavin operators as well as some of their relations. For more details we refer to, for example, \cite[Section 2]{LPS}.

We start with a pathwise product formula for the difference operator.

\begin{lemm}\label{lem:DProduct}
For Poisson functionals $F$ and $G$ and $x\in\mathbb{X}$,
$$
D_x(FG)=(D_xF) G + F (D_xG) + (D_xF) (D_xG).
$$
\end{lemm}

The second moment and the variance of a Poisson functional can be bounded in terms of the difference operator:

\begin{theo}[Poincar\'e inequality]\label{thm:Poincare}
For a Poisson functional $F$ with $\E |F|<\infty$,
$$
\E F^2 \leq \big( \E F \big)^2 + \E \int_{\mathbb{X}} (D_xF)^2 \, \lambda(\dint x).
$$
\end{theo}

For $n\in\N$ let us denote by $I_n(g)$ the multiple Wiener-It\^o integral of $g\in L^2(\lambda^n)$ with respect to the Poisson process $\eta$. Note that for $g\in L^2(\lambda^n)$, $n\in\N$, and $h\in L^2(\lambda^m)$, $m\in\N$,
\begin{equation}\label{eqn:OrthogonalityItoIntegrals}
\E I_n(g) I_m(h) = {\bf 1}\{n=m\} n! \int_{\mathbb{X}^n} g(x) h(x) \, \lambda^n(\dint x).
\end{equation}
Any square integrable Poisson functional $F$ has a so-called Wiener-It\^o chaos expansion
$$
F=\E F + \sum_{n=1}^\infty I_n(f_n),
$$
where the functions $f_n\in L^2(\lambda^n)$, $n\in\N$, are symmetric and $\lambda^n$-a.e.\ uniquely defined and the right-hand side converges in $L^2(\mathbb{P})$. Together with \eqref{eqn:OrthogonalityItoIntegrals} one sees that
$$
\Var F = \sum_{n=1}^\infty n! \|f_n\|^2_{n},
$$
where $\|\cdot\|_n$ denotes the usual norm in $L^2(\lambda^n)$ for $n\in\N$.

If $F\in\operatorname{dom} D$ (see \eqref{eqn:DomD}), the difference operator defined in \eqref{eqn:DifferenceOperator} satisfies the identity
$$
D_xF = \sum_{n=1}^\infty n I_{n-1}(f_n(x,\cdot)) \quad \mathbb{P}\text{-a.s.}
$$
for $\lambda$-a.e.\ $x\in\mathbb{X}$. Here, $f_n(x,\cdot)$ denotes the function in $n-1$ variables one obtains after fixing the first argument to be $x$. Moreover, $F\in\operatorname{dom} D$ is equivalent to
$$
\sum_{n=1}^\infty n \, n! \|f_n\|_n^2<\infty.
$$
The inverse Ornstein-Uhlenbeck generator of $F$ is given by
$$
L^{-1} F = -\sum_{n=1}^\infty \frac{1}{n} I_n(f)
$$
and is the pseudo-inverse of the Ornstein-Uhlenbeck generator $L$, which we do not need for our purposes. Next we present the definition of the Skorohod integral $\delta$. We say that a random function $g: \mathbb{X}\to\R$ depending only on $\eta$ such that
\begin{equation}\label{eqn:LPlambda}
\E \int_{\mathbb{X}} g(x)^2 \, \lambda(\dint x) <\infty
\end{equation}
belongs to $\operatorname{dom} \delta$ if
$$
g(x) = g_0(x) + \sum_{n=1}^\infty I_{n}(g_n(x,\cdot))
$$
for $\lambda$-a.e.\ $x\in\mathbb{X}$ with functions $g_n\in L^2(\lambda^{n+1})$, $n\in\N\cup\{0\}$, such that
$$
\sum_{n=0}^\infty (n+1)! \|\tilde{g}_{n}\|^2_{n+1}<\infty.
$$
Here, $\tilde{g}_n\in L^2(\lambda^{n+1})$ denotes the symmetrization
$$
\tilde{g}_{n}(x_1,\hdots,x_{n+1}) = \frac{1}{(n+1)!} \sum_{\pi\in\Pi(n+1)} g_n(x_{\pi(1)},\hdots,x_{\pi(n+1)})
$$
of $g_n$, where $\Pi(n+1)$ stands for the set of all permutations of $\{1,\hdots,n+1\}$. For $g\in\operatorname{dom} \delta$ the Skorohod integral $\delta(g)$ is defined as
$$
\delta(g)=\sum_{n=0}^\infty I_{n+1}(\tilde{g}_{n}),
$$
i.e., $\delta$ maps a random function to a random variable. The difference operator and the Skorohod integral are adjoint operators in the sense that they satisfy the following well-known integration by parts formula.

\begin{lemm}\label{lem:PartialIntegration}
For $F\in \operatorname{dom} D$ and $g\in \operatorname{dom}\delta$,
$$
\E \int_{\mathbb{X}} D_xF g(x) \, \lambda(\dint x) = \E F \delta(g).
$$
\end{lemm}

The following lemma (see \cite[Proposition 2.3 and Corollary 2.4]{LPS}) provides a criterion for $g$ belonging to $\operatorname{dom}\delta$ and an upper bound for the second moment of $\delta(g)$.

\begin{lemm}\label{lem:BoundDelta}
Let $g$ be a random function depending only on $\eta$ and satisfying \eqref{eqn:LPlambda} and
\begin{equation}\label{eqn:AssumptionDg}
\E\int_{\mathbb{X}^2} (D_yg(x))^2\, \lambda^2(\dint(x,y)) <\infty.
\end{equation}
Then, $g\in \operatorname{dom} D$ and
$$
\E \delta(g)^2 \leq \E \int_{\mathbb{X}} g(x)^2 \, \lambda(\dint x) + \E\int_{\mathbb{X}^2} (D_yg(x))^2\, \lambda^2(\dint(x,y)).
$$
\end{lemm}

\end{document}